%% file: arxiv_version_181126.tex
\documentclass[a4paper, 11pt, reqno]{amsart}

\usepackage{Packages}
\usepackage{Commands}
\usetikzlibrary{intersections,arrows.meta,decorations.pathreplacing}

\title[Low Regularity Potentials]{Low regularity potentials in heterogeneous Cahn--Hilliard functionals}

\author{Riccardo Cristoferi, Jakob Deutsch, Luca Pignatelli}

\begin{document}

\begin{abstract}

    In this paper, we study the prototypical model of liquid–liquid phase separation, the Cahn–Hilliard functional, in a highly irregular setting.
    Specifically, we analyze potentials with low regularity vanishing on space-dependent wells.
    Under remarkably weak hypotheses, we establish a robust compactness result.
    Strengthening the regularity of the wells and of the growth of the potential close to the wells only slightly, we completely characterize the asymptotic behavior of the associated family of functionals through a $\Gamma$-convergence analysis.
    As a notable technical result, we prove the existence of geodesics for a degenerate metric and establish a uniform bound on their Euclidean length.
\end{abstract}

\maketitle

\tableofcontents

\section{Introduction}

The Cahn--Hilliard functional is the prototypical model for liquid-liquid phase separation.
It was initially proposed by van der Waals in \cite{vanDW}, and later independently rediscovered by Cahn and Hilliard in \cite{CH}.

We start by describing the physical situation that we are dealing with.
We consider a physical system inside a container $\o \subset \R^N$.
We assume that the system is described by a phase (or order) parameter $u \colon \o \to \R^M$.
For instance, in the original derivation of the model, the system under investigation was a mixture of two fluids. In such a case, the phase parameter represents the density.
We are interested in describing stable equilibrium configurations of the system.
Using a variational point of view, these correspond to local minimizers of the Gibbs free energy.
Under \emph{homogeneous conditions}, and after a rescaling, this is given by a functional of the form
\[
F_\varepsilon(u) \coloneqq \int_\Omega \left[ \frac{1}{\varepsilon}W(u(x)) + \varepsilon|\nabla u(x)|^2 \right] \dd x.
\]
Here, $W\colon \R^M\to[0,\infty)$ is the free energy density.
Assuming for simplicity that the physical system has two stable phases, modeled by $a,b \in\R^M$, we impose that the potential $W$ vanishes only at $a$ and $b$.
The second term of the functional penalizes sharp oscillations of the phase variable.
Finally, the parameter $\varepsilon>0$ is related to the scale of the transition region between bulk regions where the phase variable is close to the stable phases $a$ and $b$, and it is usually very small.

The mathematical study of the functional $F_\varepsilon$ was, and still is, a source of many deep and interesting problems.
In particular, here we focus on the study of its asymptotic behavior as $\varepsilon\to0$.
In the community of calculus of variations, such studies make use of $\Gamma$-convergence, the notion of variational convergence introduced in \cite{DGFra}, that allows to capture the asymptotic behavior of minimizers of the family of functionals $\{F_\varepsilon\}_{\varepsilon>0}$, as well as the asymptotic of the minimal energy.
A similar analysis can be performed also in the case where a \emph{mass constraint} is imposed. Namely, given $m\in\R^M$, we consider the problem
\begin{equation}\label{eq:min_pb_intro}
\min\left\{ F_\varepsilon(u) : u\in W^{1,2}(\Omega;\R^M), \, \int_\Omega u(x)\dd x = m  \right\}.
\end{equation}
Here, we recall some contributions, without claiming to be exhaustive.
We first focus on the case where the phase variable is scalar-valued, namely when $M=1$.
Gurtin (see \cite{Gur}) conjectured that in the limit as $\varepsilon\to0$, any family of minimizers of $\{F_\varepsilon\}_{\varepsilon>0}$ converges to a piecewise constant function that partitions $\Omega$ into two regions separated by an interface with minimal surface area.
This conjecture was proved rigorously 
by Carr, Gurtin, and Slemrod for $N = 1$ (see \cite{CarGurSle}), and by Modica (see \cite{modica87}) and by Sternberg (see \cite{Ste_Sing}) for $N \ge 2$ (see also \cite{ModMor_Limite, ModMor_Esempio}).

The case of vector-valued wells, namely $M\geq 1$, was considered in \cite{FonTar} by Fonseca and Tartar.
What the authors proved is a compactness result with respect to the strong $L^1(\Omega;\R^M)$ topology for sequences of uniformly bounded energy, and that the $\Gamma$-limit of the sequence $\{F_\varepsilon\}_\varepsilon$ is given by
\[
F_0(u)\coloneqq \sigma\,\mathrm{Per}(\{u=a\};\Omega),
\]
for functions $u\in BV(\Omega;\R^M)$ taking values only in $\{a,b\}$, and $+\infty$ else in $L^1(\Omega;\R^M)$.
Here, $\mathrm{Per}(\{u=a\};\Omega)$ denotes the perimeter of the set $\{u=a\}$ in $\Omega$ (see Section \ref{sec:BV} for more details on sets of finite perimeter), and the constant $\sigma>0$ is given by
\begin{equation}\label{eq:sigma_intro}
\sigma\coloneqq \inf\left\{\int_{-1}^1 2\sqrt{W(\gamma(t))}|\gamma'(t)|\dd t \,:\, \gamma\in W^{1,1}([-1,1];\R^M), \gamma(-1)=a, \gamma(1)=b \right\}.
\end{equation}
Such a quantity represents the minimal energy needed to transition from one well to the other.
An optimal profile, namely a solution to the above minimization problem, approximates the behavior of a solution to \eqref{eq:min_pb_intro}.
Note that the minimization problem is a degenerate geodesic problem, where the degeneracy comes from the fact that the weight $\sqrt{W}$ vanishes at the endpoints of the curves.
In scalar case $M=1$, the change of variables $s=\gamma(t)$ gives that
\begin{equation}\label{eq:sigma_scalar_intro}
\sigma = \int_a^b 2\sqrt{W(s)}\dd s,
\end{equation}
for all admissible curves.

Since it is of particular interest for the goal of this paper, we remark that in \cite{FonTar} the authors assumed that the potential $W$ behaves quadratically near the wells $a$ and $b$.

There are several extensions and variants of the above mentioned results.
For a more comprehensive overview of such literature, the reader can consult the introduction of the paper \cite{cristoferi2025phaseseparationmultiplyperiodic}.\\

The study of non-homogeneous Cahn--Hilliard functionals initiated with the paper \cite{Bou} by Bouchitt\`{e}, where the author studied the fully coupled case with two space-dependent scalar wells.
Namely, for $u\in W^{1,2}(\Omega)$, Bouchitt\`{e} considered the functional
\[
G_\varepsilon(u) \coloneqq \frac{1}{\varepsilon}\int_\Omega g(x, u(x), \varepsilon\nabla u(x))\dd x,
\]
where $g:\Omega\times\R\times\R^N\to[0,\infty)$ is a continuous function with $g(x,u,0)=0$ if and only if $u\in\{a(x), b(x)\}$ for all $x\in\Omega$, and $p\mapsto g(x,u,p)$ is a convex function achieving its strict minimum at $p=0$, where it is also assumed to be differentiable.
The space-dependent wells $a,b:\Omega\to\R$ are assumed to be Lipschitz continuous.
The author identified the limiting functional in the sense of $\Gamma$-convergence. Since we do not want to enter into the technicalities needed in order to treat the general case, we consider here the special case
\[
g(x,u,p) = W(x,u) + |p|^2.
\]
In such a special case, the limiting functional writes as
\[
G_0(u)\coloneqq \int_{J_u\cap\Omega} \sigma(x; a(x), b(x))\dhno(x),
\]
for functions $u\in BV(\Omega)$ such that $u(x)\in\{a(x),b(x)\}$ for a.e. $x\in\Omega$, and $+\infty$ else in $L^1(\Omega)$.
Here, $J_u$ denotes the jump set of the function $u$ (see Section \ref{sec:BV}) for all $x\in\Omega$, and $p,q\in\R^M$, we set
\begin{equation}\label{eq:sigma_moving_scalar}
\sigma(x;p,q)\coloneqq \int_p^q 2\sqrt{W(x,s)}\dd s.
\end{equation}
Roughly speaking, the limiting energy density is space-dependent, and is given, at any fixed point $x\in \Omega$ by the same degenerate geodesic problem \eqref{eq:sigma_scalar_intro} where the space variable is frozen.

Then, in \cite{CriGra} the first author together with Gravina considered the case of multiple vector-valued wells for the uncoupled functional.
Namely, the authors considered the functional
\[
H_\varepsilon(u) \coloneqq \int_\Omega \left[ \frac{1}{\varepsilon}W(x,u(x)) + \varepsilon|\nabla u(x)|^2 \right] \dd x,
\]
for $u\in W^{1,2}(\Omega;\R^M)$, where $W(x,u)=0$ if and only if $u\in\{z_1(x),\dots,z_k(x)\}$ for all $x\in\Omega$.
Being the first result in that direction for vector-valued wells, strong assumptions were needed.
Indeed, the wells $z_1,\dots,z_k\colon\Omega\to\R^M$ were required to be Lipschitz continuous, and the potential $W\colon\Omega\times\R^M\to[0,\infty)$ to be Lipschitz continuous in the first variable, and of class $C^2$ in the second. Moreover, it was assumed that the potential was \emph{exactly} quadratic near the wells.
{In such a case, the authors proved a compactness result, namely that any energy bounded sequence in $W^{1,2}(\Omega; \R^M)$ converges to $u \in BV(\Omega; \{ z_1, \ldots, z_k\})$ of the form
$$ u = \sum_{i=1}^k z_i \mathbbm{1}_{\Omega_i}, $$
where $\{ \Omega_1, \ldots, \Omega_k \}$ is a Caccioppoli partition of $\Omega$.
Moreover, the limiting functional with respect to the strong $L^1(\Omega;\R^M)$ topology in the $\Gamma$-convergence sense, is given by
\begin{equation}\label{eq:functional_H0}
H_0(u)\coloneqq \sum_{i=1}^{k-1} \sum_{j=i+1}^k\int_{\partial^*\Omega_i \cap \partial^* \Omega_j} \sigma(x; z_i(x), z_j(x))\dhno(x),
\end{equation}
where
\begin{equation}\label{eq:sigma_intro_moving}
\sigma(x;p,q)\!\coloneqq \inf\left\{\int_{-1}^1 \!2\sqrt{W(x,\gamma(t))}|\gamma'(t)|\dd t : \gamma\in\! W^{1,1}([-1,1];\R^M), \gamma(-1)=p, \gamma(1)=q \right\}\!.
\end{equation}}
Note that this is the same problem as in \eqref{eq:sigma_intro} where the space variable is frozen, similarly to what happens with \eqref{eq:sigma_moving_scalar} and \eqref{eq:sigma_scalar_intro}. 
The latter assumption has been weakened in the paper \cite{CriFonGan_moving_sub} by the first author, Ganedi and Fonseca, by requiring the potential to be \emph{controlled} from above and from below by two quadratic functions near the wells.
Moreover, the authors were able to also consider the case where the  wells and the potential are discontinuous on the boundary of a polyhedral partition of the domain $\Omega$.\\

We would like to remark that, in the case of scalar-valued wells, like in \cite{Bou}, there is no need to assume any particular behavior of the potential near the wells. The technical reason is that the optimal profile at a point $x$ will be a suitable reparametrization of the interval $[a(x), b(x)]$. In particular, using the continuity of the wells, the length of all of these intervals is uniformly bounded from above.
On the other hand, in the vector-valued case, the optimal profile is a \emph{curve} solving the minimization problem in \eqref{eq:sigma_intro}.
In such a case, the behavior of the potential $W$ around the wells is essential to ensure that such a curve has an \emph{Euclidean length} that is uniformly bounded from above. This requirement is needed in order to prove the $\Gamma$-convergence result.
This is the main difference from the technical point of view between the scalar and the vector-valued case.\\

The goal of the paper is to prove that it is possible to significantly weaken the requirements on the regularity of the potential and on its behavior close to the wells and still be able to obtain compactness and the $\Gamma$-limit result.
{This is motivated by the several cases of interest for applications, where liquid-liquid phase separation happens in rough heterogeneous environments where the stable phases are space dependent, such as the organization of the cell cytoplasm (see \cite{Lee2013PositionDependent}), rock porous media (see \cite{vanDuijn1995CapillaryForces}), and elastic networks (see \cite{Style2018ElasticNetwork}).}

In particular, we consider the functional
\[
\mathcal{F}_\varepsilon(u) \coloneqq \int_\Omega \left[ \frac{1}{\varepsilon}W(x,u(x)) + \varepsilon|\nabla u(x)|^2 \right] \dd x,
\]
where, for each $x\in\Omega$, the function $u\mapsto W(x,u)$ vanishes at $\{a(x), b(x)\}$ for functions $a,b:\Omega\to\R^M$.
We prove that compactness holds under extremely weak assumptions (see \ref{H1}-\ref{H5}): we consider a Carath\'{e}odory potential with mild bounds close to the wells and at least linear growth at infinity, and bounded wells of bounded variation (see Theorem \ref{thm:compactness}).
We are able to obtain such a result by exploiting the theory of Young measures, and by carefully using the chain rule for functions of bounded variation.

The $\Gamma$-convergence result (see Theorem \ref{thm:Gamma_conv}) requires strengthening these assumptions.
Indeed, we consider wells that are allowed to jump on a piecewise $C^2$ partition of the domain $\Omega$, and that are of class $W^{1,2}\cap C^0$ inside each set of such a partition (see \ref{H6}).
The potential is assumed to also satisfy a relative continuity condition (see \ref{H7}) and a more quantitative control close to the wells (see \ref{H8}).
This set of assumptions is extremely weak compared to the one used in the literature.

Under these assumptions, we are able to prove that the limiting energy is given by \eqref{eq:functional_H0}, where the limiting surface energy density is given by \eqref{eq:sigma_intro_moving}.
The proof of the $\Gamma$-convergence result requires to face non-trivial technical challenges. On the one hand, we need to prove a uniform bound on the Euclidean length of geodesics of a class of degenerate metrics under very weak assumptions on the metric (see Theorem \ref{existence-of-geodesics-with-bounded-path-length}).
Moreover, we are able to prove that such a bound on the Euclidean length of geodesics fails if our assumptions are not satisfied (see Example \ref{ex:geodesics_infinite_Euclidean_length}).
Finally, in the liminf inequality, we need to \emph{freeze} the first variable of the potential.
This is done by introducing a suitable adjustment of the potential (see \eqref{eq:adjustment}), which allows us to work \emph{as if} the wells were not space-dependent, but only the potential is.
Finally, the construction of the optimal profile requires multiple careful estimates of energies related to changing the optimal profile from point to point.

It is the goal of several forthcoming papers to investigate the case where the wells enjoy even lower regularity assumptions.


\section{Assumptions and main results}

In this section we collect all the several assumptions we used in establishing the results of the manuscript.
We will comment on each of them, highlighting the reason why we require it, the comparison with similar assumptions in previous works in the literature, and the possibility to weaken it.\\

Let $\Omega\subset\R^N$, where $N\geq 1$, be an open bounded set with Lipschitz boundary.
Consider a function $W \colon \Omega\times\R^M\to[0,\infty)$, with $M\geq 1$, satisfying the following assumptions:
 \begin{enumerate}[label=(H\arabic*), ref=(H\arabic*)]
            \item \label{H1} $W$ is a Carath\'{e}odory function;
            \item \label{H2} There exist functions $a,b\in BV(\Omega;\R^M)\cap L^\infty(\Omega;\R^M)$ such that
            \[
                W(x,u) = 0 \,\,\,\text{ if and only if }\,\,\, u\in\{a(x), b(x)\},
            \]
            for a.e. $x\in\Omega$;
            \item \label{H3} There exists $\delta>0$ such that
            \[
                |a(x) - b(x)| \geq \delta,
            \]
            for a.e. $x\in\Omega$;
            \item \label{H4} There exist a continuous non-decreasing function $f\colon[0,\infty)\to[0,\infty)$ with $f(t)=0$ if and only if $t=0$, and $C_1>0$ such that
            \[
            \frac{1}{C_1} f\left(\min\{ |u-a(x)|, |u-b(x)| \}\right)
            \leq W(x,u)
                \leq C_1 f\left(\min\{ |u-a(x)|, |u-b(x)| \}\right),
            \]
            for a.e. $x\in\Omega$, and all $u\in \R^M$;
            \item \label{H5} There exists $C_2>0$ such that
            \[
            W(x,u) \geq \frac{1}{C_2}|u|,
            \]
            for a.e. $x \in \Omega$ and all $u\in\R^M$ with $|u|\geq C_2$.\\
    \end{enumerate}

We refer to the above assumptions as the \emph{structural} assumptions on the potential $W$.

\begin{remark}
The continuity of the potential $W$ in the second variable is needed in order to make the composition $x\mapsto W(x,u(x))$ measurable for functions $u\in W^{1,2}(\Omega;\R^M)$.
It can be relaxed to Borel measurable in the second variable.
Nevertheless, note that for the Gamma-convergence result (see Theorem \ref{thm:Gamma_conv}) we will need the potential to be piecewise continuous.
\end{remark}

\begin{remark}
We decided to work with only two wells (see \ref{H2}) just for simplicity of notation and in order to focus on the main novel ideas of the manuscript.
The extension to multiple wells requires a careful definition of the adjustment (see \eqref{eq:adjustment}) and a strong approximation result.
These extensions will be treated in a forthcoming paper.
\end{remark}
 
\begin{remark}
The assumption that the wells are in $L^\infty(\Omega;\R^M)$ is required in order to have \ref{H5}.
Indeed, consider the case where $a,b\in BV(\Omega;\R^M)\setminus L^\infty(\Omega;\R^M)$, and the potential is given by 
\[
W(x,u) \coloneqq \min\{ |u-a(x)|^2, |u-b(x)|^2 \}.
\]
Then, it is not possible to find $C_2>0$ such that \ref{H5} holds uniformly for a.e. $x \in \Omega$.
\end{remark}

\begin{remark}
Assumption \ref{H3} on the separation of wells is used for technical convenience. Indeed, by using techniques similar to those of \cite{Bou} and \cite{CriGra}, it is possible to extend our results, \emph{mutatis mutandis}, to the case where \ref{H3} is dropped.
Note that, in our case, we also consider the case where the wells are allowed to jump. Therefore, a bit of care will be needed in order to characterize the space where the limiting energy is finite.
\end{remark}

\begin{remark}
Condition \ref{H4} is a mild assumption that ensures non-degeneracy of the potential $W$. 
The lower bound is essentially needed in order to ensure compactness (see Theorem \ref{thm:compactness}).
The bounds are needed in order to ensure that geodesics of the optimal profile problem have uniform Euclidean length (see Lemma \ref{lemma:almost_minimizers_have_locality_prop}).
This result is essential in the construction of the recovery sequence, but it is not used in the liminf inequality.
\end{remark}

\begin{remark}
The linear growth at infinity, namely \ref{H5}, is a standard condition in the literature, and it is required to get compactness as well as a uniform $L^\infty$ bound on the optimal profiles. It goes back to the work \cite{FonTar} by Fonseca and Tartar.
In case a mass constraint is imposed, Leoni showed in \cite{Leo} that this growth at infinity can be relaxed to a non-degeneracy of the potential at infinity.
\end{remark}

The above set of assumptions will be the ones that we will need in order to establish the compactness result (see Theorem \ref{thm:compactness}).
We now introduce what we will refer to as the \emph{regularity} assumptions on the wells and on the potential. These will be needed in establishing the $\Gamma$-convergence result (see Theorem \ref{thm:Gamma_conv}).
First, we restrict our attention to a specific subclass of wells and potentials, where their jumps as well as their regularity are controlled.

Here, we use the terminology that the boundary of a set is \emph{piecewise $C^2$} to mean that it is contained in a finite union of $C^2$ hypersurfaces.
{For this, we first introduce a notion of admissible partition of $\Omega$ that we will consider in this paper.

\begin{definition}\label{def:piece_wise_C2_admissible}
    Let $\Omega_1,\ldots,\Omega_k$ be an open partition of Lipschitz domains of $\Omega$. We call $\{\Omega_i\}$ an \emph{admissible piecewise $C^2$ partition} if
    \begin{enumerate}
        \item $\partial \Omega_i \cap \Omega$ is contained in a finite union of $C^2$ hypersurfaces;
        \item The boundaries of $\Omega_i$ in $\Omega$ are transversal to the boundary of $\Omega$, namely 
        \[
            \hno(\partial \Omega \cap \overline{A}) = 0,
        \]
        where
        \[
        A \coloneqq \Omega\cap \bigcup_{i=1}^k \partial\Omega_i;
        \]
        \item We have that $\hno(\overline{S_A}) = 0$, where the singular set is defined as
        \[
         S_A \coloneqq \left\{ x \in A  \, | \, \nexists \, O \in \mathcal N(x), \, O \subset \Omega \text{ open} :  O \cap A \text{ is a $C^2$ surface } \right\}.
        \]
    \end{enumerate}
\end{definition}

\begin{remark}
We restrict our attention to such a class of partitions because we want to avoid dealing with pathological cases where the singular set of the partition has positive $\hno$-dimension.
A similar condition has been used in \cite{FrancGiac}.
We think that a careful refinement of our arguments can be used to include also such cases.
For the sake of focusing on the main novel ideas, we refrain from doing that in this investigation.
\end{remark}

In the following, it will be useful to introduce the following regularity for sets.

\begin{definition}\label{def:piecewei_C2_boundary}
A set $E\subset\Omega$ is said to have a \emph{piecewise $C^2$ boundary} if $E, \Omega\setminus E$ is an admissible $C^2$ partition.
\end{definition}

Next, we give the regularity assumptions on the wells that we need for the upper bound.
}

\begin{enumerate}[resume, label=(H\arabic*), ref=(H\arabic*)]
            \item \label{H6} There exists an {admissible piecewise $C^2$} open partition $\Omega_1,\ldots,\Omega_k$ of $\Omega$ {in the sense of Definition \ref{def:piece_wise_C2_admissible}}, such that
            \[
            W(x,u) = \sum_{i=1}^k W_i(x,u)\ca_{\Omega_i}(x),
            \]
            \[
            a(x) = \sum_{i=1}^k a_i(x)\ca_{\Omega_i}(x),\quad\quad\quad\quad
            b(x) = \sum_{i=1}^k b_i(x)\ca_{\Omega_i}(x),
            \]
            where $W_i\in C^0(\overline{\Omega}\times\R^M)$ and $a_i, b_i\in W^{1,2}(\Omega;\R^M) \cap C^0(\overline{\Omega};\R^M)$ are given such that \ref{H1}--\ref{H5} hold.
            Moreover, we assume that
            \begin{equation}\label{eq:separation_at_interface}
            a_i(x)\neq a_j(x), \quad\quad\quad a_i(x)\neq b_j(x),
            \end{equation}
            for all $x\in \partial\Omega_i\cap\partial\Omega_j$, for all $j\neq i$.
\end{enumerate}

\begin{remark}
Condition \eqref{eq:separation_at_interface} is only assumed for technical reasons.
It can be removed without having to change the main ideas of the paper. This is why we decided to assume it in order to focus on the main novel ideas of the manuscript.
\end{remark}

\begin{definition}\label{def:local_adjustment}
    Let $U \subset \Omega$ be open. With the notation of \ref{H6}, we say that 
    $T_i \colon U \times\R^M\to\R^M$ is a \emph{local adjustment} on $U$ if it is of the form
    \begin{equation}\label{eq:adjustment}
        T_i(x,w) \coloneqq  \frac 1 2 \big( {a_i(x)+b_i(x)} + \left(b_i(x)-a_i(x), v_i(x)\right)\cdot w \big),
    \end{equation}
    where $v_i: U \to \R^{M \times (M-1)}$ is such that:
    \begin{itemize}
        \item[(i)] For each $x\in U$, the matrix
        \[
            \frac{1}{|b_i(x)-a_i(x)|}\left(b_i(x)-a_i(x), v_i(x)\right)
        \]
        is orthonormal;
        \item[(ii)] We have 
        \[
            v_i \in W^{1,2}(U;\R^{M \times (M-1)}) \cap C(\overline U;\R^{M \times (M-1)})
        \]
        with 
        \[
            |\nabla v_i| \leq C(|\nabla a_i| + |\nabla b_i|)
        \]
        almost everywhere.
    \end{itemize}
\end{definition}

We are now in the position to state our next assumption.

\begin{enumerate}[resume, label=(H\arabic*), ref=(H\arabic*)]
            \item \label{H7} {For every point $x_0 \in \Omega$ there exists an open neighborhood $U$ of $x_0$ and a non-decreasing modulus of continuity $\omega\colon[0,\infty)\to[0,\infty)$ such that for all $i = 1,\dots, k$ there exists a local adjustment $T_i$ on $U$ in the sense of Definition \ref{def:local_adjustment} fulfilling             
            \[
                | W_i(x, T_i(x,u)) - W_i(y, T_i(y,u)) | \leq \omega(|x-y|)  W_i(x, T_i(x,u)),
            \]
            for all $x,y \in U$ and all $u\in\R^M$. }
\end{enumerate}

We will see in Proposition \ref{prop:local_adjustments_existence} that, under our assumptions, a local adjustment always exists.

\begin{remark}
A condition in the spirit of \ref{H7} has been used in similar investigations on models for phase transition (see, for instance \cite{FonPopo}), i.e., it was required that 
\begin{align}\label{eq:relative_cont_intro}
    | W(x, u) - W(y, u) | \leq \omega(|x-y|)  W(x, u),
\end{align}
for all $x,y\in\Omega$ and $u\in\R^M$.
Its main role is to allow to \emph{freeze} the potential at a given point, with an error proportional to the potential itself.
We note that in our case we need to compose the potential with the map $T$ because we consider the case of \emph{moving} wells.
Otherwise, condition \eqref{eq:relative_cont_intro}
would imply that the functions $a$ and $b$ are constant.
\end{remark}

Despite its technical nature, we want to show that assumption \ref{H7} is satisfied by several classes of potentials of interest. Here, we present some of them.

\begin{example}
The first example that we consider is the prototypical case of the potential for the Cahn-Hilliard functional with moving wells.
Let $W \colon \Omega\times\R\to[0,\infty)$ be defined as
\[
    W(x,u) \coloneqq (u^2-a^2(x))^2,
\]
where $a\in C^0([-1,1])$, with $a(x)\geq \delta$ for all $x\in [-1,1]$. Then, it holds that
\[
T(x,u) = a(x)u.
\]
Therefore, we have that
\[
W(x,T(x,u)) = a^4(x)(u^2-1)^2.
\]
We now check that \ref{H7} holds:
\begin{align*}
| W(x, T(x,u)) - W(y, T(y,u)) | &= |a^4(x) - a^4(y)| (u^2-1) \\
    &\leq \widetilde{\omega}(|x-y|) (u^2-1)^2 \\
    &\leq \frac{\widetilde{\omega}(|x-y|)}{\delta^4} \min\{a^4(x), a^4(y)\} (u^2-1)^2 \\
    &= \frac{\widetilde{\omega}(|x-y|)}{\delta^4} \min\{ W(x, T(x,u)), W(y, T(y,u)) \},
\end{align*}
where $\widetilde{\omega} \colon [0,\infty)\to[0,\infty)$ is the modulus of continuity of the function $x\mapsto a^4(x)$. 
\end{example}

Next, we generalize the above to a similar class of potentials.

\begin{example}
Let $q\in[1,\infty)$.
Define the potential $W\colon\Omega\times\R^M\to[0,\infty)$ as
    \[
        W(x, u) = \min \{ |u - a(x)|^q, |u + a(x)|^q \},
    \]
    where $a \in W^{1,2}(\Omega; \R^M) \cap C^0(\Omega;\R^M)$ with $|a(x)|\geq \delta$, for all $x\in\Omega$.
    {Then, for any open set $U \subset \Omega$, we have that
    \[
        W(x, T(x,u)) = |a(x)|^q\min \{ |u - e_1|^q, |u + e_1|^q \},
    \]
    where $T$ is an arbitrary local adjustment on $U$ with respect to the wells $a, -a$.} 
    Therefore, using similar computations as above, we have that
    \[
        |W(x, T(x,u)) - W(y, T(y,u))|
        \leq \frac{2^q}{\delta^p}\widetilde{\omega}(|x - y|)\min \{W(x, T(x,u)), W(y, T(y,u)) \},
    \]
    where $\widetilde{\omega}\colon[0,\infty)\to[0,\infty)$ is the modulus of continuity of the function $x\mapsto |a|^q(x)$.
\end{example}

Next, we show that assumption \ref{H7} is verified even by potentials with very low regularity.

\begin{example}
Let $V\colon\R^M\to[0,\infty)$ be a Borel function, and let {$a \in W^{1,2}(\Omega; \R^M) \cap C^0(\Omega;\R^M)$}. Set $b\coloneqq -a$.
{Suppose that there exists an adjustment $T$ on $\Omega$ with respect to the wells $\pm a$. Then, we define $W\colon \Omega \times\R^M\to[0,\infty)$ as
\[
    W(x,u)\coloneqq V(T(x,u)^{-1}),
\]
where for each $x \in \Omega$, the map $T(x, \cdot)^{-1}$ is the inverse of the adjustment in the second variable.}
Then, $W$ satisfies assumption \ref{H7}, since
\[
    |W(x, T(x,u)) - W(y, T(y,u))|=0
\]
for all $x,y\in \Omega$ and all $u\in\R^M$.
Note that, in order for $W$ to satisfy assumption \ref{H1}, we need $V$ to be continuous.
\end{example}

{
\begin{remark}\label{remark:for_liminf}
    Let $U \in \Omega$ be open, and let $T_i$ be the local adjustment on $U$ given in Definition \ref{def:local_adjustment}. Note that we find $l_i \in C(U; \R^M)$ and $L_i \in C(U; \R^{M \times M})$, which are pointwise invertible, such that
    \[
        T_i(x, u) = l_i(x) + L_i(x)u
    \]
    for all $x \in U$ and $u \in \R^M$. By the structure of $L_i$, we also have $L_i^{-1} \in C(U, \R^{M \times M})$. In particular, we can write the inverse with respect to the second variable by
    \[
        T_i^{-1}(x, u) =  L_i^{-1}(x)(u - l_i(x)). 
    \]
    As a consequence, we have for all $x,y \in U$
    \[
        T_i^{-1}(x, T(y, u)) = L_i^{-1}(x)( L_i(y)u + l_i(y) - l_i(x)) = \tilde L_i(x,y) u + \tilde l_i(x,y)
    \]
    with suitable functions $\tilde L_i$ and $\tilde l_i$ such that
    \[
        \lim_{y \to x} \tilde l_i(x,y) = 0 \text{ and } \lim_{y \to x} \tilde L_i(x,y) = Id_M. 
    \]
\end{remark}
}

Finally, we introduce a very mild condition on the function controlling the growth around the wells. This will be needed in order to ensure that the geodesics of the limiting distance function have uniformly bounded Euclidean length (see Theorem \ref{existence-of-geodesics-with-bounded-path-length}).

\begin{enumerate}[resume, label=(H\arabic*), ref=(H\arabic*)]
            \item \label{H8} There exists $C_3>0$ such that
            \[
            f(2t) \leq C_3 f(t),
            \]
            for all $t>0$, where $f\colon[0,\infty)\to[0,\infty)$ is the function given by \ref{H4}.
\end{enumerate}

\begin{remark}
We would like to stress that this assumption is very mild.
Indeed, in previous works, stronger conditions were imposed in order to have a uniform bound on the Euclidean length of certain geodesic problems.
In \cite{CriGra} it was required that the potential behaves as a quadratic function close to the wells, while in \cite{CriFonGan_moving_sub}, that it is controlled by quadratic functions from above and from below close to the wells.

In particular, note that our assumption also covers the case where the potential $W$ is \emph{concave near the wells}, that was excluded by the assumptions in the literature.
Indeed, for any $\alpha>0$, the function $f(t)\coloneqq t^\alpha$ satisfies assumption \ref{H8}.
\end{remark}

Finally, we introduce the family of functionals that we consider.

\begin{definition}
For $\varepsilon>0$, we define the functional $\mathcal{F}_\varepsilon\colon L^1(\Omega;\R^M)\to[0,+\infty]$ as
\begin{gather*}
    \mathcal{F}_\varepsilon(u) \coloneqq \begin{cases}
        \displaystyle\int_\Omega \left[ \frac{1}{\varepsilon} W(x,u(x)) + \varepsilon|\nabla u(x)|^2 \right] \dd x \qquad &u \in W^{1,2} (\Omega; \R^M), \\
        + \infty \qquad &u \in L^1 (\Omega;\R^M) \setminus W^{1,2} (\Omega; \R^M).
    \end{cases}
\end{gather*}
\end{definition}

\begin{definition}
    We define the space $BV(\Omega;\{a,b\})$ as the space of functions $u\in BV(\Omega;\R^M)$ such that $u(x)\in\{ a(x), b(x) \}$ for a.e. $x\in\Omega$.
\end{definition}

{\begin{remark}
    It is straightforward to see that if \ref{H3} holds, then each element $u \in BV(\Omega; \{a,b\})$ can be written as 
    \[
        u = a\chi_E + b(1-\chi_E)
    \]
    for a set of finite perimeter $E \subset \Omega$.
\end{remark}}
	
{\begin{definition}\label{def:limiting_functional}
    With the notation of \ref{H6} we denote the \emph{geodesic distance} with respect to $W_i$ for each $i = 1,\dots, k$ by
    \[
        \dd_{W_i}(x;p,q) \coloneqq \inf\left\{ \int_{-1}^1 2\sqrt{W_i(x,\gamma(t))}|\gamma'(t)| \,:\, \gamma\in W^{1,1}((-1,1);\R^M), \gamma(-1)=p, \gamma(1)=q  \right\}.
    \] 
    We denote the \emph{geodesic distance at the boundary} with respect to $W_i$ and $W_j$ for $1 \leq i < j \leq k$ by
    \[
        \dd_{W_{i,j}}(x;p,q) \coloneqq \inf_{r \in \R^M} \big(\dd_{W_i}(x;p,r) + \dd_{W_j}(x;r,q) \big).
    \]
    Now, we define for $\H^{N-1}$-a.e. $x \in \Omega$
    \[
        \dd_W(x; p, q) \coloneqq \begin{cases}
            d_{W_i}(x; p, q) & \text{ if }x \in \Omega_i, \\
            d_{W_{i,j}}(x; p, q) & \text{ if }x \in \partial \Omega_i \cap \partial \Omega_j, i < j.
        \end{cases}
    \]
    Moreover, we set the functional $\mathcal{F}_\infty \colon L^1(\Omega;\R^M)\to[0,+\infty]$ as
    \begin{gather*}
        \mathcal{F}_\infty(u) \coloneqq \begin{cases}
            \displaystyle\int_{J_u \cap \Omega} \dd_W(x; u^-(x), u^+(x)) \dhno(x) \qquad &u \in BV(\Omega; \{ a,b\}), \\
            +\infty \qquad &u \in L^1 (\Omega; \R^M) \setminus BV(\Omega; \{ a,b\}).
        \end{cases}
    \end{gather*}
    Here, $J_u$ denotes the jump set of the function $u$, while $u^-$ and $u^+$ denote the two traces of $u$ (see Section \ref{sec:BV} for more details). For convenience, we will use the convention that for $x \in \partial \Omega_i \cap \partial \Omega_j \cap J_u$ with $i < j$, the trace with respect to the $E_i$ region (so the lower index) is denoted by $u^-(x)$.
\end{definition}}

We are now in position to state the first two main results of the manuscript. We start with the compactness result, that holds under very mild assumptions on the potential and on the wells.

\begin{theorem}\label{thm:compactness}
Assume \ref{H1}-\ref{H5} hold.
Let $\{\e_n\}_n\subset(0,\infty)$ be an infinitesimal sequence, and 
{let $\{u_n\}_n \subset L^1(\Omega;\R^M)$} be such that
\[
\sup_{n\in\N} \mathcal{F}_{\varepsilon_n}(u_n) < +\infty.
\]
Then, there exists a subsequence (not relabeled) such that $u_{n}\to u$ strongly in $L^1(\Omega;\R^M)$, for some $u\in BV(\Omega; \{a,b\})$.
\end{theorem}
    
We are now in position to state the $\Gamma$-convergence result.

\begin{theorem}\label{thm:Gamma_conv}
Assume \ref{H1}-\ref{H8} hold.
Let $\{\e_n\}_n\subset(0,\infty)$ be an infinitesimal sequence.
Then, the sequence of functionals $\{\mathcal{F}_{\varepsilon_n}\}_n$ $\Gamma$-converges to $\mathcal{F}_\infty$ with respect to the $L^1(\Omega;\R^M)$ topology.
\end{theorem}

In particular, the liminf and the limsup inequalities will be given by Theorem \ref{thm:liminf} and Theorem \ref{thm:limsup}, respectively.
We would like to remark here that we are able to prove the liminf inequality under a weaker assumption on the regularity of the partition of $\Omega$.\\


\section{Preliminaries}

In this section, we collect some definitions and tools used throughout this paper.

\subsection{Functions of bounded variation and sets of finite perimeter}\label{sec:BV}

We recall some basic definitions and facts about functions of bounded variation and sets of finite perimeter. We refer to \cite{AFP} for more details.

\begin{definition}[Functions of bounded variation] Let $u \in L^1 (\Omega; \R^M)$. We say that $u$ has bounded variation in $\Omega$, denoted by $u \in BV(\Omega; \R^M)$, if its distributional derivative $Du$ is a matrix-valued Radon measure on $\Omega$. In particular, its variation in $\Omega$, denoted by $|Du|(\Omega)$, is equal to 
\begin{equation*}
    |Du|(\Omega) \coloneqq \sup \left\{ \sum_{i=1}^M \int_\Omega u_i \, \mathrm{div} \varphi_i \, \dd x : \varphi \in C_c^\infty(\Omega; \R^{M \times N}), \| \varphi \|_\infty \leq 1 \right\}.
\end{equation*}
This can also be denoted by $V(u, \Omega)$.
\end{definition}

\begin{definition}[Jump set]
    Let $u \in L^1 (\Omega; \R^M)$. We define the jump set of $u$, denoted by $J_u$, as the set of points where there exist finite one-sided Lebesgue limits. In particular, for each $x \in J_u$ there exists distinct vectors $a,b \in \R^M$ and a normal $\nu \in \mathbb{S}^{N-1}$ such that
    \begin{equation*}
        \lim_{\rho \to 0^+} \frac{1}{\rho^N} \int_{B^+ (x, \rho, \nu)} | u(y) - a | \, \dd y = 0, \qquad \lim_{\rho \to 0^+} \frac{1}{\rho^N} \int_{B^- (x, \rho, \nu)} | u(y) - b | \, \dd y = 0,
    \end{equation*}
    where $B^+ (x, \rho, \nu)$ and $B^- (x, \rho, \nu)$ are defined as
    \begin{equation*}
        B^+ (x, \rho, \nu) = \{ y \in B(x,\rho) : (y - x) \cdot \nu \geq 0 \}, \quad B^- (x, \rho, \nu) = \{ y \in B(x,\rho) : (y - x) \cdot \nu \leq 0 \}.
    \end{equation*}
    If $x \in J_u$, we denote $(a,b,\nu)$ as $(u^+ (x), u^- (x), \nu_u (x))$.
\end{definition}

\begin{theorem}[Decomposition of distributional derivative]
    The jump set $J_u$ of a function $u \in BV(\Omega; \R^M)$ is countably $\mathcal{H}^{N-1}$-rectifiable. Moreover, 
    \begin{equation*}
        Du = \nabla u\, \mathcal{L}^N + ( u^+ - u^- ) \otimes\nu_u \mathcal{H}^{N-1} \llcorner J_u + D^c u,
    \end{equation*}
    where $D^c u$ denotes the Cantor part of the distributional derivative.
\end{theorem}

\begin{theorem}[Chain rule in $BV$]
    Let $u \in BV(\Omega; \R^M)$, and let $f : \R^M \to \R^k$ be a Lipschitz continuous function. Then $v \coloneqq f \circ u$ belongs to $BV(\Omega; \R^k)$ and 
    \begin{equation*}
        D v = \nabla f(u) \nabla u\, \mathcal{L}^N + (f(u^+) - f(u^-)) \otimes \nu_u \mathcal{H}^{N-1} \llcorner J_u + \nabla f (\widetilde{u}) D^c u,
    \end{equation*}
    where $\widetilde{u}$ is the precise representative of $u$, defined everywhere except on $J_u$.
\end{theorem}

\begin{definition}[Variation along a direction]
    Let $u \in L^1 (\Omega; \R^M)$ and $\nu \in \mathbb{S}^{N-1}$. We define the \emph{total variation} of $u$ along the direction $\nu$ as 
    \begin{equation*}
        V_\nu (u, \Omega) = \sup \left\{ \sum_{i=1}^M \int_\Omega u_i \frac{\partial \varphi_i}{\partial \nu} \, \dd x : \varphi \in C_c^\infty (\Omega; \R^{M \times N}), \| \varphi \|_\infty \leq 1 \right\}.
    \end{equation*}
\end{definition}

Given now $\nu \in \mathbb{S}^{N-1}$, we denote by $\pi_\nu$ the hyperplane orthogonal to $\nu$, and by $\Omega^\nu$ the projection of $\Omega$ onto $\pi_\nu$. Now, for any $y \in \Omega^\nu$, we define a slice of $\Omega$ as
\begin{equation*}
    \Omega^\nu_y \coloneqq \left\{ t \in \R : y + t \nu \in \Omega \right\}.
\end{equation*}

Analogously, for $u \colon \Omega \to \R^M$, for any $y \in \Omega^\nu$ we define a slice of $u$ as 
\begin{equation*}
    u_y^\nu \colon \Omega_y^\nu \to \R^M, \qquad u_y^\nu (t) \coloneqq u(y + t \nu).
\end{equation*}

\begin{theorem}[Characterization of $BV$ by sections]
    Let $u \in L^1 (\Omega; \R^M)$ and $\nu \in \mathbb{S}^{N-1}$. Then
    \begin{equation*}
        V_\nu (u, \Omega) = \int_{\Omega^\nu} V( u_y^\nu, \Omega_y^\nu) \, \dd y.
    \end{equation*}
    Moreover, $u \in BV(\Omega; \R^M)$ if and only if there exist $N$ linearly independent unit vectors $\nu_1, .., \nu_N$ such that $u_y^{\nu_j} \in BV(\Omega_y^{\nu_j}; 
    \R^M)$ for $\mathcal{H}^{N-1}$-a.e. $y \in \Omega^{\nu_j}$, and
    \begin{equation*}
        \int_{\Omega^{\nu_j}} V (u_y^{\nu_j}, \Omega_y^{\nu_j}) \, \dd y < \infty \qquad \forall j = 1, \ldots, N.
    \end{equation*}
\end{theorem}

We focus now on the functions of bounded variations which are characteristic functions of sets.

\begin{definition}[Finite perimeter set]
    Let $E \subseteq \Omega$. We say that $E$ has \emph{finite perimeter in $\Omega$} if its characteristic function $\mathbbm{1}_E \colon \Omega \to \{0,1\}$ has bounded variation on $\Omega$. Equivalently, we can say $E$ has finite perimeter in $\Omega$ if
    \begin{equation*}
        \mathrm{Per}(E; \Omega) \coloneqq \sup \left\{  \int_E \mathrm{div} \varphi \, \dd x : \varphi \in C_c^\infty (\Omega; \R^N), \| \varphi \|_\infty \leq 1 \right\} < +\infty.
    \end{equation*}
\end{definition}

\begin{definition}[Reduced boundary]
    Let $E \subseteq \Omega$ be a set of finite perimeter. We define its \emph{reduced boundary} $\partial^* \!E$ as the subset of $\partial E$ for which the limit 
    \begin{equation*}
        \nu_E (x) \coloneqq - \lim_{r \to 0^+} \frac{D\mathbbm{1}_E (B(x,r))}{| D \mathbbm{1}_E | (B(x,r))},
    \end{equation*}
    exists and it is such that $| \nu_E (x) | = 1$. The vector $\nu_E (x)$ is called \emph{measure theoretic exterior normal} to $E$.
\end{definition}

\begin{theorem}[De Giorgi's structure theorem]
    Let $E \subseteq \Omega $ be a set of finite perimeter. Then $\partial^* \! E$ is countably $\mathcal{H}^{N-1}$-rectifiable and
    \begin{equation*}
        \mathrm{Per}(E; \Omega) = \mathcal{H}^{N-1}(\partial^* \! E \cap \Omega).
    \end{equation*}
    Moreover, for $x_0 \in \partial^* \! E$ and $r > 0$, by defining the blow-up set $E_r$ and the half-plane $H$ as
    \begin{equation*}
        E_r \coloneqq \frac{E - x_0}{r}, \qquad H \coloneqq \{x \in \R^N : x \cdot \nu_E (x) \geq 0 \},
    \end{equation*}
    we get that
    $$ E_r \to H \quad \text{in } L^1_\text{loc} (\Omega), \qquad \lim_{r \to 0^+} \frac{\mathcal{H}^{N-1}( \partial^* \! E \cap Q(x_0,r))}{r^{N-1}} = 1. $$
\end{theorem}

We conclude this subsection by stating the coarea formula for $(N-1)$-rectifiable sets in $\R^N$, suitably modified for our setting. We refer to \cite[Theorem 2.93]{AFP} for the general result.
\begin{theorem}[Coarea]
    Let $E \subset \R^N$ be a countably $\mathcal{H}^{N-1}$-rectifiable set. Let $f \colon E \to \R$ be a Lipschitz continuous function, and let $g : E \to [0,+\infty]$ be a Borel function. Then
    \begin{equation*}
        \int_E g(x) | \nabla^E f(x) | \, \dd \mathcal{H}^{N-1} (x) = \int_{\R} \left( \int_{E \cap \{ f = t \}} g(y) \, \dd \mathcal{H}^{N-2} (y) \right) \dd t,
    \end{equation*}
    where $\nabla^E f(x)$ represents the tangential gradient of $f$ with respect to the approximate tangent space $T_x E$.
\end{theorem}

Finally, we recall that sets of finite perimeter can be approximated strongly thanks to a result proved by de Gromard in \cite{DeGrom}.

\begin{theorem}\label{thm:strong_BV}
Let $\Omega\subset \R^M$ be an open set, and let $E\subset \Omega$ be a set of finite perimeter in $\Omega$.
Then, for each $\varepsilon>0$ there exist a set $L\subset \Omega$ of finite perimeter in $\Omega$, and a compact set $K\subset \Omega$ such that the following holds:
\begin{itemize}
\item[(i)] $L$ is piecewise $C^1$;
\item[(ii)] $\| \ca_{E} - \ca_{L} \|_{BV(\Omega)}<\varepsilon$;
\item[(iii)] $\hno((\partial L\cap \Omega)\setminus \partial^* E)<\varepsilon$;
\item[(iv)] $L\subset E+B(0,\varepsilon)$, and
    $\Omega\setminus L\subset (\Omega\setminus E)+B(0,\varepsilon)$;
\item[(v)] $K \subset \Omega\cap \partial^* E \cap \partial L$;
\item[(vi)] $\nu_E(x) = \nu_L(x)$ for all $x\in K$;
\item[(vii)] $|D\ca_E|(\Omega\setminus K)<\varepsilon$;
\item[(viii)] $\hno((\partial L \setminus \partial^* L)\cap \Omega) < \varepsilon$.
\end{itemize}
\end{theorem}

{\begin{remark}
    Condition (viii) follows from the proof of Theorem 3.1 in \cite{DeGrom}, where one defines 
    \[
        L = (F \cap \overline B) \cup (G\cap (\overline{\Omega} \setminus B))
    \]
    with $B \subset \Omega$ open and compactly contained in $\Omega$. $F$ and $G$ can be chosen such that $\partial F \cap B$ and $\partial G \cap (\overline{\Omega} \setminus B)$ are $C^1$, and $\H^{N-1}(\partial L \cap \partial B) < C\eps$. In particular, 
    \begin{equation}\label{eq:C1_small}
     \H^{N-1}(\{ x \in \partial L \cap \Omega \, | \, \nexists \, O \in \mathcal N(x), O \subset \Omega \text{ open} : O \cap \partial L \text{ is a $C^1$ surface } \}) < C\eps
     \end{equation}
    follows.
\end{remark}
}


\subsection{\texorpdfstring{$\Gamma$}{Gamma}-convergence}

    This section provides a brief overview of the definition and fundamental properties of $\Gamma$-convergence. Given the setting of this paper, we define $\Gamma$-convergence using sequences. For a comprehensive treatment of $\Gamma$-convergence in general topological spaces, we refer the reader to \cite{Dalmasobook} (see also \cite{Braides}).

    \begin{definition}
        Let $(X, \mathrm{d})$ be a metric space. A sequence of functionals $F_n: X \to [-\infty, +\infty]$ is said to \emph{$\Gamma$-converge} to a functional $F: X \to [-\infty, +\infty]$ \emph{with respect to the metric $\mathrm{d}$}, denoted by $F_n \overset{\Gamma-d}{\longrightarrow} F$, if the following two conditions are satisfied:
        \begin{itemize}
            \item[(i)] (Liminf inequality) For every $x \in X$ and for every sequence $\{x_n\}_n \subset X$ converging to $x$, the following inequality holds:
            \[
                F(x) \leq \liminf_{n\to\infty} F_n(x_n).
            \]
            \item[(ii)] (Limsup inequality) For every $x \in X$, there exists a sequence $\{x_n\}_n \subset X$ such that $x_n \to x$ and
            \[
                \limsup_{n\to\infty} F_n(x_n) \leq F(x).
            \]
        \end{itemize}
    \end{definition}

    The concept of $\Gamma$-convergence was specifically developed to provide a variational characterization of the asymptotic behavior of minimization problems. It ensures the convergence of both global minimizers and minimum values (refer to \cite[Corollary 7.20]{Dalmasobook}).

    \begin{theorem}
        Let $(X, \mathrm{d})$ be a metric space. Consider a sequence of functionals $F_n: X \to \mathbb{R} \cup \{\infty\}$ that $\Gamma$-converges to a functional $F: X \to \mathbb{R} \cup \{\infty\}$.
        Suppose that for each $n \in \mathbb{N}$, $x_n \in X$ is a minimizer of $F_n$.
        Then, any cluster point $x \in X$ of the sequence $\{x_n\}_n$ is a minimizer of $F$, and satisfies
        \[
            F(x) = \limsup_{n\to\infty} F_n(x_n).
        \]
        Moreover, if the sequence $\{x_n\}_n$ converges to $x$, the limsup above is a full limit.
    \end{theorem}


\subsection{Young measures}
Young measures are an important tool in the theory of calculus of variations. Here, we recap the definitions and the compactness property that we use. For a more general introduction to Young measures, we refer to \cite{RindlerCalcVarBook}.
\begin{definition}
    Let $\Omega \subset \R^N$ open. We call a family of probability measures $\{\nu_x\}_{x \in \Omega} \subset \mathcal P(\R^M)$ a \emph{Young measure} if
    \begin{itemize}
        \item $(x \mapsto \nu_x)$ is weakly$^*$ measurable,
        \item we have 
        \[
            \int_\Omega  \int_{\R^M} |p| \, \nu_x(p) \, \mathrm d x < \infty.
        \]
    \end{itemize}
\end{definition}
Here, we are interested in Young measures generated by $L^1$ functions.
\begin{definition}
    Let $\Omega \subset \R^N$ open. We say that a Young measure $\{\nu_x\}_{x \in \Omega}$ is \emph{generated} by a sequence of $\{u_n \}_n \subset L^1(\Omega; \R^M)$ functions if $\{\delta_{u(x)}\}_{x \in \Omega}$ converges weakly$^*$ to $\{\nu_x\}_{x \in \Omega}$ in $L^\infty_{w^*}(\Omega; \M_b(\R^M))$.
\end{definition}
{Here, $L_{w^*}^\infty(\Omega; \M_b(\R^M))$ refers to the Bochner space of all weakly-$^*$ measurable functions that are essentially bounded with values in $\M_b(\R^M)$, the space of bounded Radon measures on $\R^M$.}
One of the features of Young measures are the compactness properties. More specifically, we will use the following result (cf.\ \cite[Lemma 4.3]{RindlerCalcVarBook}:
\begin{theorem}\label{thm:compactness_YM}
    Let $\Omega \subset \R^N$ open. Let $\{u_n\}$ be an bounded sequence in $L^1(\Omega; \R^M)$. Then, there exists a Young measure $\nu$ such that $u_n$ generates $\nu$ (after extracting a subsequence). Furthermore, if $\{f(\cdot , u_n(\cdot))\}_{n \in \N} \subset L^1(\Omega; \R)$ is a bounded and equiintegrable sequence in $L^1(\Omega; \R)$ for a Carath\'eodory function $f: \Omega \times \R^M \to \R$ then 
    \[
        \int_\Omega f(x, u_n(x)) \, \mathrm d x \to \int_\Omega \int_{\R^M} f(x, p) \, \mathrm d \nu_x(p) \, \mathrm d x.
    \]
\end{theorem}


\section{Technical results}

In this section we collect the technical results needed to prove the main result of the manuscript. We decided to present them in here because they are also of separate interest for the reader.


\subsection{Properties of the geodesic distance}

In this subsection, we discuss properties of the geodesic distance. More precisely, given a non-negative Borel function $W_0: \R^M \to [0, \infty)$ we associate to $W_0$ a geodesic distance via 
\[
    \dd_{W_0}(p,q) \coloneqq  \inf\left\{ \int_{-1}^1 W_0(\varphi)|\varphi'| \, \mathrm{d}s :\, \varphi \in W^{1,1}\left([-1,1]; \R^M \right),\, {\varphi(-1) = p, \, \varphi(1) = q} \right\},
\]
for all $p,q \in \R^M$.
The typical choice of $W_0$ in the case of phase transitions is $W_0 = 2\sqrt W$ where $W$ is the double-well potential of the Cahn--Hilliard functional. If $W_0$ is such that $W_0 \geq \delta >0$, lower semi-continuous and coercive, then it is easy to see that the infimum is in fact a minimum and that the corresponding geodesics have bounded path length. If we allow $W_0$ to be zero at certain points this does not hold in general. Even if there exist geodesics, their path length does not necessarily have to be bounded (see Example \ref{ex:geodesics_infinite_Euclidean_length}). The main result of this subsection provides a general condition for this to hold true.

The setting we assume throughout the section is the existence of a continuous {non-decreasing} function $f\colon [0,\infty) \to [0, \infty)$ with $f(0) = 0$ and $f > 0$ on $(0,\infty)$ such that certain growth conditions hold true. In what follows, $a,b \in \R^M$ are always fixed single pointed wells of our potential, i.e., we assume 
$$
    \{ W_0 = 0\} = \{a,b\}.
$$
We introduce a few different growth conditions:   

\begin{enumerate}[label=(G\arabic*), ref=(G\arabic*)]
    \item\label{G1} \emph{Controlled upper growth}: There exists a constant $C_G > 0$ such that 
    \[
        W_0(u) \leq C_Gf(\min \{|u - a|, |u - b|\}),  
    \]
    for all $u \in \R^M$.
    \item\label{G2} \emph{Summable growth around the wells}: There exist $\alpha, \, R, \, C_G > 0$ such that $f$ fulfils 
    \begin{enumerate}
        \item $f(2t) \leq C_Gf(t)$ for all $t \in [0, \infty)$;
        \item for all $u \in B_R(a)\cup B_R(b)$ it holds that
        \[
            \frac{1}{C_G}f(\min \{|u - a|, |u - b|\}) \leq W_0(u) \leq C_Gf(\min \{|u - a|, |u - b|\}).
        \]
    \end{enumerate}
    \item\label{G3} \emph{$f$-coercivity}: We have that there exists a constant $C_G > 0$ such that 
    \[
        {\frac{1}{C_G}f(\min \{|u-a|, |u - b|\}) \leq W_0(u)},
    \]
    for all $u \in \R^M$.
\end{enumerate}
{
We will prove some results under slightly weaker coercivity assumptions which we introduce next: 
\begin{enumerate}[label=(G3'), ref=(G3')]
    \item\label{G3'} \emph{Non-summable growth at infinity}: There exists $g \in C([0,\infty);[0,\infty))$ such that  
    \[
        \int_0^\infty g(t) \, dt = \infty,
    \]
    and there exists a constant $C_G > 0$ such that 
    \[
        \frac{1}{C_G}g(|u|) \leq W_0(u),
    \]
    for all $u \in \R^M$.
\end{enumerate}
\begin{enumerate}[label=(G4'), ref=(G4')]
    \item\label{G4'} \emph{Strong non-summable growth at infinity}: There exists $g \in C([0,\infty);[0,\infty))$, and $\alpha, R > 0$ such that 
    \[
        g(t) \geq \alpha
    \]
    for $t \geq R$,
    and there exists a constant $C_G > 0$ such that 
    \[
        \frac{1}{C_G}g(|u|) \leq W_0(u),
    \]
    for all $u \in \R^M$.
\end{enumerate}

\begin{remark}
We remark here that \ref{G3} implies \ref{G4'} with 
\[
    g(t) \coloneqq  \begin{cases}
        0 & t < 2\max\{|a|, |b|\}, \\
        f(t/2) & t \geq 2\max\{|a|, |b|\},
    \end{cases}
\] 
$R = 2 \max\{|a|,|b|\}$ and $\alpha = f(R)$.
\end{remark}
}

To simplify the notation in the proofs, in the following we will assume without loss of generality, that
$$ f(1) < \frac{1}{2},\qquad \alpha = R = 1,\qquad  |a - b| = 2.$$
{ The general case can be obtained be a simple rescaling argument. }
We start the discussion with the following elementary lemma.

\begin{lemma}\label{arc-length-estimate}
    Let $p, q \in \R^2$, with polar coordinates $p = r_pe^{i\psi_p}, q = r_qe^{i\psi_q}$ with $\psi_p, \psi_q \in [0,2\pi)$ and $p_q, r_q \geq 0$. Consider $r(s) \coloneqq  sr_q + (1-s)r_p$ and $\psi(s) \coloneqq  s\psi_q + (1-s)\psi_p$ and set
    \[
        \gamma(s) \coloneqq  r(s)e^{i\psi(s)}.
    \]
    Then,
    \[
        \int_0^1|\gamma'|\, \mathrm{d}s \leq |r_p - r_q| + 2\pi\max\{r_p, r_q\}.
    \]
\end{lemma}

\begin{proof}
    \noindent We directly compute
    \[
        \int_0^1 | \gamma ' | \, \mathrm d s \leq \int_0^1 |r_p - r_q| + | r(s) (\psi_p - \psi_q) | \, \mathrm d s{,}
    \]
    {from which we get the required bound.}
\end{proof}

In \cite{ZunSter} it has been proved that if $W_0$ is continuous, then the minimization problem defining $\dd_{W_0} (p,q)$ admits a solution for all $p,q \in \R^M$.
Since minimizers might not exist if $W_0$ is only Borel, we use the concept of $\eps$-minimizers.

\begin{definition}
    Let $\varepsilon > 0$. A curve $\varphi \in W^{1,1}([-1,1]; \R^M)$ is called an \emph{$\eps$-minimizer} with respect to the geodesic distance {$\dd_{W_0}(\varphi(-1), \varphi(1))$} if 
    \begin{align}\label{def:eps-minimizer}
        \int_{-1}^1 W_0(\varphi)|\varphi'|\,\mathrm{d}s \leq \dd_{W_0}(\varphi(-1), \varphi(1)) + \varepsilon.
    \end{align}
    Since the left-hand side is invariant under reparameterization we will also call any curve $\psi \in W^{1,1}([c,d]; \R^M)$ an $\eps$-minimizer if there exists a reparameterization to a curve $\varphi \in W^{1,1}([-1,1]; \R^M)$ such that \eqref{def:eps-minimizer} holds.
\end{definition}

Before stating the main result of this subsection, we want to recap some well-known properties of the geodesic distance and its $\eps$-minimizers. We give some proofs for convenience of the reader.

\begin{lemma}\label{lemma:almost_minimizers_have_locality_prop}
    Let $W_0: \R^M \to [0,\infty)$ be a locally bounded Borel function, and let $\eps > 0$.
    Suppose that $\varphi \in W^{1,1}([-1,1], \R^M)$ is an $\eps$-minimizer with respect to the geodesic distance $\dd_{W_0}(\varphi(-1), \varphi(1))$.
    Then,  
    \[
        \int_{t_1}^{t_2} W_0(\varphi)|\varphi'|\,\mathrm{d}s \leq d_{W_0}(\varphi(t_1), \varphi(t_2)) + \eps
    \]
    for all $t_1 \leq t_2 \in [-1,1]$. 
\end{lemma}

\begin{proof}
    We argue by contradiction.
    Suppose that there exist $t_1, t_2 \in [-1,1]$ with $t_1 \leq t_2$ such that
    \[
        \int_{t_1}^{t_2} W_0(\varphi)|\varphi'|\,\mathrm{d}s > \dd_{W_0}(\varphi(t_1), \varphi(t_2)) + \eps.
    \]
    This directly implies 
    \begin{align*}
        \int_{-1}^1 W_0(\varphi)|\varphi'|\,\mathrm{d}s &= \int_{-1}^{t_1} W_0(\varphi)|\varphi'|\,\mathrm{d}s + \int_{t_1}^{t_2} W_0(\varphi)|\varphi'|\,\mathrm{d}s + \int_{t_2}^{1} W_0(\varphi)|\varphi'|\,\mathrm{d}s \\
        & > \dd_{W_0}(\varphi(-1), \varphi(t_1))+ (\dd_{W_0}(\varphi(t_1), \varphi(t_2)) + \eps) + \dd_{W_0}(\varphi(t_2), \varphi(1)) \\
        & \geq \dd_{W_0}(\varphi(-1), \varphi(1)) + \eps.
    \end{align*}
    This would contradict that $\varphi$ was an $\eps$-minimizer with respect to $\dd_W(\varphi(-1), \varphi(1))$.
\end{proof}

The next lemma states that even if the function $W_0$ is not continuous, the corresponding geodesic distance is still Lipschitz continuous.

\begin{lemma}\label{lemma:geodesic_distance_is_locally_lipschitz}
    Let $W_0: \R^M \to [0, \infty)$ be a locally bounded Borel function. Then:
    \begin{itemize}
    \item[(i)] $\dd_{W_0}$ is locally Lipschitz;
    \item[(ii)] For any bounded convex set $K \subset \R^M$, it holds that that
    \[
    L_{K \times K} \leq \|W_0\|_{L^\infty(K)},
    \]
    where $L_{{K \times K}}$ is the Lipschitz constant of $\dd_{W_0}$ in $K\times K$;
    \item[(iii)]  If \ref{G1} holds there exists a constant $C > 0$ (only dependent on $|a|$, $|b|$, $f$ and $C_G$) such that
    \[
        L_{{\overline{B_r(0)} \times \overline{B_r(0)}}} \leq C(1 + f({2}r)),
    \]
    for any ball $B_r(0)$.
    \end{itemize}
\end{lemma}
The proof can be found, for instance, in \cite[Lemma 2.7]{DeutschSolidSolid}. It {is} well-known that $\eps$-minimizers with respect to the geodesic distance are uniformly bounded if the potential is coercive. Here, we show that, in fact, the weaker condition of a non-summable growth condition at infinity suffices.

\begin{lemma}\label{lemma:eps_minimizers_are_bounded}
    Suppose that $W_0: \R^M \to [0, \infty)$ is a locally bounded Borel function, and $\ref{G3'}$ holds. Moreover, let $\eps \in (0,1)$ and $r > 0$. Then, there exists a constant $C > 0$ (only dependent on  $r$, $W_0$, {$g$} and $C_G$) such that for all $p, q \in B_r(0)$ and all $\eps$-minimizers $\varphi$ with respect to $\dd_{W_0}(p,q)$ we have 
    \[
        \| \varphi \|_{L^\infty([-1,1];\R^M)} \leq C. 
    \]
    {Moreover, if also \ref{G1} holds, then $C$ is only dependent on $r$, $|a|$, $|b|$, $f$, $g$ and $C_G$.}
\end{lemma}

\begin{proof}
    {By \ref{G3'}, we have 
    \[
        \int_0^\infty g \, dt = \infty. 
    \]
    In particular, there exists $r \leq t_0 < \dots < t_n$ such that 
    \begin{align}\label{eq:correct_partition_choice}
        \sum_{k = 1}^n \min_{t \in [t_{i - 1}, t_i]}g(t) (t_i - t_{i - 1}) > 2C_G(2Lr + 1).
    \end{align}
    Note that $t_n$ only depends on $L$, $r$, $C_G$ and $g$. 
    Now, by Lemma \ref{lemma:geodesic_distance_is_locally_lipschitz}, we have 
    \begin{equation}\label{eq:est_W0_ball}
        \max_{p,q \in B_r(0)}\dd_{W_0}(p,q) \leq 2Lr < \infty,
    \end{equation}
    where $L = L_{\overline{B_r(0)} \times \overline{B_r(0)}} $ is the Lipschitz constant of $d_{W_0}$ on $\overline{B_r(0)} \times \overline{B_r(0)}$ which is only dependent on $\|W_0\|_{L^\infty(\overline{B_r(0)})}$, if \ref{G1} does not hold, and $|a|$, $|b|$, $f$ and $C_G$ if \ref{G1} does hold.
    }
    
    Now, we argue by contradiction: Suppose that there exist $\widetilde p, \widetilde q \in B_r(0)$ such that 
    \[
        \| \varphi \|_{L^\infty([-1,1];\,  \R^M)} > t_n 
    \]
    for an $\eps$-minimizer $\varphi$ of $\dd_{W_0}(\widetilde p, \widetilde q)$. Then, there exist $-1 \leq s_0 < \dots < s_n \leq 1$ such that {$|\varphi(s_i)| = t_i$} for $i = 0,\dots, n$. {For each interval $(s_{i-1}, s_{i})$ we pick a sub-interval $(c,d)$ such that 
    \[
        |\varphi(c)| = t_{i-1} , \quad \, |\varphi(d)| = t_{i}, \quad \text{ and } \quad \, |\varphi(s)| \in (t_{i-1}, t_{i})
    \]
    for all $s \in (c,d)$. Now, we observe 
    \begin{align*}
        \int_{s_{i-1}}^{s_i} W_0(\varphi)|\varphi'|\,ds
            &\geq \frac{1}{C_G} \int_{c}^{d} g(|\varphi|)|\varphi'|\,ds \\
            &\geq \frac{1}{C_G} \min_{t \in [t_{i - 1}, t_i]}g(t) \int_{c}^{d}|\varphi'|\,ds \\
            &\geq  \frac{1}{C_G}\min_{t \in [t_{i - 1}, t_i]}g(t) |\varphi(d) - \varphi(c)| \\
            &=\frac{1}{C_G}\min_{t \in [t_{i - 1}, t_i]}g(t) |t_i - t_{i-1}|.
    \end{align*}}
    Using \eqref{eq:correct_partition_choice} and \eqref{eq:est_W0_ball}, we obtain
    \begin{align*}
        \dd_{W_0}(\widetilde{p},\widetilde{q}) + \varepsilon
         &\geq \int_{-1}^1 W_0(\varphi)|\varphi'|\,ds \\
         &\geq \sum_{k = 1}^n \int_{s_{i-1}}^{s_i} W_0(\varphi)|\varphi'|\,ds \\ 
        &\geq \sum_{k = 1}^n \min_{t \in [t_{i - 1}, t_i]}g(t) (t_i - t_{i - 1}) \\
        &\geq {2(2L r + 1)} \\
        &\geq 2\left(\max_{p,q \in B_r(0)}\dd_{W_0}({p, q}) + 1 \right) \\
        &> \dd_{W_0}(\widetilde p, \widetilde q) + \eps.
    \end{align*}
    This contradicts $\varphi$ being an $\eps$-minimizer. {Setting $K = t_n$ we conclude.}
\end{proof}

\begin{remark}\label{remark:geodesic_parameter_discussion_1}
    The constant in Lemma \ref{lemma:geodesic_distance_is_locally_lipschitz} can be made explicit in the following sense. Since $g$ is non-negative in $(0,\infty)$, the function 
    $${G(T)} \coloneqq \int_r^T  g(t) \, d t$$
    is non-decreasing for any $r > 0$. Thus, it admits a non-decreasing left inverse $H\colon(0,\infty) \to (r, \infty)$. In particular, we have 
    \[
        {H(K)} = 3\left(\max_{p,q \in B_r(0)}\{\dd_{W_0}(p, q)\} + 1 \right),
    \]
    which we can express via the left inverse as
    \begin{align}\label{eq:expression_of_constant}
        K = {H}\left(3\left(\max_{p,q \in B_r(0)}\{\dd_{W_0}(p, q)\} + 1 \right)\right).
    \end{align}
    This has the following consequence: if \ref{G1} and \ref{G4'} hold, repeating the proof of Lemma \ref{lemma:eps_minimizers_are_bounded} with $\widetilde g= \chi_{[R, \infty)}\alpha$ we obtain that the resulting $K$ is then just an affine combination of $R$ (from \ref{G4'}) and $f$ (from \ref{G1}). Indeed, this can also be derived by choosing the partition $t_0 = R$ and $t_1 = K \coloneqq  R + C(1 + f(2r))$ in \eqref{eq:correct_partition_choice} for a suitable large constant $C > 0$. In particular, we have
    \begin{align}\label{eq:K_f_bound}
        K \leq C(1 + R + f(2r)) 
    \end{align}
    where $C$ is only dependent on $|a|$, $|b|$, $C_G$ and $\alpha$. 
\end{remark}

\begin{remark}\label{remark:geodesic_parameter_discussion_2}
    A consequence of Remark \ref{remark:geodesic_parameter_discussion_1} is the following. Let $g_1, g_2 \in C([0,\infty); [0,\infty))$ with $g_1 \leq g_2$ both satisfying $\ref{G3'}$, in the sense that they fulfil the non-summability condition at $\infty$. Then, Lemma \ref{lemma:eps_minimizers_are_bounded} now gives constants $K_{g_1}$ and $K_{g_2}$. {By \eqref{eq:expression_of_constant}, we directly see that we can choose the constants such that $K_{g_2} \leq K_{g_1}$ holds. Vice versa, if \ref{G1} and \ref{G4'} holds with $f_1 \leq f_2$ then \eqref{eq:K_f_bound} implies that we can choose $K_{f_1} \leq K_{f_2}$ in the conclusion of the Lemma.}
\end{remark}

Next, we will observe that there exists a minimizing sequence of curves $\varphi_m$ for the geodesic distance which has bounded path-length.

\begin{theorem}\label{existence-of-geodesics-with-bounded-path-length}
    Suppose that $W_0$ fulfils \ref{G1}, \ref{G2} and \ref{G3}. Let $r > 0$ and let $p, q \in B_r(0)$. Then, there exists a sequence $\{ \varphi_m \}_m \subset W^{1,1}([-1,1]; \R^M)$ with $\varphi_{m}(-1) = p$, $\varphi_{m}(1) = q$ and 
    \[
        \int_{-1}^1 W_0(\varphi_m)|\varphi_m'| \,\mathrm{d}s \leq \dd_{W_0}(p, q) + C\frac{1}{2^m}
    \]
    such that
    \begin{equation}\label{eq:geodesic_bounded_Euclidean}
        \int_{-1}^1 |\varphi_m'| \,\mathrm{d}s \leq C(1 + f(|p|)|p| + f(|q|)|q|),
    \end{equation}
    with a constant $C >0 $ only depending on $f$, $|a|$, $|b|$, $r$ and $C_G$. 
    
    As a consequence, if $W_0$ is lower-semicontinuous, then the infimization problem defining $\dd_{W_0}(p, q)$ admits a solution {fulfilling \eqref{eq:geodesic_bounded_Euclidean}}. 
\end{theorem}

\input{fig_geodesic}

\begin{proof} We divide the proof in steps. {Step 0 introduces the notation and setup of the proof. In Step 1, we show that we can find an $\eps$-minimizer that does not oscillate `very close' to the wells. In Step 2, that around the wells we can estimate the path length using \ref{G2}, and in Step 3, we estimate the path length away from the wells using the coercivity condition \ref{G3}.} \\[0.5em]
    \textbf{Step 0: The setup.} {Let $\delta > 0$ be fixed, but chosen sufficiently small which will be specified later on}.
    Let $\{ \widetilde \varphi_m \}_m\subset W^{1,1}([-1,1]'\R^M)$ be a sequence such that
    \[
       \int_{-1}^1 W_0(\widetilde \varphi_m)|\widetilde \varphi_m'| \,\mathrm{d}s \leq \dd_{W_0}(p, q) + \eta_m \dd_{W_0}(a,b), 
    \]
    for every $m \in \N$ { with }
    \begin{align}\label{def:eta_m}
        {\eta_m \coloneqq  \frac{f(\delta 2^{-m})}{2^m}.}
    \end{align}
    {In what follows, we assume that, for $l = 1,2$
    \[
        (\widetilde \varphi_m)^{-1} \left(\bigcup_{k \geq m}\T^{l,k}\right) \neq \emptyset.
    \]
    If not, then the respective substitution that is applied in Step 1 is not needed, and Step 2 and 3 can be carried out analogously.}

    Note that $\eta_m < 1/2$ by our assumption that $f(1) < 1/2$ holds.
    Define the following sets for $k \in \N$:
    \[
        \T^{1,k} \coloneqq  \left\{ u \in \R^M: \frac{\delta}{2^{k+1}} \leq |u-a| \leq \frac{\delta}{2^{k}}\right\},
    \]
    \[
        \T^{2,k} \coloneqq  \left\{ u \in \R^M: \frac{\delta}{2^{k+1}} \leq |u-b| \leq \frac{\delta}{2^{k}}\right\}.
    \]
    For $l = 1,2$, {we denote the entry and exit time of $\widetilde \varphi$ into $B_{\delta/2^m}(a)$ and $B_{\delta/2^m}(b)$ by}
    \[
        t_0^l \coloneqq  \inf \, (\widetilde \varphi_m)^{-1} \left(\bigcup_{k \geq m}\T^{l,k}\right) \quad \text{ and } \quad
        t_1^l \coloneqq  \sup \, (\widetilde \varphi_m)^{-1} \left(\bigcup_{k \geq m}\T^{l,k}\right),
    \]
    respectively.
    
    By the local optimality of $\eps$-minimizers for the geodesic distance (cf.\ Lemma \ref{lemma:almost_minimizers_have_locality_prop}), we know that for any $r, t \in [-1,1]$, $r < t$ with $\widetilde \varphi_m(r), \widetilde \varphi_m(t) \in B_\delta(a)$
    \[
        \int_{r}^tW_0(\widetilde \varphi_m)|\widetilde \varphi_m'|\, \mathrm ds  \leq {d_{W_0}(\widetilde \varphi_m(r), \widetilde \varphi_m(t)) + \eta_m {\dd_{W_0}(a,b)}}  \leq \omega(\delta) + {\eta_m} {\dd_{W_0}(a,b)}
    \]
    for some modulus of continuity $\omega$ {of $d_{W_0}$} on a suitable large ball. In particular, {if $\delta > 0$ is chosen sufficiently small such that $\omega(\delta) < \dd_{W_0}(a,b)/4$, and by our choice of $\eta_m$ in \eqref{def:eta_m}, we obtain
    \begin{equation}\label{eq:est_34}
        \int_{s}^tW_0(\widetilde \varphi_m)|\widetilde \varphi_m'|\, \mathrm ds < \frac{3}{4}\dd_{W_0}(a,b) < \dd_{W_0}(a,b).
    \end{equation}
    Now, suppose that there exists $s_0 \in (r,t)$ such that $\widetilde \varphi_m(s_0) \in B_\delta(b)$. By definition of $\dd_{W_0}$ we observe 
    \begin{align*}
        \int_{s}^tW_0(\widetilde \varphi_m)|\widetilde \varphi_m'|\, \mathrm ds \geq \dd_{W_0}(\widetilde \varphi_m(r), \widetilde \varphi_m(s_0)) + \dd_{W_0}(\widetilde \varphi_m(s_0), \widetilde \varphi_m(t)) \geq 2\dd_{W_0}(a,b) - 2\omega(\delta).
    \end{align*}
    This is a contradiction with \eqref{eq:est_34}.} In particular, this implies 
    \[
        \sup \, (\widetilde \varphi_m)^{-1} (B_\delta(a)) < \inf \, (\widetilde \varphi_m)^{-1} (B_\delta(b))
    \]
    or 
    \[
        \sup \, (\widetilde \varphi_m)^{-1} (B_\delta(b)) < \inf \, (\widetilde \varphi_m)^{-1} (B_\delta(a)).
    \]
    Without loss of generality, we assume that we are in the first case. Here, we immediately infer $t_1^1 \leq t_0^2$. 
    
    Now, in what follows, we will substitute $\widetilde \varphi_m$ with a straight path in $\bigcup_{k \geq m}\T^{l,k}$. This substitution {changes} the path length and the length with respect to $W_0$ only by a {controllable} amount. Then, we estimate the path length of $\widetilde \varphi_m$ in each $\T^{l,k}$ separately by finding a suitable competitor for the geodesic distance in this set. Lastly, we estimate the path length of the remaining curve. \\[-0.5em]
    
    \textbf{Step 1: Substituting $\widetilde \varphi_m$ around the wells.}
    Write $\widetilde \varphi_m(t_0^1) = P_1 + a$ and $\widetilde \varphi_m(t_1^1) = P_2 + a$ with $P_1, P_2 \in \R^M$. Now, define $\theta_{1,m}$ as the straight line from $P_1$ to $P_2$ (the red segments in Figure \ref{fig:geodesic})
    \[
        \theta_{1,m}(t) \coloneqq  \frac{t_1^1 - t}{t_1^1 - t_0^1}P_1 + \frac{t - t_0^1}{t_1^1 - t_0^1}P_2 + a.
    \]
    Analogously, we define $\theta_{2,m}$.
    Consider the curve 
    \[
        \varphi_m \coloneqq  \begin{cases}
            \theta_{1,m} & \text{ in } {[t_0^1, t_1^1]}, \\
            \theta_{2,m} & \text{ in } {[t_0^2, t_1^2]}, \\
            \widetilde \varphi_m & \text{ otherwise}.
        \end{cases}
    \]
    Notice now that by the {upper $f$-bound given by \ref{G1}}
    \begin{align*}
        \int_{t_0^1}^{t_1^1} {W_0(\theta_{1,m})}|\theta_{1,m}'| \,\mathrm{d}s &\leq Cf\left( \frac{\delta}{2^m} \right)|P_1 - P_2|
        \leq \frac{C}{2^m},
    \end{align*}
    where we used $|P_1 - P_2| \leq 2^{-(m-1)}$ and $f(1) < 1/2$ in the last line. Similarly, we infer 
    \begin{align*}
        \int_{t_0^1}^{t_1^1} {W_0(\theta_{2,m})}|\theta_{2,m}'| \,\mathrm{d}s &\leq \frac{C}{2^m}.
    \end{align*}
    In particular, this substitution {changes} the length with respect to $W_0$ by a small {controllable} amount, which means that for $m$ sufficiently large these are $1/2$-minimizers of the geodesic distance. Applying Lemma \ref{lemma:eps_minimizers_are_bounded}, we get a constant $K > 0$  such that 
    \begin{align}\label{eq:uniform_estimate_sequence}
        \|\varphi_m\|_{L^\infty} \leq K.
    \end{align}
    Now, we want to estimate the path length of $\varphi_m$ in the interval $[-1, t^1_0]$. Assume $t^1_0 > -1$, otherwise there is nothing to estimate. Now, we define inductively a decreasing sequence $s_0 > s_1 > \dots > s_{m}$ the following way: Set
    \[
        s_0 \coloneqq  t^1_0,
    \]
    and, for $k = 1, .., m$
    \[
        s_{k} \coloneqq  
            \inf (\widetilde \varphi_m)^{-1}(\T^{1, m - k}) \cap [-1, s_{k - 1}].
    \]
    {We note that $s_k$ is the first time that the curve enters the layer $\T^{1, m - k}$ after $s_{k_1}$ meaning that before that we cannot pass through $\T^{1, m - k}$.} 
    Notice that these values need not be well defined for all $k$. {For instance the curve could be starting in the layer $\T^{1, m - 1}$, i.e., then we have $s_{1} = -1$ and $s_k$ for $k > 1$ is then not defined anymore.} We denote the last index where this finite sequence is well defined with $k_0$. Notice that $s_{k_0} = -1$ has to hold if $k_0 < m$. \\[-0.5em]
    
    \textbf{Step 2: Estimating the path length of $\varphi_m|_{[s_{k+1}, s_k]}$.} 
    Notice first, that by construction and by \ref{G2}
    \[
        W_0(\widetilde \varphi_m)|_{[s_{k + 1}, s_{k}]} \geq f\left(\frac{\delta}{2^{m - k}}\right).
    \]
    Now, we use the local optimality property of Lemma \ref{lemma:almost_minimizers_have_locality_prop}, i.e., we have for every $k \in \{0,.., k_0 - 1\}$ 
    \[
        \int_{s_{k + 1}}^{s_{k}} W_0(\widetilde \varphi_m)|\widetilde \varphi_m'| \,\mathrm{d}s \leq \dd_{W_0}\left(\widetilde \varphi_m(s_{k + 1}), \widetilde \varphi_m(s_{k })\right) + {\eta_m \dd_{W_0}(a,b)}.
    \]
    Consider the circular arc $\zeta^k_m$ connecting $\widetilde \varphi_m(s_{k+1})$ and $\widetilde \varphi_m(s_k)$. More precisely, write $\widetilde \varphi_m(s_{k+1}) = P_1 + a$ and $\widetilde \varphi_m(s_{k}) = P_2 + a$ with corresponding $P_1, P_2 \in \R^M$. Then, let $\gamma$ be the circular arc from Lemma \ref{arc-length-estimate} connecting $P_1$ and $P_2$ in the plane spanned by $\{0, P_1, P_2\}$ which we reparameterise to the interval $[s_{k+1}, s_k]$. Now, set 
    \[
        \zeta^k_m \coloneqq  \gamma + a.
    \]
    We have 
    \begin{align*}
        f\left(\frac{\delta}{2^{m - k}}\right) \int_{s_{k + 1}}^{s_{k}}|(\widetilde \varphi_m)'|\, \mathrm{d}s 
        &\leq \int_{s_{k + 1}}^{s_{k}} W_0(\widetilde \varphi_m)|(\widetilde \varphi_m)'|\, \mathrm{d}s \\
        &\leq \dd_{W_0}\left(\widetilde \varphi_m(s_{k + 1}), \widetilde \varphi_m( s_{k })\right) + {\eta_m \dd_{W_0}(a,b)} \\
        &\leq \int_{s_{k + 1}}^{s_{k}}W_0(\zeta_m^{k})|(\zeta_m^{k})'|\, \mathrm{d}s + C{\frac{f(\delta 2^{-(m-k)})}{2^{m-k}}},
    \end{align*}
    {where we used \eqref{def:eta_m} in the last inequality.}
    Now, observe that we obtain the following estimate by using Lemma \ref{arc-length-estimate} and the doubling condition on $f$:
    \begin{align*}
        \int_{s_{k + 1}}^{s_{k}} W_0(\zeta^{k}_m)|(\zeta^{k}_m)'|\, \mathrm{d}s & \leq c_2\int_{s_{k + 1}}^{s_{k}}f(|\zeta^k_m-a|)|( \zeta^{k}_m)'|\, \mathrm{d}s \\
        & \leq C f\left(\frac{\delta}{2^{m - k - 1}}\right)\int_{s_{k + 1}}^{s_{k}}|( \zeta_m^{k})'|\, \mathrm{d}s \\
        &\leq Cf\left(\frac{\delta}{2^{m-k}} \right) \frac{1}{2^{m - k}}.
    \end{align*}
    Putting these inequalities together { and dividing by $f(\delta 2^{-(m-k)})$}, we derive
    \[
        \int_{s_{k+1}}^{s_k}|(\widetilde \varphi_m)'|\, \mathrm{d}s \leq \left(\frac{C}{2^{m - k}} \right) .
    \]
    Summing over $k$ we infer
    \begin{align}\label{eq:geodesic_estimate_1}
        \int_{s_{k_0}}^{s_{0}} |(\widetilde \varphi_m)'| \, \mathrm{d}s &\leq C.
    \end{align}

    \textbf{Step 3: Estimating the path length of $\varphi_m|_{[-1, s_{k_0}]}$.} 
    If $k_0 < m$ we are done since then $s_{k_0} = -1$ and we have nothing to estimate. In the other case, we observe that \eqref{eq:uniform_estimate_sequence} implies $W_0(\varphi_m|_{(-1, s_m)}) \geq \widetilde \alpha$ where 
    \begin{align}\label{eq:tilde_alpha}
        {\widetilde \alpha \coloneqq  \min\left\{f(2R), \inf_{B_R(0)\setminus(B_\delta(a)\cup B_\delta(b))}W_0(u) \right\}  > 0, }
    \end{align}
    with $R = 2\max\{|a|, |b|\}$.
    
    Using this, we estimate with \ref{G1}: 
    \begin{align}\label{eq:geodesic_estimate_2}
        \widetilde \alpha \int_{-1}^{s_{m}} |(\widetilde \varphi_m)'| \, \mathrm{d}s &\leq \int_{-1}^{s_m} W_0(\widetilde \varphi_m)|(\widetilde \varphi_m)'|\, \mathrm{d}s \nonumber \\
        &\leq \dd_{W_0}\left(\widetilde \varphi_m(-1), \widetilde \varphi_m( s_{m} )\right) + {\eta_m \dd_{W_0}(a,b)} \nonumber \\
        &\leq \left(\sup_{u \in B_{1 + |p|}(a)} W_0(u)\right) |p - \widetilde \varphi_m(s_m)| + {\eta_m \dd_{W_0}(a,b)} \nonumber \\
        &\leq C \left( f(1 + |p|)|p - a| \right) + {\eta_m \dd_{W_0}(a,b)},
    \end{align}
    {where in the second-to-last inequality we used Lemma \ref{lemma:geodesic_distance_is_locally_lipschitz}, together with $\widetilde \varphi_m (-1) = p$.
    Note that in the last inequality we used that $|a - \widetilde \varphi_m(s_m)| = \delta$ by construction and $|p - a| > \delta$ and, therefore,
    \[
        |p - \widetilde \varphi_m(s_m)| \leq |p-a| + \delta \leq 2|p-a|
    \]
    holds.}
    Notice now that the doubling condition of $f$ (from \ref{G2}) implies $f(1 + |p|) \leq C(1 + f(|p|))$ {since for $|p| \leq 1$ we have 
    \[
        f(1 + |p|) \leq 2
    \]
    and for $|p| > 1$ we infer }
    \[
        {f(1+|p|) \leq f(2|p|) \leq C_Gf(|p|).}
    \]
    So we can use \eqref{eq:geodesic_estimate_1} and \eqref{eq:geodesic_estimate_2} to obtain
    \[
        \int_{-1}^{t_0^1} |(\varphi_{m})'| \, \mathrm{d}s \leq C(1 + f(|p|)|p|).
    \]
    In an analogous way, one estimates the path lengths in the remaining intervals $[t_1^1, t_0^2]$ and $[t_0^2, t_1^2]$ to derive 
    \[
        \int_{t_1^1}^{t_0^2} |(\varphi_{m})'| \, \mathrm{d}s \leq C(1 + f(|q|)|q| + f(|p|)|p|).
    \]
    and 
    \[
        \int_{t_0^2}^1 |(\varphi_{m})'| \, \mathrm{d}s \leq C(1 + f(|q|)|q|).
    \]

    \textbf{Step 4: Existence of a solution for $\dd_{W_0}(p,q)$.} 
    This follows by a standard argument based on the Arzel\`{a}-Ascoli Theorem. For instance, see \cite[Lemma 3.1]{CriGra} and note that the continuity of the potential can be {weakened to lower semicontinuity. Thanks to the lower semicontinuity of the total variation, the bounds on the minimizing sequence transfer to the limiting geodesic.}
\end{proof}
{
\begin{remark}
    We note that if \ref{G1}, \ref{G2} and \ref{G3'} hold then Theorem \ref{existence-of-geodesics-with-bounded-path-length} still concludes the same way if $W$ is bounded away from $0$ away from the wells, i.e., if 
    \[
        \inf_K W_0 > 0
    \]
    for any compact set $K \subset \R^M$ that does not contain $a$ and $b$. Indeed, one then only needs to substitute $\widetilde \alpha$ in \eqref{eq:tilde_alpha} with 
    \[
        \widetilde \alpha \coloneqq  \inf_{B_K(0)\setminus(B_\delta(a)\cup B_\delta(b))}W_0(u)
    \]
    and $K > 0$ given by $\eqref{eq:uniform_estimate_sequence}$.
\end{remark}
}

{
\begin{remark}\label{remark:continuity_of_geodesic_distance}
    Let $K \subset \R^M$ be compact. With the notation of \ref{H6} we observe that if \ref{H1}--\ref{H6} and \ref{H8} hold then we have for any $p, q \in K$ and any $i = 1, \dots, k$
    \begin{align}\label{eq:cont_geo}
        |d_{W_i}(x; p, q) - d_{W_i}(y; p, q)| \leq C_K\omega(|x - y|)
    \end{align}
        
    where $\omega$ is the modulus of continuity on $2\sqrt W_i$ on $\overline \Omega \times B_R(0)$ for a suitably large ball.
    Indeed, take a $\eps$-minimizer $\varphi \in W^{1,1}([-1,1]; \R^M)$ in the definition of $d_{W_i}(y; p, q)$. Due to \ref{H8} and Theorem \ref{existence-of-geodesics-with-bounded-path-length}, we have that there exists a constant $C = C(K, \sqrt f, \|a_i\|_\infty, \|b_i\|_\infty)) $ such that 
    \[
        \int_{-1}^1 |\varphi'| \, dt \leq C.
    \]
    In particular, $\varphi \in B_R(0)$ for $R > 0$ sufficiently large. Consequently, we infer
    \[
        d_{W_i}(x; p, q) - d_{W_i}(y; p, q) \leq \int_{-1}^1 2(\sqrt W_i(x, \varphi) - \sqrt W_i(y, \varphi))|\varphi'| \, dt \leq C\omega(|x - y|).
    \]
    Due to the symmetry of the problem, we infer \eqref{eq:cont_geo} and conclude.

    With the same reasoning it is possible to prove that $\dd_{W_{i,j}} (x; p, q)$ is continuous for $x \in \partial \Omega_i \cap \partial \Omega_j$, and $p,q \in K$, for any $i \neq j$.
\end{remark}}

\begin{remark}
Note that we are able to prove the existence of a geodesics for $\dd_{W_0}(p,q)$ even if the potential $W_0$ is not continuous, but lower semi-continuous and satisfies certain growth conditions.
Our result is in the same spirit as that obtained by Zuniga and Sternberg in \cite{ZunSter}, where they require the potential $W_0$ to be continuous, but without assuming any growth condition near the wells.
In their case, though, they cannot ensure that geodesics have finite Euclidean length. 
Indeed, as Example \ref{ex:geodesics_infinite_Euclidean_length} shows, there exist potentials satisfying the assumptions of the main result by Zuniga and Sternberg whose corresponding geodesics have infinite Euclidean length.
\end{remark}

\begin{example}\label{ex:geodesics_infinite_Euclidean_length}
We sketch a construction of a {single well} potential for which geodesics with respect to the induced distance exist, but necessarily have unbounded Euclidean path length. \\
Let 
$$ \frac{1}{2} < r_4 < r_3  < \frac{3}{4} < r_2 < r_1 < 1.$$
For $\delta < \pi/16$ consider the annular sectors
\[
    A_1 \coloneqq \left\{ x \in \R^2 : r_2 < |x| < r_1 \, -\frac{\pi}{2} - \delta < \arg(x) < -\frac{\pi}{2} + \delta \right\},
\]
and
\[
    A_2 \coloneqq \left\{ x \in \R^2 : r_4 < |x| < r_3, \, \frac{\pi}{2} - \delta < \arg(x) < \frac{\pi}{2} + \delta \right\}
\]
where $\arg: \R^2 \to [-\pi, \pi)$ denotes the polar angle.
Let 
\[
    C_r \coloneqq  \{ x \in \R^2 : |x| = r \}
\]
be the circle lines with radius $r > 0$, and 
\[
    S_1 = \{0\}\times(-1, -3/4) \qquad \text{ and } \qquad S_2 = \{0\} \times (1/2, 3/4).
\]
For $M_1 = 1$ and $M_1/2 > \eps_1 > \eps_2$ define
\[
    \widetilde W(p) \coloneqq  \begin{cases}
        \eps_1 & p \in C_1 \cup S_1, \\
        M_1 & p \in A_1, \\
        \eps_1 & p \in C_{3/4}, \\
        M_1 & p \in A_2, \\
        \eps_2 & p \in C_{1/2}, \\
        (1-t) \eps_2 + t \eps_1 & p = \left(0, {\frac 1 2 + \frac t 4} \right) \in S_2
    \end{cases}
\]
and let $W$ be the natural piecewise linear extension to $\overline{B_1/B_{1/2}}$. If $\eps_1$ is chosen small enough, any geodesic $\varphi$ with respect to $\dd_W$ connecting two points $p_1 \in C_1$ and $p_2 \in C_{1/2}$ will not pass through $A_1, A_2$. Indeed, suppose that $\varphi(s) \in A_1 \cup A_2$ and $\varphi(t) \in C_{3/4}$ for some $s,t \in [-1,1]$. Then,
\[
    \left| \int_{s}^t W_0(\varphi)|\varphi'| \, \dd s \right| \geq C \int_{{\eps_1}}^{M_1} r \dd r \geq CM_1^2.
\]
But taking an injective competitor $\widetilde \varphi$ that is contained in $C_1 \cup S_1 \cup C_{3/4} \cup S_2 \cup C_{1/2}$ we derive 
\[
    \dd_{W_0}(p_1, p_2) \leq C\eps_1
\]
which yields a contradiction for $\eps_1 $ chosen small enough since $\varphi$ is a geodesic with respect to $\dd_{W_0}(p_1, p_2)$. This has the implication that the geodesic does not pass through $A_1$ and $A_2$, and we have that the path length of $\varphi$ is bounded from below by $1$.  \\
To extent $W$ to $\overline{B_{1/n}\setminus B_{1/(n+1)}}$, we repeat this construction with suitable $M_n > 0$ and $\eps_n > 0$ such that
$$  
    M_n \xrightarrow{n \to \infty} 0,
$$ 
and that any geodesic $\psi$ with respect to $d_W$ connecting two points in $C_{1/n}$ and $C_{1/(n+1)}$ fulfils
\begin{align*}
    \int_{-1}^1 |\psi'| \, \dd r \geq \frac{C}{n}.
\end{align*}
By extending $W$ with $M_1$ outside of the closed unit ball, we get a continuous potential defined on $\R^M$. In particular, geodesics between two points exist (cf.\ \cite{ZunSter}). 
Next, let $\varphi$ be a geodesic connecting $p \in C_1$ with $0$. Since $\varphi([-1,1]) \cap C_{1/n} \neq \emptyset$, we can take an increasing sequence $s_n \in \varphi^{-1}(C_{1/n})$. Now, we can compute a lower bound on the path length by 
\begin{align*}
    \int_{-1}^1 |\varphi'| \dd r \geq \sum_{n \in \N} \int_{s_n}^{s_{n+1}} |\varphi'| \, \dd r \geq  
    \sum_{n \in \N} \frac 1 n = +\infty.
\end{align*}
We have shown that any geodesic with respect to this constructed potential $W$ always has unbounded path length. {It is possible to use a similar idea to build an example also in the case of multi-well potentials.}

\begin{figure}
    \centering
    \includegraphics[width=0.5\linewidth]{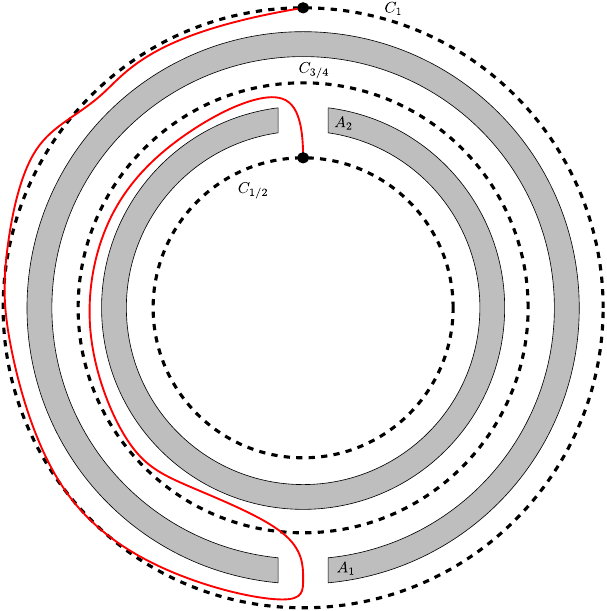}
    \caption{The idea for the construction of potentials $W$ in Example \ref{ex:geodesics_infinite_Euclidean_length}. Heuristically, we create obstacles in the energy landscape such that geodesics connecting points in $C_1$ to points in $C_{1/2}$ cannot pass through.}
\end{figure}

\end{example}
In our setting, we will encounter potentials that jump. For this reason, we introduce an adapted version of a geodesic distance function.
\begin{definition}
    Let $W_0, W_1: \R^M \to [0, \infty)$ be non-negative Borel functions. We define the \emph{adapted geodesic} distance by
    \[
        \dd_W(p,q) \coloneqq  \inf_{r \in \R^M} (\dd_{W_0}(p, r) + \dd_{W_1}(r, q)).
    \]
\end{definition}

An analogous result as in Lemma \ref{lemma:geodesic_distance_is_locally_lipschitz} holds also for the adapted distance. More precisely, we have the following statement.
\begin{lemma}\label{lemma:adapted_geodesic_distance_is_locally_lipschitz}
    Suppose that $W_0, W_1: \R^M \to [0,\infty)$ are locally bounded Borel functions. Then, $\dd_W$ is locally Lipschitz, and the local Lipschitz constants of $\dd_W$ are lower than or equal to the local Lipschitz constants of $\dd_{W_0}$ and $\dd_{W_1}$.
\end{lemma}

If \ref{G3'} holds then we also can prove an analogous result as in Lemma \ref{lemma:eps_minimizers_are_bounded}, whose proof works in exactly the same way. {We formulate it here with the stronger assumption \ref{G3}.}
\begin{proposition}\label{prop:infimum_attained_in_ball}
    Suppose that $W_0, W_1: \R^M \to [0,\infty)$ are locally bounded Borel functions that fulfill \ref{G3}, and let $r > 0$. Then, there exists $R_0 > r$ only depending on $r$,  $|a|$, $|b|$, $W_0$ , $g$ and $C_G$ (given in \ref{G3}) such that 
    \[
        \dd_W(p,q) = \inf_{s \in B_{R_0}(0)} (\dd_{W_0}(p, s) + \dd_{W_1}(s, q)),
    \]
    for all {$p, q \in B_r(0)$}.
    
    {Moreover, if also \ref{G1} holds, then $R_0$ is only dependent on $r$, $|a|$, $|b|$, $f$, $g$ and $C_G$.}
\end{proposition}

The next corollary is a direct application of Proposition \ref{prop:infimum_attained_in_ball} and Lemma \ref{lemma:eps_minimizers_are_bounded}.
\begin{corollary}\label{corol:truncation_constant}
    Suppose that $W_0, W_1$ are locally bounded Borel functions {fulfilling} \ref{G1}, \ref{G3}, and let $r > 0$. Then, there exists $M_0 = M_0(r, a, b, f, C_G) > 0$ such that for all $M > M_0$ we have that 
    \[
        \dd_{W \land M}(p, q) = \dd_{W}(p, q),
    \]
    for all $p, q \in B_R(0)$, where 
    \[
        \dd_{W \land M}(p, q) \coloneqq  \inf_{r \in \R^N} ( \dd_{W_0 \land M}(p, r) + \dd_{W_1 \land M}(r, q) ).
    \]
\end{corollary}
\begin{proof}
    We have
    \[
        \dd_{W \land M}(p, q) \leq \dd_{W}(p, q),
    \]
    for all $p,q \in B_R (0)$.
    Now, let $K_0 > 0$ be the constant given in Proposition \ref{prop:infimum_attained_in_ball} and $K_1$ be the constant of Lemma \ref{lemma:eps_minimizers_are_bounded} with respect to the parameters $K_0$, $|a|$, $|b|$, $f$ and $C_G$. Set 
    $$ M_0 \coloneqq \max \left\{\|W_0 \|_{L^\infty(B_{K_1}(0))}, \|W_1 \|_{L^\infty(B_{K_1}(0))} \right\}.$$ 
    Let $M > M_0$. We observe first that Proposition \ref{prop:infimum_attained_in_ball} coupled with Remark \ref{remark:geodesic_parameter_discussion_2} applied to $\widetilde f \coloneqq f \land M$ gives $r_0 \in B_{K_0}(0)$ such that  
    \[
         \dd_W(p,q) = \dd_{W_0\land M}(p, r_0) + \dd_{W_1 \land M}(r_0, q).
    \]    
    By Lemma \ref{lemma:eps_minimizers_are_bounded} together with Remark \ref{remark:geodesic_parameter_discussion_2} and applied to $\widetilde f \coloneqq f\land M$, for any $\eps \in (0,1)$ we can find curves $\varphi_0, \, \varphi_1 \in W^{1,1}([-1,1], \R^M)$ which are $\eps$-minimizers of $\dd_{W_0 \land M}(p, r_0)$ (resp. $\dd_{W_1 \land M}(r_0, q)$) such that they are bounded in $L^\infty$ by $K_1$.  In particular, we infer 
    \begin{align*}
        \dd_W(p,q) &\leq \int_{-1}^1 W_0(\varphi_0)|\varphi_0'| \, ds + \int_{-1}^1 W_1(\varphi_1)|\varphi_1'| \, ds  \\
        &=  \int_{-1}^1 (W_0\land M)(\varphi_0)|\varphi_0'| \, ds + \int_{-1}^1 (W_1\land M)(\varphi_1)|\varphi_1'| \, ds \\
        &= \dd_{W_0}(p, r_0) + \dd_{W_1}(r_0, q) + {2\eps} \\
        &= \dd_{W \land M}(p, q) + {2\eps},
    \end{align*}
    which concludes the proof {since $\eps>0$ is arbitrary}.
\end{proof}

{
\begin{remark}
    Lastly we comment on the connection of \ref{H1}--\ref{H8} to \ref{G1}--\ref{G3}. We first note that we always consider geodesic distances with respect to a frozen space variable. In fact, given a potential $W_i$ with $i = 1, \dots, k$ from \ref{H6} the geodesic distance we consider is given for a fixed point $x_0 \in \Omega$ by 
    \[
        \dd_{W_i}(x_0;p,q) = \inf\left\{ \int_{-1}^1 2\sqrt{W_i(x_0,\gamma(t))}|\gamma'(t)| \,:\, \gamma\in W^{1,1}((-1,1);\R^M), \gamma(-1)=p, \gamma(1)=q  \right\}.
    \]
    We note that with $f$ given by \ref{H4}, we naturally have
    \[
        \frac{2}{\sqrt{C_1}}\sqrt{f(\min\{ |u -a(x_0)|, |u - b(x_0)| \})} \leq 2\sqrt{W_i(x_0, u)} \leq 2\sqrt{C_1}\sqrt{f(\min\{ |u - a(x_0)|, |u - b(x_0)| \})}
    \]
    for all $u \in \R^M$, so \ref{G1} and \ref{G3} are fulfilled with $W_0 \coloneqq  W(x_0, \cdot)$ and $\sqrt f$. If \ref{H8} holds we also observe 
    \[
        \sqrt{f(2t)} \leq \sqrt C \sqrt{f(t)}
    \]
    for all $t \in [0,\infty)$ which also implies that the doubling condition holds for $\sqrt f$, i.e., \ref{G2}.     
\end{remark}
}


{
\subsection{An approximation result for sets of finite perimeter}

{In this section, we present an approximation result for sets of finite perimeter with piecewise $C^2$ sets, both in configuration and in energy.
The peculiarity of such an approximation is that the density of the energy jumps on a piecewise $C^2$ surface.
Therefore, the smooth approximating sequence needs to preserve the parts of the original set on such boundaries, as much as possible.
A similar result was obtained in \cite[Lemma 5.25]{CriFonGan_moving_sub} in the case where the jumps of the energy were on sets with polyhedral boundary. Recall the definition of $\mathcal F_\infty$ in Definition \ref{def:limiting_functional} and
\[
    A = \Omega \cap \bigcup_{i = 1}^k  \partial \Omega_i.
\]
}

{\begin{proposition}\label{prop:approx_sets}
Suppose that \ref{H1}--\ref{H6} and \ref{H8} holds and let $\kappa > 0$. Let $E\subset\Omega$ be a set with finite perimeter in $\Omega$. 

Then, there exists an open set $F \subset \Omega$ with piecewise $C^2$ boundary in the sense of Definition \ref{def:piecewei_C2_boundary}, such that:
\begin{enumerate}
    \item For 
    \[
         u = a \ca_{F} + b (1-\ca_{F}), \qquad \qquad v = a \ca_E + b (1-\ca_E)
    \]
    we have
    \[
    |\mathcal{F}_\infty(u) - \mathcal{F}_\infty(v)|< \kappa,\quad\quad\quad  \| u - v \|_{L^1(\Omega)}< \kappa;
    \]
    \item We have
    \begin{align}\label{eq:approximation_lemma_property_3}
        \hno ( \partial_A (\partial F \cap A)) = 0,
    \end{align}
    where $\partial_A$ is the relative boundary in $A$.
\end{enumerate}
\end{proposition}}

\begin{proof}
The idea is the following (see Figure \ref{fig:approx_lemma}). First of all, by using Theorem \ref{thm:strong_BV}, in Step 0 we reduce ourselves to the case of a set $E$ whose boundary is contained in a finite union of $C^1$ hypersurfaces.

Then, after having isolated the singularities of the set $A$ in Step 1, we devote Step 2 to taking care of the contact set $\partial E\cap A$: in order to preserve it during the approximation process, we cover it with a nice relatively open set $D\subset A$. In Step 3 we build the required approximation $F$ as the union of two sets: the first is a smooth approximation of $E$ in $\Omega$. The second is the normal extension of the set $D$ with respect to the normals of $A$ far from the singularities.
Taking such a union ensures that the part of the boundary of $F$ on $A$ far from the singularities is exactly the set $D$, which we have a control over.
Finally, in Step 4 we verify that such an approximation satisfies the required properties.\\

\input{fig_steps_approx}

\textbf{Step 0: Reduction to finite union of $C^1$ hypersurfaces.}
We claim that it is sufficient to prove the result for piecewise $C^1$ sets.
Indeed, fix $\mu>0$. We first show that there exists a set $L\subset\Omega$ with piecewise $C^1$ boundary such that for 
\[
    \widetilde u \coloneqq a \ca_{L} + b (1-\ca_{L})
\]
we have
\[
\| \widetilde u - v \|_{L^1(\Omega)}< \mu\quad\quad\quad |\mathcal{F}_\infty(\widetilde u) - \mathcal{F}_\infty(v)|< \mu,
\]
Here, $E^{(1)}$ denotes set of points density one with respect to $E$.

Indeed, for fixed $\varepsilon>0$, let $L\subset\Omega$ be the set given by the strong approximation result for sets of finite perimeter by de Gromard (see Theorem \ref{thm:strong_BV}).
This gives us that
\[
\| \widetilde u - v \|_{L^1(\Omega)}< \eps.
\]
We now estimate the energy. We have that 
\begin{align*}
|\mathcal{F}_\infty(\widetilde u) - \mathcal{F}_\infty(v)|
&\leq
    C\Bigl[ \H^{N-1}((\partial^*E \cap \partial^* L) \cap \{\nu_E(x) \neq \nu_L(x)\} \cap \Omega) + \H^{N-1}((\partial^*E \triangle \partial^* L )\cap \Omega) \\
    & \qquad + \H^{N-1}((E^{(1)}\triangle L) \cap A ) \Bigr],
\end{align*}
where we used the fact that the function $\dd_W$ is bounded. Now, using the compact set $K \subset \Omega$ provided by Theorem \ref{thm:strong_BV}, we observe 
\[
    \H^{N-1}((\partial^*E \triangle \partial^* L )\cap \Omega) \leq |D\ca_E|(\Omega\setminus K ) + \hno\left(\partial^* L \setminus \partial^*E \right) \leq 2\eps
\]
and 
\[
    \H^{N-1}((\partial^*E \cap \partial^* L) \cap \{\nu_E(x) \neq \nu_L(x)\} \cap \Omega) \leq |D\ca_E|(\Omega\setminus K ) \leq \eps.
\]
Here, we used property (iii) and (vii) of Theorem \ref{thm:strong_BV}.
To estimate the last term. Using \ref{H6} and the continuity of traces in $BV$ with respect to strong $BV$ convergence (see (ii) of Theorem \ref{thm:strong_BV}), we have $BV(\Omega) \hookrightarrow L^1(T; \H^{N-1})$ so we get
\[
\H^{N-1}((E^{(1)} \triangle L^{(1)}) \cap A) \leq C\eps.
\]
Finally, using (viii) of Theorem \ref{thm:strong_BV}, we obtain
\[
\H^{N-1}((E^{(1)} \triangle L) \cap A) \leq C\eps.
\]
Therefore, choosing $\varepsilon>0$ small enough, we obtain the desired claim.

Furthermore, we can also assume that
\begin{equation}\label{eq:nice_boundary}
\hno( \overline{\partial L}\cap \partial\Omega ) = 0.
\end{equation}
Indeed, we can modify $L$ around $\partial\Omega$ in order to preserve the above properties and obtain \eqref{eq:nice_boundary} as follows.
Let $V\colon\R^N\to\R^N$ be a $C^1$ vector fields with $\|V\|_{C^1}<\infty$ and $V(x)\neq 0$ for all $x\in\partial\Omega$. In addition, we can assume that there exists $J>0$, depending on the Lipschitz constants of the local graph representation of $\partial\Omega$, with the following property (compare with \cite{HofmannMitreaTaylor07}) 
\begin{equation}\label{eq:V_transversal}
    V(x)\cdot \nu_{\partial\Omega}(x)> J
\end{equation}
for $\hno$-a.e. $x\in\partial\Omega$, where $\nu_{\partial\Omega}(x)$ is the outer normal to $\partial\Omega$, at a point $x\in\partial\Omega$ where it exists. 
Now, fix $\eta>0$. Let $\varphi_\eta\colon\R^N\to[0,1]$ be a $C^1$ function such that $\varphi_\eta\equiv 0$ on $\Omega_\eta\coloneqq\{x\in\Omega : \mathrm{dist}(x,\partial\Omega)>\eta\}$, $\varphi_\eta\equiv1$ in $\{x \in \R^N: \mathrm{dist}(x, \partial \Omega) < \eta/2 \}$ and $\varphi_\eta\neq 0$ else.
Define $\Phi_\eta\colon\R^N\to\R^N$ as $\Phi_\eta(x)\coloneqq x +\eta\varphi_\eta(x) V(x)$.
Consider the set $\Phi_\eta(L)\cap \Omega$. Then, it is a set of finite perimeter whose boundary is contained in a finite union of $C^1$ hypersurfaces, and such that \eqref{eq:C1_small} holds.
Note that
\[
\lim_{\eta\to0}\| \ca_{L} - \ca_{\Phi_{\eta}(L)} \|_{BV(\Omega)} = 0.
\]
Moreover, using \eqref{eq:V_transversal} it is also possible to choose a vanishing sequence $\{\eta_n\}_n$ such that
\[
\hno(\partial\Phi_{\eta_n}(L)\cap\partial \Omega)=0,
\]
for all $n\in\N$. This gives the desired claim.

Finally, using a diagonal argument, we can thus reduce to constructing the required approximate set in the case the set $E$ has boundary contained in a finite union of $C^1$ surfaces, \eqref{eq:C1_small} holds, and \eqref{eq:nice_boundary} holds.
Therefore, in the following, we will assume that $E$ satisfies these properties.\\

\textbf{Step 1: Isolation of singularities.}

We claim that there exists an open set $S\subset \Omega$ such that the following properties hold:
\begin{enumerate}
\item $S$ has piecewise $C^2$ boundary and any component of $(A \setminus S)$ can be locally written as a $C^2$ graph;
\item $(\partial E \setminus S) \cap \Omega$ is contained in a pairwise disjoint union of finitely many $C^1$ surfaces compactly contained in $\Omega$;
\item $\H^{N-1}(\partial S \cap (\partial E \cup A)) = 0;$
\item $\hno(\partial E\cap \overline{S})\leq \mu$; 
\item $\hno(A \cap \overline S) \leq \mu$; 
\end{enumerate}

Let $S_{\partial E}$ and $S_A$ be the sets of singular points of $\partial E$ and $A$ in $\Omega$, respectively.
Now,  we set 
\[
    \widetilde \Sigma \coloneqq (\overline{S_{\partial E}} \cup \overline{S_A}) \cap \overline{\Omega}.
\]
Recall that Theorem \ref{thm:strong_BV}, \eqref{eq:C1_small}, and \ref{H6} imply 
\[
    \hno(\widetilde \Sigma) < \eps.
\]
Also, $\widetilde \Sigma$ is compact.
Define now
\[
    \Sigma_1 \coloneqq \widetilde \Sigma \cup (\partial \Omega \cap \overline{A}) \cup (\partial \Omega \cap \overline{\partial E}).
\]
Note that $\Sigma_1$ is compact, and by \eqref{eq:nice_boundary} and since $\hno(\overline{S_A}\cap\partial\Omega)=0$, it holds that
\[
\hno(\Sigma_1) = \hno(\widetilde \Sigma) < \eps.
\]
Thus, for any $r>0$ sufficiently small, it is possible to cover $\Sigma_1$ with a finite union of balls of radius $r$. We denote this finite union by $S$. Then, for $S$, requirement (1) and (2) above holds directly from our assumptions on $A$ and the regularity of $E$, respectively. Moreover, (3) follows from a standard perturbation argument on the radii $r$, while (4) and (5) follow if $r>0$ and $\eps>0$ are chosen small enough.\\


\textbf{Step 2: Approximation of the contact set.}
Fix $\alpha>0$.
We claim that there exists a set $D\subset A$ with the following properties:
\begin{itemize}
\item[(1)] $D$ can be locally written as a $C^2$ graph;
\item[(2)] It holds that
\[
\hno \left( (A \setminus D) \cap \partial L \right) \leq\alpha;
\]
\item[(3)] There exists an open set $O\subset \Omega$ such that $D = O\cap A$ and
\[
    \hno(\partial O \cap A) = 0.
\]
\end{itemize}

In particular, note that from $D\subset A$ and (v) of Step 1, we get that
\begin{equation}\label{eq:est_D_S}
\H^{N-1}(S \cap D) \leq \alpha.
\end{equation}

In order to prove the existence of the set $D\subset A$, we reason as follows.
Note that $A\setminus S$ is a finite union of pairwise disjoint $C^2$ hypersurfaces whose mutual distance is strictly positive.
For $x\in A\setminus S$, let $i < j\in\{1,\dots, k\}$ be such that $x\in \partial\Omega_i\cap\partial\Omega_j$.
Define the normal $\nu_A(x)\coloneqq \nu_{\partial\Omega_i}(x)$.
We first notice that the disjoint sets
\[
D^+\coloneqq \{ x\in (\partial E \cap A)\setminus S  : \nu_E(x)=\nu_A(x) \},\quad
D^-\coloneqq \{ x\in  (\partial E \cap A)\setminus S : \nu_E(x)=-\nu_A(x) \},
\]
where $\nu_E$ denotes the outer unit normal to $E$, are compact due to the continuity of the normals.
Indeed, note that the sets $\partial E$ and $A$ are relatively closed in $\Omega$; thus from the definition of $S$ in Step 1, we conclude.
Moreover, we have that
\[
\mathrm{dist} (D^+, D^-) > 0.
\]
Using the outer regularity of the Radon measure $\hno \restr A$, and using the $C^2$ regularity of $A\setminus S_A$, it is possible to find pairwise disjoint smooth open sets $U, V \Subset \Omega \setminus \overline{S_A}$, such that, setting
\[
P\coloneqq U\cap A\quad\quad N\coloneqq V\cap A,
\]
the following hold: 
\[
    D^+\subset P, \qquad D^-\subset N,
\]
that $\partial_A P$ and $\partial_A N$ are $C^2$ embedded manifolds with
\begin{equation}\label{eq:boundary_Pi_Ci}
\mathcal{H}^{N-2}(\partial_A P)<\infty, \quad\quad\quad
\mathcal{H}^{N-2}(\partial_A N)<\infty,
\end{equation}
where $\partial_A$ is the relative boundary in $A$,
\begin{align*}
    \text{dist}(P, S_A) + \text{dist}(N, S_A) > 0,
\end{align*}
and
\begin{equation}\label{eq:C_close}
\hno(P \setminus D^+) + \hno(N \setminus D^-) < \alpha.
\end{equation}

We define
\[
D\coloneqq P \cup N.
\]
By its definition and by using \eqref{eq:C_close}, we see that the set $D$ satisfies the required properties.
Note that, since $\partial E \setminus S$ and $A \setminus S$ are $C^1$ and $C^2$ hypersurfaces respectively, the transverse intersection 
$$\{x \in \partial E \cap (A \setminus (S\cup D)) : \nu_E(x) \neq \pm \nu_A(x)\}$$
is an $(N-2)$-dimensional $C^1$ submanifold, and hence has $\H^{N-1}$-measure zero.
Therefore, we have
\begin{equation}\label{eq:transverse_null}
\hno\bigl(\partial E \cap (A \setminus (S \cup D))\bigr) = 0,
\end{equation}
as required.\\


\textbf{Step 3: The approximating set.}
We are now in a position to define the approximating set.
For $s_0>0$ with $\alpha < s_0$, define
\[
    P(s_0) \coloneqq  \{x -  s\nu_A(x) : x\in P, s \in (0, s_0] \} \quad \text{ and }\quad
    N(s_0) \coloneqq \{x + s\nu_A(x) : x\in N, s \in (0, s_0] \},
\]
where $P, N\subset A$ are the sets defined in Step 2.
Note that, thanks to the Tubular Neighborhood Theorem, it is possible to choose $s_0>0$ small enough so that $P(s_0)$ and $N(s_0)$ have piecewise $C^2$ boundaries, and each point in $P(s_0)$ and $N(s_0)$ has a unique normal projection on $P$ and $N$, respectively.
Moreover, \eqref{eq:boundary_Pi_Ci} ensures that they have finite perimeter.
Set
\begin{equation}\label{eq:R}
    R\coloneqq \left(P(s_0)  \cup N(s_0) \right) \setminus \overline{S}.
\end{equation}
By choosing $s_0>0$ sufficiently small, it is possible to ensure that
\[
P(s_0), N(s_0) \quad\text{ pairwise disjoint },\quad\quad\quad
\overline{P(s_0)}\subset\Omega, \quad\quad\quad
\overline{N(s_0)}\subset\Omega,
\]
that
\begin{equation}\label{eq:zero_contact_1}
    \hno( (\partial P(s_0) \cup \partial N(s_0)) \cap \partial S) = 0, \qquad s_0 \bigl( \mathcal{H}^{N-2}(\partial P) + \mathcal{H}^{N-2}(\partial N)\bigr) < \alpha,
\end{equation}
and that
\begin{equation}\label{eq:zero_contact_3}
    \hno\bigl(\overline{R} \cap (\partial E \setminus A)\bigr) < \alpha.
\end{equation}

We now claim that there exists a smooth set $G\subset\Omega$ satisfying the following properties
\begin{enumerate}
    \item The Hausdorff distance between $E$ and $G$ is smaller than $\alpha$ and
    \[
    \partial G \cap \Omega \subset \{x \in \Omega : \mathrm{dist}(x, \partial E \cap \Omega) \leq \alpha\};
    \]
    \item It holds that
    \begin{align}\label{eq:L_alpha_property_zero}
    \|\ca_{G} - \ca_{E}\|_{L^1(\Omega)}
        + \left| |D \ca_{G}|(\Omega) -
                |D \ca_{E}|(\Omega) \right|\leq \alpha;
    \end{align}
    \item We have that $G\subset E$ and $\hno(\partial \Omega \cap \overline{E\setminus G})<\alpha$;
    \item The following
    \begin{align}\label{eq:L_alpha_property_two}
    \hno\left(\partial G \cap \left( A \cup \partial S \cup \partial R \cup \partial\Omega\right) \right) = 0
    \end{align}
    holds;
    \item It holds that
    \begin{align}\label{eq:approx_est_1}
    | \mathcal{F}_\infty(u; B) - \mathcal{F}_\infty(v; B) | < \frac{\kappa}{4}.
    \end{align}
    where $B \coloneqq \Omega \setminus (\overline{S} \cup \overline{R} \cup A)$.
\end{enumerate}

Indeed, extending $E$ across $\partial \Omega$ to a set of finite perimeter $\widetilde E \subset \R^N$ such that $|D\ca_{\widetilde E}|(\partial \Omega) = 0$.
Then, using the interior approximation result of \cite{gui2023smooth}, we get that there exists a set $\widetilde G\subset \widetilde E$ with
\[
\left| |D \ca_{\widetilde{G}}|(\Omega) - |D \ca_{E}|(\Omega) \right| \leq C \hno(\partial E\setminus \partial^* E).
\]
From (viii) of Theorem \ref{thm:strong_BV}, we can control the right-hand side of the above equation.
Thus, by restricting that approximation to $\Omega$, we obtain a set $G$ satisfying properties (1), (2), (3), and (4).
Since $\hno(\partial B\cap \partial E\cap\Omega)=0$, thanks to Portmanteau Theorem, since $\dd_W$ is continuous in $B$ (cf.\ Remark \ref{remark:continuity_of_geodesic_distance}), we can also assume the last property above.\\

We are finally in position to define our approximating piecewise $C^2$ open set
\[
F \coloneqq G \cup P(s_0) \cup N(s_0) .
\]
By construction, $F$ is piecewise of class $C^2$ and
\[
\|u-v\|_{L^1(\Omega)} \leq C(\alpha+s_0).
\]
Thus, choosing $\alpha$ and $s_0$ small enough ensures the closeness in $L^1$ norm.
Moreover, the fact that $P(s_0)$ and $N(s_0)$ are compactly contained in $\Omega$, together with \eqref{eq:L_alpha_property_two} ensures the validity of \eqref{eq:approximation_lemma_property_3}.
Now, from \eqref{eq:L_alpha_property_two}, we get that
\begin{equation}\label{eq:boundary_F}
\hno \left( (\partial F\cap A) \triangle (P\cup N) \right)=0.
\end{equation}
Finally, by eventually further reducing the values of $\alpha$ and $s_0$, we can also ensure that
\begin{align}\label{eq:F_density_1_estimate}
    \hno( (F \triangle E) \cap A ) \leq 3\alpha,
\end{align}
since $E$ is a set of finite perimeter in $\Omega$.\\

\textbf{Step 4: Energy estimates.}
We now show that the set $F$ defined above satisfies the desired energy estimate. 
Set
\[
u \coloneqq a \ca_{F} + b (1-\ca_{F}), \qquad \qquad v \coloneqq a \ca_E + b (1-\ca_E).
\]
Observe that 
\[
    J_u = \partial F \cup A.
\]
We introduce the localization of the energy: given a Borel set $Z\subset \Omega$, we set
\[
\mathcal{F}_\infty(u;Z) \coloneqq \int_{Z\cap J_u} \dd_W(x;u^+(x), u^-(x)) \dhno(x),
\]
and similarly for the function $v$.
Write
\[
\Omega = \overline{S} \cup (A \setminus \overline{S}) \cup (\overline{R}\setminus (A \cup \overline{S})) \cup B,
\]
where $R$ is the set defined in \eqref{eq:R}, and $B = \Omega \setminus (\overline{S} \cup \overline{R} \cup A)$.
We will perform the estimate of the difference in energy in each of the sets on the right-hand side.

Recall that the energy in $B$ is estimated by (5) of Step 3.
We now perform the estimate in $\overline{R}\setminus (A \cup \overline{S})$.
By \eqref{eq:zero_contact_3}, we have
\[
    \mathcal{F}_\infty(v; \overline{R} \setminus (A \cup \overline{S})) \leq C\alpha.
\]
For $u$, note that $J_u \cap R = \emptyset$, while on $\partial R \setminus (A \cup \overline{S})$ the lateral boundary has $\hno$-measure at most $C\alpha$ by \eqref{eq:zero_contact_1}.
Thus, using this together with (3) of Step 2 and \eqref{eq:L_alpha_property_two}, we have that 
\[
| \mathcal{F}_\infty(u; \overline{R} \setminus (A \cup \overline{S})) - \mathcal{F}_\infty(v; \overline{R} \setminus (A \cup \overline{S})) | < C\alpha.
\]
By choosing $\alpha>0$ sufficiently small, we can ensure that
\begin{equation}\label{eq:approx_est_2}
| \mathcal{F}_\infty(u; \overline{R} \setminus (A \cup \overline{S})) - \mathcal{F}_\infty(v; \overline{R} \setminus (A \cup \overline{S})) | < \frac{\kappa}{4}.
\end{equation}
We now perform the estimate in $\overline{S}$. Using (4) of Step 1 together with \eqref{eq:zero_contact_1} and \eqref{eq:L_alpha_property_zero}, and by choosing $\alpha>0$ sufficiently small, we can ensure that
\begin{equation}\label{eq:approx_est_3}
| \mathcal{F}_\infty(u; \overline{S}) - \mathcal{F}_\infty(v; \overline{S}) | < \frac{\kappa}{4}.
\end{equation}
Finally, we estimate the difference in energy in $A \setminus \overline{S}$.
Thanks to (5) of Step 1, (2) of Step 2, \eqref{eq:est_D_S}, \eqref{eq:C_close}, \eqref{eq:boundary_F}, and \eqref{eq:F_density_1_estimate} we can choose $\alpha>0$ small enough, so that
\begin{equation}\label{eq:approx_est_4}
| \mathcal{F}_\infty(u; A \setminus \overline{S}) - \mathcal{F}_\infty(v; A \setminus \overline{S}) | < \frac{\kappa}{4}.
\end{equation} 
Thus, from \eqref{eq:approx_est_1}, \eqref{eq:approx_est_2}, \eqref{eq:approx_est_3}, and \eqref{eq:approx_est_4} we obtain the desired result.
\end{proof}
}


\section{Compactness}

This section is devoted to investigating the compactness of sequences with uniformly bounded energy. {We start by refining the idea of a local adjustment (recall Definition \ref{def:local_adjustment}) to the situation of $BV$ wells. Indeed, our strategy to prove compactness hinges on being able to reduce \emph{globally} to the situation of a potential with fixed wells. Therefore, the notion of \emph{local} adjustment introduced in Definition \ref{def:local_adjustment} is not enough for the argument provided in this section.

\begin{definition}
    We say that 
    $T \colon \Omega\times\R^M\to\R^M$ is a \emph{BV-adjustment} if it is a map of the form
    \begin{equation*}
        T(x,w) \coloneqq \frac12 \big(  (a(x)+b(x)) + (b(x)-a(x), v(x))\cdot w \big),
    \end{equation*}
    where $v \in BV(\Omega; \R^{M \times (M-1)})$ is such that:
    \begin{itemize}
        \item[(i)] For each $x\in \Omega$, the matrix
        \[
            \frac{1}{|b(x)-a(x)|}\left(b(x)-a(x), v(x)\right)
        \]
        is orthonormal;
        \item[(ii)] There exists a Caccioppoli partition $U_1, \dots, U_{M} \subset \Omega$ of $\Omega$ such that the map $x\mapsto v(x)$ fulfils 
        \begin{align*}
            J_v \subset J_{b - a} \cup \bigcup_{l = 1}^{M} \partial^* U_l.
        \end{align*}
    \end{itemize}
\end{definition}

Next, we will show that for $BV$ wells such an adjustment always exists. First, recall a basic fact from linear algebra.

\begin{lemma}\label{lemma:linear_algebra}
    Suppose that $c \in \R^M$ such that $|c| = 1$ and $|c_1| > t$. Then for every $j = 2, \dots, N$ we have
    \[
        \dist(e_j, \spann(c,e_2, \dots, e_{j-1})) \geq \frac{t}{2}.   
    \]
\end{lemma}

\begin{proof}
    Let $j = 2, \dots, N$ and 
    \[
        v \coloneqq \alpha c + \sum_{i = 2}^{j-1} \lambda_i e_i \in \spann(c,e_2, \dots, e_{j-1})
    \]
    for scalars $\alpha, \lambda_i \in \R$. Then, 
    \[
        |v - e_j| \geq \min\{ |\alpha c_j - 1|, |\alpha c_1| \} \geq \min\{|1 - |\alpha||, t|\alpha|\} \geq t/2.
    \]
\end{proof}

Now, we show that we can always extend a $BV$ function with values in $\mathbb{S}^{M-1}$ to an orthonormal matrix field, provided one component is bounded away uniformly from $0$ almost everywhere in such a way that we have a quantitative control on the derivative.

\begin{lemma}\label{lemma:gram_schmidt_application}
    Let $c \in BV(\Omega; \R^M) \cap L^\infty(\Omega; \R^M)$ such that $|c| = 1$ and $|c_l| > t$ a.e. for some $t > 0$ and some $l = 1, \dots, M$. Then, there exist a constant $C>0$ and vector fields $v^1, \dots, v^{M-1} \in BV(\Omega; \R^M)$ such that $c, v^1, \dots, v^{M-1} $ form an orthonormal basis of $\R^M$, and such that 
    \begin{align}\label{eq:prop_compactness}
        |\nabla v^i| \leq C|\nabla c|, \qquad |D^j v^i| \leq C|D^j c|, \qquad |D^c v^i| \leq C|D^c c|,
    \end{align}
    holds pointwise a.e. and for all $i = 1, \dots, M-1$.
\end{lemma}

\begin{proof}
    Without loss of generality, we assume $l = 1$.
    We apply the Gram-Schmidt procedure. First, we define 
    \[
        \widetilde v^1 \coloneqq  e_2 - (c \cdot e_2) c = e_2 - c_2 c.
    \]
    Observe that $\widetilde v^1 \in BV(\Omega; \R^M) \cap L^\infty(\Omega; \R^M)$ by assumption, and we have \eqref{eq:prop_compactness} for $\widetilde v^1$ by the chain rule for $BV$ functions (see \cite[Theorem 3.101]{AFP}). Moreover, by Lemma \ref{lemma:linear_algebra}, we have for almost every $x \in \Omega$
    \[
        |e_2 - c_2(x) c(x)| \geq \frac{t}{2}.
    \]
    In particular, by the chain rule we infer $|e_2 - c_2c|^{-1} \in BV(\Omega) \cap L^\infty(\Omega)$ which consequently yields 
    \[
        v^1 \coloneqq  \frac{\widetilde v^1}{|\widetilde v^1|} \in BV(\Omega; \R^M) \cap L^\infty(\Omega; \R^M).
    \]
    Note, again by the chain rule, we also obtain \eqref{eq:prop_compactness}. Now, defining recursively 
    \[
        \widetilde v^j \coloneqq  e_j - \sum_{i = 1}^{j-1} v^i_j v^i - c_j c,
        \quad\quad\quad
        v_j \coloneqq  \frac{\widetilde v^j}{|\widetilde v^j|},
    \]
    and repeating the above arguments in an induction argument, we conclude.
\end{proof}

Next, we prove the existence of a BV-adjustment.

\begin{lemma}\label{lemma:BV-adjustment_construction}
    Let $a, b\in BV(\Omega; \R^M) \cap L^\infty(\Omega; \R^M)$ such \ref{H1}--\ref{H5} hold. Then, there exist $C>0$, dependent only on $\| a \|_\infty,  \| b \|_\infty$ and $\delta$ from \ref{H3}, and a $BV$-adjustment such that the following properties hold:
    \begin{itemize}
        \item[(i)] $|\nabla v| \leq C|\nabla (b -a)|$ almost everywhere;
        \item[(ii)] There exists a Caccioppoli partition of $\Omega$ such that 
        \[
        |D^j v|(\Omega) \leq C|D^j (b -a)|(\Omega) + \sum_{l = 1}^M\H^{N-1}(\partial^* U_l);
        \]
        \item[(iii)] $|D^c v|(\Omega) \leq C|D^c (b -a)|(\Omega)$.
    \end{itemize}
\end{lemma}

\begin{proof}
    Set 
    \begin{align*}
        c \coloneqq \frac{b - a}{|b-a|}
    \end{align*}
    which is in $BV(\Omega;\R^M)$ by \ref{H3}.
    Using the co-area formula, for $l = 1, \dots, M$, we define 
    \[
        \widetilde U_l \coloneqq \{ |c_l| > t_l \}
    \]
    for some 
    \[
        t_l \in \left( \frac{1}{M}, \frac{1}{2M} \right),
    \]
    such that $\widetilde U_l$ is a set of finite perimeter. Observe that 
    \[
        \Omega \subset \bigcup_{l = 1}^{M}\widetilde U_l
    \]
    by our choice of $t_l$. Now, we set 
    \[
        U_l \coloneqq \widetilde U_l \setminus \bigcup_{i = 1}^{l-1} \widetilde U_i.
    \]
    Note that $\{U_l\}_l$ is a Caccioppoli partition of $\Omega$. Next, let 
    \[
         \widetilde c_l \coloneqq  c\ca_{U_l} + (1 - \ca_{U_l})e_l.
    \]
    Note that 
    \begin{align}\label{eq:comp_prop_2}
        |D^j \widetilde c_l|(\Omega) \leq |D^j c|(\Omega) + \H^{N-1}(\partial^* U_l).
    \end{align}
    Since $|\widetilde c_l| > t_l$, we apply Lemma \ref{lemma:gram_schmidt_application} to the map to get  a matrix field 
    \[
        v_l \in BV(\Omega; \R^{M\times (M-1)}) \cap L^\infty(\Omega; \R^{M\times (M-1)})
    \]
    fulfilling \eqref{eq:prop_compactness} such that the matrix given by $(\widetilde c_l, v_l)$ is orthonormal almost everywhere. We also observe that $(\widetilde c_l(x), v_l(x))$ is the identity matrix $Id_M$ for $x \in \Omega \setminus U_l$ by the construction in the proof of Lemma \eqref{lemma:gram_schmidt_application}. Now, setting 
    \[
        v \coloneqq |b-a| \left( \sum_{l = 1}^M \widetilde v_l \ca_{U_l} \right)
    \]
    we conclude the proof since $v$ inherits the properties \eqref{eq:prop_compactness} from $\widetilde v_l$ and by \eqref{eq:comp_prop_2}.
\end{proof}

For completeness, before the proof of the compactness result, we will now also state  that local adjustments in the sense of Definition \ref{def:local_adjustment} always exist. 
\begin{proposition}\label{prop:local_adjustments_existence}
    Let $a_i,b_i \in W^{1,2}(\Omega; \R^M)\cap C^0(\overline{\Omega}; \R^M)$ be the vector fields given by \ref{H6}. Then, there exist local adjustments $T_i$ in the sense of Definition \ref{def:local_adjustment}.
\end{proposition}
\begin{proof}
    Fix $x_0 \in \Omega$. Since 
    \[
        |c_i| = |a_i - b_i| \geq \delta
    \]
    by \ref{H3} in $\Omega$, we find an index $l_0 = 1,\dots, M$ and an open neighbourhood $U \subset \Omega$ of $x_0$ such that for all $x \in U$
    \begin{align}\label{eq:basis_local_adj}
        c_i(x), e_1, \dots, e_{l_0 - 1}, e_{l_0 + 1}, \dots, e_M 
    \end{align}
    forms a basis and that 
    \[
        |c_i(x) \cdot e_{l_0}| > \frac{\delta}{2M}
    \]
    holds.
    Applying now the Gram-Schmidt procedure to \eqref{eq:basis_local_adj}, as in the proof of Lemma \ref{lemma:BV-adjustment_construction}, yields a suitable matrix field $v_i \in W^{1,2}(\Omega; \R^{M\times (M-1)})\cap C^0(\overline{\Omega}; \R^{M\times (M-1)})$ such that the correspondingly defined adjustment $T_i$ fulfills the assumptions of Definition \ref{def:local_adjustment}.
\end{proof}}

{
Now, we prove the main result in this section.

\begin{proof}[Proof of Theorem \ref{thm:compactness}]
    Let $\{u_n\}_n\subset H^1(\Omega;\R^M)$ be such that
    \[
    \lambda \coloneqq \sup_{n\in\N} \mathcal{F}_{\varepsilon_n} <\infty.
    \]
    The main idea of this proof is to use a BV-adjustment in order to work with a potential with fixed wells. Then, we will exploit the standard strategy of compactness for such class of potentials to obtain compactness of a certain modified sequence of functions.
    With a bit of work, we are then able to leverage on such compactness to obtain the desired result.\\

    \textbf{Step 1: Properties of the sequence.}
     We start with showing that $\{u_n\}_n$ is uniformly bounded in $L^1$ and equiintegrable. Indeed, using \ref{H5}, we choose a constant $C_1 > C_2$ such that 
    \begin{equation}\label{eq:un_equiintegrable}
        \int_{\{|u_n| > C_1\}} |u_n| \, \mathrm{d}x \leq C_2 \int_\Omega W(x, u_n) \, \mathrm{d}x \leq C_2\lambda\eps_n,
    \end{equation}
    which implies equiintegrability. Using again \ref{H5}, we also have 
    \begin{align*}
        \int_\Omega |u_n|\, \dd x &= \int_{\{|u_n|\leq C_1\}}|u_n| \, \dd x + \int_{\{|u_n| > C_1\}}|u_n| \, \dd x \\
        &\leq C_1|\Omega| + C_2 \int_\Omega W(x, u_n) \, \dd x\\
        &\leq C_1 | \Omega| + C_2 C \eps_n,
    \end{align*}
    which implies uniform boundedness in $L^1(\Omega; \R^M)$.\\
    
    \textbf{Step 2: Truncation.}
    {We define, for $L>\max\{C_1, \|a\|_{L^\infty}, \|b\|_{L^\infty}\}$, the component-wise truncation map $\mathcal{T}_L:\R^M\to\R^M$, and set
    \[
        w_n(x) \coloneqq \mathcal{T}_L w_n(x).
    \]
    Note that $\{u_n\neq w_n\}\subset \{|u_n|>L\}$.
    Using the fact that $a, b \in L^\infty(\Omega; \R^M)$ together with \ref{H4}, we have 
    \begin{align*}
        \frac{1}{\eps_n}\int_\Omega W(v_n) \, dx
        &\leq \frac{1}{\eps_n}\int_\Omega W(u_n) + \frac{1}{\eps_n}\int_{\{u_n\neq w_n\}} W(w_n) \\
        &\leq \frac{1}{\eps_n}\int_\Omega W(u_n) + \frac{1}{\eps_n}\int_{\{|u_n|>L\}} W(w_n) \\
        & \leq \frac{1}{\eps_n}\int_\Omega W(u_n)
            + \frac{C_2}{C_1} \frac{1}{\eps_n}\int_{\{|u_n|>L\}} W(u_n) \\
        &\leq \frac{1}{\eps_n}\int_\Omega W(u_n) \, dx + C\lambda,
    \end{align*}
    where in the last step we used \eqref{eq:un_equiintegrable}.
    Furthermore, note that if $\{w_n\}_n$ converges to some $w \in L^1(\Omega; \R^M)$ then (after possibly extracting a subsequence), we also have that $\{u_n\}_n$ converges to $v$ in $L^1(\Omega; \R^M)$ due to Vitali's theorem. Therefore, it is only necessary to prove the statement for a sequence $\{u_n\}_n$ that is uniformly bounded by $L$, i.e., uniformly bounded in $L^\infty$, which we will assume for the rest of the proof.\\

    \textbf{Step 3: BV-adjustment.}
    Now, let $T$ be the $BV$-adjustment constructed in Lemma \ref{lemma:BV-adjustment_construction}.
    For each $x\in\Omega$, the affine map $T(x,\cdot)\colon\R^M\to\R^M$ is invertible, by assumption.
    Define
    \[
    v_n(x) \coloneqq T^{-1}(x, u_n(x)),
    \]
    where for each $x \in\Omega$, the map $T(x,\cdot)^{-1}$ is the inverse of the adjustment in the second variable (cf.\ also Remark \ref{remark:for_liminf}).
    The idea is now to use the standard strategy for compactness for the sequence $\{v_n\}_n$.
    Since $a, b \in L^\infty(\Omega; \R^M)$, \ref{H3}, and the assumption from Step 2 that $\{u_n\}_n$ is uniformly bounded in $L^\infty$, we have that the sequence $\{v_n\}_n$
    is uniformly bounded in $L^\infty$.
    In particular, it therefore generates a suitable Young measure $(\nu_x)_x$ (cf.\ Theorem \ref{thm:compactness_YM}), and $v_n \rightharpoonup v$ in $L^1(\Omega; \R^M)$.}
    
    We now notice that the map $T(x, \cdot)$ is affine by definition which yields that
    \begin{align}\label{eq:W_tilde_W_compare}
        W(x,u_n) &= W(x,T(x,v_n)) \nonumber \\
        & \geq C_1 f\left( \min\{ |T(x,v_n(x)) - a(x)|, |T(x,v_n(x)) - b(x)| \} \right)  \nonumber\\
        & = C_1 f\left( \min\{ |T(x,v_n(x) - e_1)|, |T(x,v_n(x) + e_1)| \} \right) \nonumber \\
        &= C_1 f\left( |b(x) - a(x)| \min\{ |v_n(x) - e_1|, |v_n(x) + e_1| \} \right) \nonumber \\
        &\geq C_1 f(\delta \min \{|v_n - e_1|, \, |v_n + e_1| \}) \nonumber \\
        &\eqqcolon \widetilde W(v_n),
    \end{align}
    where in the previous to last step we used \ref{H3} and the monotonicity of $f$.
    Since $\widetilde{W}$ is continuous, by the fundamental theorem of Young measures, we obtain
    \[
        \int_\Omega \int_{\R^M} \widetilde W(p) \dd \nu_x(p) \dd x \leq \lim_{n \to \infty} \int_\Omega \widetilde W(u_n) \, \dd x = 0.
    \]
    Since $\widetilde{W}$ is positive outside of the wells $\pm e_1$, we infer that for almost every $x \in \Omega$ we have
    \[
        \spt \nu_x \subset \{\pm e_1 \},
    \]
    which implies that there exists a measurable function $\theta\colon \Omega \to [0,1]$ such that
    \begin{align}\label{eq:formula_Young_measure}
        \nu_x = \theta(x) \delta_{e_1} + (1-\theta(x)) \delta_{-e_1},
    \end{align}
     for almost every $x \in \Omega$.\\

    \textbf{Step 4: Estimates for $\{v_n\}_n$.}
    Let $K > \max\{ \|a\|_{L^\infty}, \|b\|_{L^\infty} \}$.
    Consider the geodesic distance function
    \[
        \dd_{\widetilde W}(p,q) \coloneqq \inf \left\{ \int_{-1}^1 2 \min\left\{\sqrt{\widetilde W(\varphi)}, K \right\}|\varphi'| \, \mathrm{d}s : \varphi \in W^{1,1}([-1,1], \R^M), \, \varphi(-1) = p,\, \varphi(1) = q \right\},
    \]
    and define
    \[
        z_n(x) \coloneqq \dd_{\widetilde W}(-e_1, v_n(x)),
    \]
    as the geodesic distance from a well in the transformed space. 
    First, we claim that $z_n\in BV(\Omega;\R^M)$.
    Indeed,  thanks to Lemma \ref{lemma:geodesic_distance_is_locally_lipschitz}, this distance is globally Lipschitz continuous and its Lipschitz constant is bounded by 
    \[
        |\nabla \dd_{\widetilde W}| \leq C(1 + 2K).
    \]  
    {Since $w_n \in BV(\Omega; \R^M)$ and the geodesic distance belongs to $\mathrm{Lip}(\R^M \times \R^M; \R)$, by \cite[Proposition 3.69(c)]{AFP} we get that $z_n \in BV(\Omega; \R)$.\\
    
    Next, we want to estimate $| D z_n | (\Omega)$. In the following, we will use the notation $\nabla_q \dd_{\widetilde W}$ to denote the gradient of $\dd_{\widetilde W}$ with respect to the second set of variables.
    We use the chain rule in $BV$ (see \cite[Theorem 3.101]{AFP}) to infer
    \begin{align}\label{eq:comp_est_recall_3}
        |D z_n|(\Omega) &\leq \int_\Omega |(\nabla_q \dd_{\widetilde W})(-e_1, v_n)| |\nabla v_n| \, \mathrm{d}x + C |D^j v_n|(\Omega) + C |D^c v_n|(\Omega).
    \end{align}
    We now estimate each term on the right-hand side separately.
    For the former, observe that
    \[
        |(\nabla_q \dd_{\widetilde W})(-e_1, q)| \leq 2\sqrt{\widetilde W(q)},
    \]
    for all $q \in \R^M$.
    Moreover, Using the properties of $T$ in Lemma \ref{lemma:BV-adjustment_construction}, we get
    \[
    |\nabla v_n| \leq C(|\nabla a| + |\nabla b| + |\nabla u|)).
    \]
    Thus, using \eqref{eq:W_tilde_W_compare} and \eqref{eq:comp_est_recall_1} combined with the fact that $v_n$ is uniformly bounded in $L^\infty$ we obtain
    \begin{align}\label{eq:est_1}
        \int_\Omega 2 \sqrt{\widetilde{W}(v_n)} |\nabla v_n| \, \dd x &\leq C\left( \int_\Omega 2 \sqrt{W(x,u_n)}| \nabla u_n| \, \mathrm{d} x + \int_{\Omega} 2 \sqrt{\widetilde{W}(v_n)} (| \nabla a | + |\nabla b|) \, \mathrm{d} x \right) \nonumber \\
        &\leq C\left( \mathcal{F}_{\eps_n} (u_n) + \int_{\Omega} (| \nabla a | + |\nabla b|) \, \mathrm{d} x \right).
    \end{align}}
    For the latter two terms, using the properties of $T$ in Lemma \ref{lemma:BV-adjustment_construction}, we get
    \begin{align}\label{eq:comp_est_recall_1}
        |D^c v_n|(\Omega) \leq C (|D^c a|(\Omega) + |D^c b|(\Omega)),
    \end{align}
    and 
    \begin{align}\label{eq:comp_est_recall_2}
        |D^j v_n|(\Omega) \subset C\left( |D^j a|(\Omega) + |D^j b|(\Omega) + \bigcup_{l = 1}^M \H^{N-1}(\partial^* U_l) \right)
    \end{align}
    where $C$ depends on $\|a\|_\infty$ and $\|b\|_\infty$.
    
     Therefore, from \eqref{eq:comp_est_recall_3}, \eqref{eq:est_1}, \eqref{eq:comp_est_recall_1}, and \eqref{eq:comp_est_recall_2}, together with the fact that $\{ u_n \}_n$ is uniformly bounded in energy, we obtain that
    \[
        \sup_{n\in\N}|D z_n|(\Omega) < \infty.
    \]
    In particular, by weak star compactness, we have $z_n \overset{*}{\rightharpoonup} z$ in $BV(\Omega; \R)$ up to a subsequence (without relabelling). By the Rellich-Kondrachov $BV$ compactness theorem, we know that we can also assume strong convergence in $L^1$.\\
    
    Since $v_n$ is equiintegrable and $\dd_{\widetilde W}(-e_1, \cdot)$ is a Lipschitz function, we infer that $z_n$ is also equiintegrable. Therefore, it generates a Young measure $\eta_x$ and, by the strong convergence in $L^1$, we derive $\eta_x = \delta_{z(x)}$. \\

     \textbf{Step 5: Representation of $\nu$.}
    Now, notice that we can test with $\varphi \in C_0(\Omega)$ and $\psi \in C_0(\R)$ and {infer from fundamental theorem of Young measures, the definition of $z_n$ and \eqref{eq:formula_Young_measure}}
    \begin{align*}
        \int_{\Omega} \int_{\R} \varphi(x)\psi(y) \, \dd\eta_x(y) \, \mathrm{d}x &= \int_{\Omega} \varphi(x) \psi(z(x)) \, \mathrm{d}x \\
        &= \lim_{n \to \infty}\int_\Omega \varphi(x)\psi(z_n(x)) \, \mathrm{d}x \\
        &= \lim_{n \to \infty}\int_\Omega \varphi(x) \psi(\dd_{\widetilde W}(-e_1, v_n(x))) \dd \nu_x(p) \dd x \\
        &= \int_\Omega \int_{\R^N} \varphi(x) \psi(\dd_{\widetilde W}(-e_1, p)) \dd \nu_x(p) \dd x \\
        &= \int_\Omega \varphi(x) (\theta(x) \psi(\dd_{\widetilde W}(-e_1, e_1)) + (1-\theta(x))\psi(0)) \, \dd x.
    \end{align*}
    From the arbitrariness of $\varphi$, we infer for every $\psi \in C_0(\R)$
    \begin{align}\label{eq:duality_comp}
        \psi(z(x)) = \langle \eta_x, \psi \rangle  &=  \theta(x) \psi(\dd_{\widetilde W}(-e_1, e_1)) + (1-\theta(x))\psi(0) \nonumber \\
        &= \left\langle \theta(x)\delta_{\dd_{\widetilde W}(-e_1, e_1)} + (1 - \theta(x))\delta_0, \psi \right\rangle,
    \end{align}
    for almost every $x \in \Omega$. {Since $C_0(\R)$ is separable, we find a countable dense set $F \subset C_0(\R)$ and a negligible set $N\subset\Omega$ such that \eqref{eq:duality_comp} holds 
    for every $\psi \in F$ and for every $x \in \Omega\setminus N$}.
    Using duality, we infer 
    \[
        \delta_{z(x)} = \eta_x = \theta(x)\delta_{\dd_{\widetilde W}(-e_1, e_1)} + (1 - \theta(x))\delta_0,
    \]
    for almost every $x \in \Omega$. 
    Since $\dd_{\widetilde W}(e_1, -e_1) > 0$, we derive $\theta(x) \in \{0, 1\}$.
    Therefore,
    \[
    z(x) \in\{ 0, \dd_{\widetilde W}(-e_1, e_1) \},
    \]
    and thus
    \[
        z(x) = \theta(x)\dd_{\widetilde W}(-e_1, e_1).
    \]
    Since $z \in BV(\Omega; \R)$, in turn implies $\theta(x) = \ca_E$ for a set of finite perimeter $E\subset\Omega$. 
    
    Now, taking an arbitrary $\varphi \in C_0(\Omega)$, we derive 
    \begin{align*}
        \int_\Omega v(x) \varphi(x) \, \mathrm{d}x&= \lim_{n \to \infty} \int_\Omega v_n(x) \varphi(x) \, \mathrm{d}x  \\
        &= \int_\Omega \varphi(x) \int_{\R^N} p \, \dd \nu_x(p) \, \dd x \\
        &= \int_\Omega \varphi(x) (\ca_E(x)e_1 + (1 - \ca_E(x)) (-e_1)) \, \dd x.
    \end{align*}
    Consequently, 
    \[
        v(x)= \ca_E(x)e_1 + (1 - \ca_E(x)) (-e_1)
    \]
    for almost every $x \in \Omega$.
    {Furthermore, we have that
    \[
        \lim_{n \to \infty} \int_\Omega |v_n|^2 \, \mathrm{d}x = \int_{\Omega} \int_{\R^N}|p|^2 \, \dd\nu_x(p) \, \mathrm{d}x = |\Omega|.
    \]
    Since $\{v_n\}_n$ is uniformly bounded in $L^\infty$, it also converges weakly to some $v \in L^2$. Thanks to the reflexivity of $L^2$, since we have convergence of the norms and weak convergence of $\{v_n\}_n$ to $v$, we know that $\{v_n\}_n$ converges strongly to $v$ in $L^2$, and therefore in $L^1$. 
    
    Now, using the definition of $T$ and recalling that $a,b \in L^\infty$ we finally obtain     
    \[
        \lim_{n \to \infty} u_n = \lim_{n \to \infty} T(\cdot,v_n) = T(\cdot,v) =  \ca_Ea + (1 - \ca_E)b,
    \]
    in $L^1$, with which we conclude.}

\end{proof}
}


\section{Liminf inequality}

The aim of this section is to prove the $\liminf$ inequality of our $\Gamma$-limit result (see Theorem \ref{thm:Gamma_conv}).
We are able to prove it under a weaker assumption on the regularity of the partition of $\Omega$, compared to \ref{H6}.
    
\begin{theorem}\label{thm:liminf}
    Assume that \ref{H1}-\ref{H5}, and that \ref{H7} hold.
    Moreover, assume that there exists a Caccioppoli partition $\Omega_1,\dots,\Omega_k$ of $\Omega$, such that
    \[
        W(x,u) = \sum_{i=1}^k W_i(x,u)\ca_{\Omega_i}(x),
    \]
    \[
        a(x) = \sum_{i=1}^k a_i(x)\ca_{\Omega_i}(x),\quad\quad\quad\quad
        b(x) = \sum_{i=1}^k b_i(x)\ca_{\Omega_i}(x),
    \]
    where $W_i\in C^0(\overline{\Omega}\times\R^M)$ and $a_i, b_i\in W^{1,2}(\Omega;\R^M) \cap C^0(\overline{\Omega};\R^M)$ {such that
    \begin{align}\label{eq:difference_assumption}
        a_i \neq a_j, \quad b_i \neq b_j, \quad \text{ and } \quad a_i \neq b_j
    \end{align}
    on $\bigcup_{i=1}^k \partial^*\Omega_i$ for $i \neq j$.}

    Let $u \in L^1(\Omega;\R^M)$ and {$\{u_n\}_n \subset L^1(\Omega;\R^M)$} such that $u_n \to u$ in $L^1(\Omega;\R^M)$. Then, we have
    \[
        \mathcal{F}(u) \leq \liminf_{n \to \infty} \mathcal{F}_{\eps_n}(u_n),
    \]
    for every sequence $\eps_n \to 0$.
\end{theorem}

\begin{proof}
We divide the proof into steps. {In Step 0, we introduce the notation and set up the blow-up used in the computation of the lower bound. Next, in Step 1, we extract a well-behaved set of slices that capture most of the slices where the function $u$ jumps exactly once. In Step 2, we truncate both the sequence and the potential. The core part of the proof happens in Step 3, where we use \ref{H7} to localize the potential, namely we consider the potential only at the blow-up point $x_0$, and in Step 4, where we adapt the singular perturbation to the adjustment that one needs to apply in Step 3. Lastly, in Step 5, we conclude by applying a classical lower-semicontinuity argument along one-dimensional slices.}\\[0.5em]
    \textbf{Step 0: The setup.}
    {Let $\delta \in (0,1)$ be sufficiently small.} By passing to a subsequence (not relabeled), we may assume 
    \begin{align}\label{eq:bounded_radon}
        \lim_{n\to \infty} \mathcal{F}_{\eps_n}(u_n) = \liminf_{n \to \infty} \mathcal{F}_{\eps_n}(u_n) < \infty.
    \end{align}
    {Moreover, we can assume $\{ u_n \}_n \subset W^{1,2}(\Omega;\R^M)$, and by Theorem \ref{thm:compactness}, we also have $u \in BV (\Omega; \{a,b\})$.}
    
    We apply the blow-up method of Fonseca and M\"uller \cite{fonsecamueller1999}. Define the Radon measures 
    \[
        \mu_n(B) \coloneqq  \mathcal{F}_{\eps_n}(u_n, B)
    \]
    for measurable sets $B \subset \Omega$.
    By \eqref{eq:bounded_radon}, $\{\mu_n\}_n$ is uniformly bounded in $\mathcal{M}_b(\Omega)$, the space of Radon measures in $\Omega$. Thus, there exists $\mu \in \mathcal{M}_b(\Omega)$ such that $\mu_n \rightharpoonup^* \mu$ in $\mathcal{M}_b(\Omega)$ (after extracting a subsequence). Consider now $\lambda \coloneqq  \H^{N-1} {\llcorner} J_u$ and observe that, for $\H^{N-1}$-a.e. $x_0 \in J_u$,
    \[
        \frac{d\mu}{d\lambda}(x_0) < \infty \qquad \text{and} \qquad \frac{d\H^{N-1}}{d\lambda}(x_0) = 1.
    \]
    By the Besicovitch differentiation theorem, we also know that for $\H^{N-1}$-a.e. $x_0 \in J_u$ we have 
    \[
        \frac{d\mu}{d\lambda}(x_0) = \lim_{\rho \to 0^+} \frac{\mu(Q(x_0, \rho))}{\lambda(Q(x_0, \rho))} = \lim_{\rho \to 0^+} \frac{\mu(Q(x_0, \rho))}{\rho^{N-1}},
    \]
    where $Q(x_0, \rho)$ are cubes with side length $\rho > 0$ with two faces orthogonal to $\nu(x_0)$. Since the set of $\rho$ such that $\mu(\partial Q(x_0, \rho)) > 0$ is at most countable, it follows that
    \[
        \frac{\mu(Q(x_0, \rho))}{\rho^{N-1}} = \lim_{n \to \infty} \frac{\mu_n(Q(x_0, \rho))}{\rho^{N-1}},
    \]
    for $\rho \in E$ where $\H^{1}((0,\infty)\setminus E) = 0$. {We may now reduce to such a generic point $x_0 \in J_u$ such that the above assertions hold and such that the approximate tangent plane exists}.
    Without loss of generality, we assume $x_0 = 0$, $\nu(x_0) = e_N$. For $\rho > 0$, we set
    \begin{align*}
        &Q(\rho) \coloneqq  (-\rho/2, \rho/2)^N, \quad {Q^+(\rho)\coloneqq  (-\rho/2, \rho/2)^{N-1}\times (0,\rho/2)},   \\
        &Q^-(\rho) \coloneqq  {(-\rho/2, \rho/2)^{N-1}\times(-\rho/2,0)}, \quad \text{and} \quad Q'(\rho) \coloneqq  (-\rho/2, \rho/2)^{N-1}.
    \end{align*}
    We also set $Q \coloneqq  Q(1)$ and $Q' \coloneqq  Q'(1)$. Now, by applying Lemma \ref{lemma:technical_approx_limits}, we choose a sequence $\rho_m \to 0$ such that 
    \[
        u_n \xrightarrow{n \to \infty} u
    \]
    in $L^1(\partial Q(\rho_m); \R^M)$ holds. { Let us now denote for $\rho > 0$
    \[
        u^\rho(x) \coloneqq u(\rho x).
    \]
    Then (after selecting a further subsequence)} we have
    \begin{align}\label{eq:conv_1}
        {u^{\rho_m}( \cdot , 1/2)} \xrightarrow{m \to \infty} u^+(x_0),
    \end{align}   
    and
    \begin{align}\label{eq:conv_2}
        {u^{\rho_m}(\cdot, -1/2)} \xrightarrow{m \to \infty} u^-(x_0),
    \end{align}
    in $L^1(Q'; \R^M)$.
    Without loss of generality, we can assume that
    \[
    u^+(x_0)=a_i(x_0),\quad\quad\quad
    u^-(x_0)=a_j(x_0),
    \]
    with $i\neq j$ {and $j < i$, i.e., we are in the case $x_0 \in \partial^*E_i \cap \partial^*E_j$. Recall that by our assumption we have $a_i(x_0) \neq a_j(x_0)$.}
    The case $u^+(x_0)=a_i(x_0)$ and $u^-(x_0)=b_j(x_0)$, for some $i,j\in\{1,\dots,k\}$, follows by using a similar argument.    
    
    Furthermore, by the slicing theory for $BV$-functions (cf.\ section 3.11 in \cite{AFP}) we infer that for $\H^{N-1}$-a.e. $x' \in {Q'(\rho_m)}$ we have 
    \[
        u_{x'} \coloneqq  u(x', \cdot) \in BV((-\rho_m/2,\rho_m/2), \R^M)
    \]
    and {for $\H^{N-1}$-a.e. $x' \in {Q'}$} 
    \[
        {u^{\rho_m}_{x'} \coloneqq  u^{\rho_m}(x', \cdot) \in BV((-1/2,1/2), \R^M)}
    \]
    with the additional properties that
    \[
        {\rho_mJ_{u^{\rho_m}_{x'}} =  J_{u_{\rho_m x'}} }
    \]
    and
    $$
        J_{u_{x'}} = P_N(J_u \cap (\{ x'\} \times (-\rho_m/2,\rho_m/2)))
    $$ 
    where $P_N(x) \coloneqq  x_N$.  \\[0.5em]
    
    \textbf{Step 1: Boundary conditions via slices.}
    We will now argue the existence of a large subset of $Q'({\rho_m})$ such that $u_{x'}$ is well-behaved.
    We refer to Figure \ref{fig:blowup_liminf} for the construction and to Figure \ref{fig:measure_liminf} for the estimates.\\

    \begin{figure}[t]
    \centering
    \input{fig_nice_set_liminf_blowup}
    \caption{The construction made in Step 0 and Step 1: blow-up and the choice of the good slides.}
    \label{fig:blowup_liminf}
    \end{figure}
    
    By the application of the co-area formula \cite[Theorem 2.93]{AFP} with $E = J_u$ and $f(x) = x'$, we derive
    \begin{align}\label{eq:upper_bound_jump_sets}
        \frac{1}{\rho_m^{N-1}} \int_{Q'(\rho_m)} \H^0(J_{u_{x'}}) \, \mathrm{d}x' = \frac{1}{\rho_m^{N-1}} \int_{J_u} |\nu \cdot e_n| \, \mathrm{d}x' \leq \frac{\H^{N-1}(J_u \cap Q(\rho_m))}{\rho_m^{N-1}} \xrightarrow{m \to \infty} 1.
    \end{align}
    Moreover, we notice that if $\H^0(J_{u_{x'}}) = 0$ holds for some $x' \in Q'(\rho_m)$ we have $u_{x'} = a_l(x', \cdot)$ for some $l=1,\dots,k$ by \ref{H4}. {In particular, for such a $x'$ we use the continuity of $a_l$ and \eqref{eq:difference_assumption} to obtain for $m$ sufficiently large}
    \[
        \rho_m \leq C\int_{(0, \rho_m/2)}| u_{x'} - a_i(x_0)|\, \mathrm{d}s + \int_{(-\rho_m/2, 0)}| u_{x'} - a_j(x_0)|\, \mathrm{d}s
    \]
    where $C$ is not dependent on $x'$. Using the definition of approximate limits {and Fubini's theorem}, we infer for $m$ sufficiently large 
    \begin{align*}
        \rho_m\H^{N-1}(\{x' \in Q'(\rho_m) : \H^0(J_{u_{x'}}) = 0 \}) & \leq C\left( \int_{Q^+(\rho_m)}| u - a_i(x_0)|\, \mathrm{d}x + \int_{Q^-(\rho_m)}| u - a_j(x_0)|\, \mathrm{d}x \right) \\
        &\leq {\frac{\delta}{8}} \rho_m^{N}.
    \end{align*}
    {Now, by \eqref{eq:upper_bound_jump_sets} we also derive that 
    \[
        \H^{N-1}(\{x' \in Q'(\rho_m): \H^0(J_{u_{x'}}) \geq 2\}) \leq \frac{\delta}{4} \rho_m^{N-1}. 
    \]
    Indeed, assume the opposite holds, i.e., 
    \[
        \H^{N-1}(\{x' \in Q'(\rho_m): \H^0(J_{u_{x'}}) \geq 2\}) > \frac{\delta}{4}\rho_m^{N-1}. 
    \]
    Then we would immediately infer for $m$ sufficiently large
    \[
        \frac{1}{\rho_m^{N-1}} \int_{Q'(\rho_m)} \H^0(J_{u_{x'}}) \, \mathrm{d}x' \geq 1 + \frac{\delta}{8}.
    \]
    which contradicts \eqref{eq:upper_bound_jump_sets}.
    
    In particular, for $m \in \N$ sufficiently large, we have that the measurable set
    \[
        M_m \coloneqq \{x' \in Q'(\rho_m) : \H^0(J_{u_{x'}}) = 1\}
    \]
    fulfils}
    \[
        \H^{N-1}(M_m) \geq \left(1 - \frac{\delta}{2} \right)\rho_m^{N-1}.
    \]

    \begin{figure}[t]
    \centering
    \input{fig_nice_set_liminf_measure}
    \caption{Estimates of the several sets that we consider}
    \label{fig:measure_liminf}
    \end{figure}
    
    Since the trace of the slices coincides with the standard trace for $\H^{N-1}$-a.e. $x' \in Q'$, by \eqref{eq:conv_1} and \eqref{eq:conv_2}, we can further assume {(after again extracting a subsequence)} that
    \begin{align*}
        {u_{x'}^{\rho_m}(1/2)} \to a_i(x_0) \quad \text{ and } \quad  {u^{\rho_m}_{x'}(-1/2)} \to a_j(x_0)
    \end{align*}
    for all $x' \in {M_m/\rho_m}$ (after possibly removing a null set). Furthermore, we denote the jump point in $J_{u_{x'}}$ by $(x', h(x'))$. 
    
    {We claim that there exists $N_m \subset M_m$ with 
    $$
        \H^{N-1}(N_m) \leq \frac{\delta\rho_m^{N-1}}{2}
    $$ 
    such that for $x' \in M_m\setminus N_m$ 
    \begin{align}\label{eq:cond_set_1}
        \H^1(\{ (x', t): t > h(x') \} \cap \Omega_l \}) = 0
    \end{align}
    for all $l \neq i$ and 
    \begin{align}\label{eq:cond_set_2}
        \H^1(\{ (x', t): t < h(x') \} \cap \Omega_l \}) = 0
    \end{align}
    for all $l \neq j$. Indeed, denote with 
    \[
        N_m \coloneqq \{ x' \in M_m : \eqref{eq:cond_set_1} \text{ or } \eqref{eq:cond_set_2} \text{ fails} \}
    \]
    Suppose there exists a subsequence $\{m_\alpha\}$ for which we have
    \[
        |N_{m_\alpha}|  \geq \frac{\delta}{2} \rho_{m_\alpha}^{N-1}.
    \]
    By rescaling and rewriting \eqref{eq:cond_set_1} and \eqref{eq:cond_set_2} accordingly, we can infer the existence of sets $\widetilde N_{m_\alpha} \subset Q'$ with $\H^{N-1}(\widetilde N_{m_\alpha}) \geq \delta/2$ such that for all $x' \in \widetilde N_{m_\alpha}$
    \begin{align}\label{eq:first_case}
        \H^1\left( \{ ({\rho_m} x', t): t > h(x') \} \cap \bigcup_{l \neq i} \Omega_l \right) > 0
    \end{align}
    or 
    \[
        \H^1\left( \{ ({\rho_m} x', t): t < h(x') \} \cap \bigcup_{l \neq j} \Omega_l \right) > 0.
    \]
    Without loss of generality, let us assume the first case \eqref{eq:first_case} holds for all $x' \in \widetilde N_{m_\alpha}$. This, in turn, implies that for some constant $C > 0$ and $m$ sufficiently large 
    \[
        |u_{x'}(\rho_{m_\alpha}/2) - a_i(x_0)| > C
    \]
    for all $x' \in  N_{m_\alpha}$ since $u_{x'}$ only has one jump point, $a_l$ are continuous, \ref{H3} and \eqref{eq:difference_assumption}. Rewriting this, we get 
    \[
        |u_{x'}^{\rho_{m_\alpha}}(1/2) - a_i(x_0)| > C
    \]
    for all $x' \in  \widetilde N_{m_\alpha}$.
    This directly contradicts \eqref{eq:conv_1} since
    \begin{align*}
        \int_{Q'} |u_{x'}^{\rho_m}(1/2) - a_i(x_0) | \, \dd x' \geq C|\widetilde N_\alpha| \geq \frac{C\delta}{2}.
    \end{align*}
    With abuse of notation, we will denote $M_m\setminus N_m$ by $M_m$ again and note that
    \begin{align}\label{eq:relative_large_good_set}
        \H^{N-1}(M_m) \geq \left(1 - \delta \right)\rho_m^{N-1}
    \end{align}
    holds.} \\[0.5em]
    
    \textbf{Step 2: Truncation.} After this setup, we now perform the blow-up. We recall
    \begin{align}\label{eq:blow_up_property}
        \frac{d\mu}{d\lambda}(x_0) = \lim_{m \to \infty} \lim_{n \to \infty} \frac{\mu_n(Q(x_0, \rho_m))}{\rho_m^{N-1}} =  \lim_{m \to \infty} \lim_{n \to \infty} \frac{1}{\rho_m }\int_{Q(\rho_m)} \frac{1}{\eps_n}W(x, u_n) + \eps_n |\nabla u_n|^2 \, \mathrm{d}x.
    \end{align}
    Now, let $M_0 = M_0(\|a\|_\infty, f, C_G)$ be the constant given by Corollary \ref{corol:truncation_constant}. In the following, the truncation of a vector $p \in \R^M$ by a scalar $L > 0$ is to be understood component-wise, i.e., we set $(p \land L)_i \coloneqq  \max\{ \min\{ p_i, L \}, - L \} $. Using \ref{H5}, we infer the existence of $M > M_0$ and $L > \| a\|_\infty$ such that the relation 
    \begin{align*}
        W(p) \land M = W(p \land L) \land M
    \end{align*}
    holds for all $p \in \R^N$.  Now, we define 
    \begin{align*}
        \widetilde W \coloneqq  W \land {M}, \, \widetilde W_i \coloneqq  W_i \land {M}\, \text{ and }\widetilde u_n \coloneqq  u_n \land {L}.
    \end{align*}
    We first observe
    \begin{align}\label{eq:liminf_partial_inequality_1}
        &\lim_{m \to \infty} \lim_{n \to \infty} \frac{1}{\rho_m^{N-1} }\int_{Q(\rho_m)} \frac{1}{\eps_n}W(x, u_n) + \eps_n |\nabla u_n|^2 \, \mathrm{d}x  \nonumber \\
        & \hspace{1cm} \geq \liminf_{m \to \infty} \liminf_{n \to \infty} \frac{1}{\rho_m^{N-1} }\int_{Q(\rho_m)} \frac{1}{\eps_n}\widetilde W(x, \widetilde u_n) + \eps_n |\nabla \widetilde u_n|^2 \, \mathrm{d}x.
    \end{align}

    {Now, by \ref{H7} we note for $m$ sufficiently large we can assume that $Q(\rho_m) \subset U$ for a neighbourhood $U$ such there exist local adjustments $T_i$ on $U$ in the sense of Definition \ref{H7}.}
    Let us denote the corresponding inverses by ${S_i(x, \cdot) = T_i(x, \cdot)^{-1}}$. Set 
    \[
        {v_n(x) \coloneqq  \sum_{i = 1}^k \ca_{\Omega_i}(x) S_i(x,\widetilde u_n(x))}
    \]

    \textbf{Step 3: Fixing the point $x_0$ in the potential.} We observe that \ref{H8} implies
    \begin{align}\label{eq:relative_continuity_cond_bounded}
        |\widetilde W_l(x, {T_{l}(x,v_n)}) - \widetilde W_l(x_0, {T_{l}(x_0, v_n)})| \leq \omega(\rho_m)W_l(x, {T_{l}(x,v_n)}),
    \end{align}
    for all $x \in \Omega_l$ and $l = 1, \dots, k$.
    Now, we set $$Q_l(\rho_m) \coloneqq  \Omega_l \cap Q(\rho_m).$$ From \eqref{eq:relative_continuity_cond_bounded}, we infer
    \begin{align}\label{eq:liminf_partial_inequality_2}
        &\liminf_{m \to \infty} \liminf_{n \to \infty} \frac{1}{\rho_m^{N-1}} \int_{Q(\rho_m)} \frac{1}{\eps_n}\widetilde W(x, \widetilde u_n(x)) + \eps_n |\nabla \widetilde u_n|^2 \, \mathrm{d}x  \nonumber \\
        &\hspace{1 cm} = \liminf_{m \to \infty}  \liminf_{n \to \infty} \frac{1}{\rho_m^{N-1}} \sum_{l = 1}^k \int_{Q_l(\rho_m)} \frac{1}{\eps_n}\widetilde W_l(x, T_{l}(x,v_n)) + \eps_n |\nabla \widetilde u_n|^2 \, \mathrm{d}x \nonumber\\
        &\hspace{1 cm} \geq \liminf_{m \to \infty} \liminf_{n \to \infty} \frac{1}{\rho_m^{N-1}}\sum_{l = 1}^k \int_{Q_l(\rho_m)} \frac{1}{\eps_n}\widetilde W_l(x_0, T_{l}(x_0,v_n)) + \eps_n |\nabla \widetilde u_n|^2 \, \nonumber \\ 
        &\hspace{9.5 cm}- \omega(\rho_m)\frac{1}{\eps_n} \widetilde W_l(x, T_{l}(x,v_n))  \, \mathrm{d}x.
    \end{align}
    Now, we notice that we can estimate the last term  
    \begin{align*}
        \frac{1}{\rho_m^{N-1}}\omega(\rho_m)\frac{1}{\eps_n} \int_{Q_l(\rho_m)} \widetilde W_l(x, T_{l}(x,v_n))  \, \mathrm{d}x \leq \frac{\omega(\rho_m)\mu_n(Q(\rho_m))}{\rho_m^{N-1}}.
    \end{align*}
    Therefore, by \eqref{eq:blow_up_property} the term vanishes. Next, we infer 
    \begin{align}\label{eq:liminf_partial_inequality_3}
        &\liminf_{m \to \infty} \liminf_{n \to \infty} \frac{1}{\rho_m^{N-1}} \int_{Q(\rho_m)} \frac{1}{\eps_n}\widetilde W(x, \widetilde u_n(x)) + \eps_n |\nabla \widetilde u_n|^2 \, \mathrm{d}x  \nonumber \\
        & \hspace{1cm} =  \liminf_{m \to \infty} \liminf_{n \to \infty} \frac{1}{\rho_m^{N-1}}\sum_{l = 1}^k \int_{Q_l(\rho_m)} \frac{1}{\eps_n} \widetilde W_l(x_0, T_{l}(x_0,v_n)) + \eps_n |\nabla \widetilde u_n|^2 \, \mathrm{d}x \nonumber \\
        & \hspace{1cm} \geq  \liminf_{m \to \infty} \liminf_{n \to \infty} \frac{1}{\rho_m^{N-1}}\sum_{l \in \{i,j\}} \int_{Q_l(\rho_m)} \frac{1}{\eps_n} \widetilde W_l(x_0, T_{l}(x_0,v_n)) + \eps_n |\nabla \widetilde u_n|^2 \, \mathrm{d}x.
    \end{align}

    \textbf{Step 4: Adjusting the singular perturbation.}
    { Note that $\widetilde u_n \in L^{\infty}(\Omega; \R^{M})$ by construction, and by the definition \ref{def:local_adjustment} of the local adjustment we have for all $u \in \R^M$
    \[
        T_l(x, u) = t_l(x) + \widetilde T_l(x)u
    \]
    and 
    \[
        S_l(x, u) = s_l(x) + \widetilde S_l(x)u
    \]
    for a vector field $s_l \in W^{1,2}(U; \R^M) \cap C(\overline U; \R^M)$ and a matrix field $\tilde S_l \in W^{1,2}(U; \R^{M \times M})\cap C(\overline U; \R^{M \times M})$ with 
    \[
        \tilde S_l\tilde T_l = \tilde T_l \tilde S_l = Id_{M}, \qquad |\tilde T_l u| = |c_l||u|  \qquad \text{ and } \qquad |\tilde S_l u| = |c_l|^{-1}|u| 
    \]
    for all $u \in \R^M$ where $c_l \coloneqq a_l - b_l$. Then, we apply the product rule for $W^{1,2}$ functions on $\Omega_l \cap U$ to derive for $i = 1,\dots, N$
    \[
        \partial_i v_n = \partial_i s_l + (\partial_i \tilde S_l) \widetilde u_n + \tilde S_l \partial_i \widetilde u_n \in L^2(U \cap \Omega_l; \R^{M}).
    \]
    Furthermore, using Young's inequality, we obtain
    \begin{align*}
        \eps_n |\partial_i \widetilde u_n|^2 &= \eps_n |\tilde T_l \tilde S_l \partial_i \widetilde u_n|^2\\
        &= \eps_n |c_l|^2|\partial_i v_n - \partial_i \tilde S_l \widetilde u_n - \partial_l s_l|^2 \\
        &\geq \eps_n(|c_l|^2(1 - \delta)|\nabla v_n|^2 - C(|\partial_i \tilde S_l |^2 + |\nabla s_l|^2)),
    \end{align*} 
    where $C > 0$ depends on $\delta, L$, $\|a\|_\infty$ and $\|b\|_\infty$.} We apply this with 
    \[
        w_n \coloneqq  \sum_{l = 1}^k \ca_{\Omega_l}T_l(x_0,v_n)
    \]
    to infer, for $l \in \{i,j\}$, that 
    \begin{align}\label{eq:liminf_partial_inequality_4}
        &  \frac{1}{\rho_m^{N-1}}\int_{Q_l(\rho_m)} \frac{1}{\eps_n}{\widetilde W}_l(x_0, w_n) + \eps_n |\nabla \widetilde u_n|^2 \, \mathrm{d}x  \nonumber \\
        & \hspace{1cm}\geq ~  (1-\delta) \frac{1}{\rho_m^{N-1}}\int_{Q_l(\rho_m)} \frac{1}{\eps_n}{{\widetilde W}_l}(x_0, w_n) + \eps_n |c_l(x)|^2|\nabla v_n|^2  - {\eps_n C_\delta (|\nabla \tilde S_l |^2 + |\nabla s_l|^2))} \, \mathrm{d}x.
    \end{align}
    Here, the last term vanishes as $n \to \infty$. We further estimate that
    \begin{align}\label{eq:liminf_partial_inequality_5}
        &\liminf_{m \to \infty} \liminf_{n \to \infty}\frac{1}{\rho_m^{N-1}}\int_{Q_l(\rho_m)} \frac{1}{\eps_n}{{\widetilde W}_l}(x_0, w_n) + \eps_n |c_l(x)|^2|\nabla v_n|^2 \, \mathrm{d}x \nonumber \\
        & \hspace{1cm} = ~ \liminf_{m \to \infty} \liminf_{n \to \infty} \frac{1}{\rho_m^{N-1}}\int_{Q_l(\rho_m)} \frac{1}{\eps_n}{\widetilde W}_l(x_0, w_n ) + \eps_n \left( \frac{|c_l(x)|}{|c_l(x_0)|} \right)^2|\nabla w_n|^2 \, \mathrm{d}x \nonumber\\
        & \hspace{1cm} \geq ~  \liminf_{m \to \infty} \liminf_{n \to \infty} \frac{1}{\rho_m^{N-1}}\int_{Q_l(\rho_m)}\frac{|c_l(x)|}{|c_l(x_0)|} 2\sqrt{{\widetilde W}_l(x_0, w_n)}|\nabla w_n| \, \mathrm{d}x  \nonumber \\
        & \hspace{1cm} \geq ~ \liminf_{m \to \infty} \liminf_{n \to \infty} \frac{1-\widetilde \omega_l(\rho_m)}{\rho_m^{N-1}}\int_{Q_l(\rho_m)} 2\sqrt{{\widetilde W}_l(x_0, w_n)}|\nabla w_n| \, \mathrm{d}x \nonumber \\
        & \hspace{1cm} \geq ~ (1-\delta) \liminf_{m \to \infty} \liminf_{n \to \infty} \frac{1}{\rho_m^{N-1}} \int_{Q_l(\rho_m)} 2\sqrt{{\widetilde W}_l(x_0, w_n)}|\nabla w_n| \, \mathrm{d}x, 
    \end{align}
    where $\widetilde \omega_l$ is the modulus of continuity of $x \mapsto |c_l(x)/c_l(x_0)|$ at $x_0$ {which is well-defined by \ref{H3}}. To summarize \eqref{eq:blow_up_property}, \eqref{eq:liminf_partial_inequality_1}, \eqref{eq:liminf_partial_inequality_2}, \eqref{eq:liminf_partial_inequality_3}, \eqref{eq:liminf_partial_inequality_4} and \eqref{eq:liminf_partial_inequality_5}, at this point we have shown
    \begin{align}\label{eq:liminf_partial_result}
        \frac{d\mu}{d\lambda}(x_0) \geq (1-\delta)^2 \liminf_{m \to \infty} \liminf_{n \to \infty} \frac{1}{\rho_m^{N-1}} \sum_{l \in \{i,j\}} \int_{Q_l(\rho_m)} 2\sqrt{{\widetilde W}_l(x_0, w_n)}|\nabla w_n| \, \mathrm{d}x.
    \end{align}

    \textbf{Step 5: Lower bound.}
    Since 
    $$|(M_m \times (-\rho_m/2, \rho_m/2)) \setminus (Q_i(\rho_m) \cup Q_j(\rho_m))| = 0$$ 
    holds by construction, using \eqref{eq:cond_set_1} and \eqref{eq:cond_set_2} we derive
    \begin{align}\label{eq:liminf_partial_inequality_6}
        &\sum_{l \in \{i,j\}} \int_{Q_l(\rho_m)} {2\sqrt{{\widetilde W}_l(x_0, w_n)}|\nabla w_n|}  \, \mathrm{d}x \nonumber \\
        & \hspace{1 cm} \geq \int_{M_m} \int_{-\rho_m/2}^{h(x')} 2\sqrt{{\widetilde W}_j(x_0, w_n)}|\nabla w_n| \, \mathrm{d}t + \int_{h(x')}^{\rho_m/2} 2\sqrt{{\widetilde W}_i(x_0, w_n)}|\nabla w_n| \, \mathrm{d}t \, \mathrm{d}x' \nonumber \\
        & \hspace{1 cm} \geq \int_{M_m} d_{\widetilde W_j}(x_0, w_n(x', -\rho_m/2), w_n^-(x', h(x'))) + d_{\widetilde W_i}(x_0, w_n^+(x', h(x')), w_n(x', \rho_m/2)) \, \mathrm{d}x',
    \end{align}
    where $w_n^+, w_n^-$ are the one-sided trace values at $(x', h(x'))$.
    Now, we observe that 
    \begin{align*}
        w_n(\cdot , -\rho_m/2) = T_{j}(x_0, S_{j}((\cdot , -\rho_m/2), \widetilde u_n(\cdot, -\rho_m/2))) \xrightarrow{n \to \infty} T_{j}(x_0, S_{j}((\cdot , -\rho_m/2), u(\cdot, -\rho_m/2)))
    \end{align*}
    in $L^1(Q'(\rho_m))$ (analogously for $w_n(\cdot , \rho_m/2)$). Moreover, we have 
    \begin{align*}
        w_n^-(x', h(x')) = T_{j}(x_0, S_{j}((x', h(x')), u_n(x', h(x'))))
    \end{align*}
    and 
    \begin{align*}
        w_n^+(x', h(x')) = T_{i}(x_0,S_{i}((x', h(x')), u_n(x', h(x')))).
    \end{align*}
    Now, we set 
    \[
        z^-_m(x') \coloneqq  T_{j}(x_0, S_{j}((x' , -\rho_m/2), u(x', -\rho_m/2))),
    \]
    \[
        z^+_m(x') \coloneqq  T_{i}(x_0, S_{i}((x' , \rho_m/2), u(x', \rho_m/2))),
    \]
    \[
        L^{j,m}_{x'}u \coloneqq  T_{j}(x_0,S_{j}((x', h(x')),u),
    \]
    and 
    \[
        L^{i,m}_{x'}u \coloneqq  T_{i}(x_0, S_{i}((x', h(x')),u)).
    \]
    Since  $d_{W_l}$ is Lipschitz on compact sets $K \subset \R^M$ for every $l = 1,\dots, k$ (cf.\ Lemma \ref{lemma:geodesic_distance_is_locally_lipschitz}), by applying Lemma \ref{lemma:adapted_geodesic_distance_is_locally_lipschitz} we obtain
    \begin{align}\label{eq:liminf_technical_equation_1}
        \inf_{r \in \R^M} d_{\widetilde W_j}(x_0, p, L^{j,m}_{x'}r) + d_{\widetilde W_i}(x_0, L^{i,m}_{x'}r, q) \leq C_K(1 + |p| + |q|). 
    \end{align}
    for any $p, q \in K$. Since the sequence $\{w_n\}$ is uniformly bounded in $L^\infty$ we can apply \eqref{eq:liminf_technical_equation_1} in conjunction with the Dominated Convergence Theorem to infer 
    \begin{align}\label{eq:liminf_partial_inequality_7}
        &\lim_{n \to \infty}\int_{M_m} d_{\widetilde W_j}(x_0, w_n(x', -\rho_m/2), w_n^-(x', h(x'))) + d_{\widetilde W_i}(x_0, w_n^+(x', h(x')), w_n(x', \rho_m/2)) \, \mathrm{d}x' \nonumber \\
        & \hspace{1cm }\geq \lim_{n \to \infty} \int_{M_m} \inf_{r \in \R^M} \left( d_{\widetilde W_j}(x_0, w_n(x', -\rho_m/2), L^{j,m}_{x'}r) + d_{\widetilde W_i}(x_0, L^{i,m}_{x'}r, w_n(x', \rho_m/2)) \right) \, \mathrm{d}x' \nonumber \\
        & \hspace{1cm} = \int_{M_m} \inf_{r \in \R^M} \left( d_{\widetilde W_j}(x_0, z_m^-(x'), L^{j,m}_{x'}r) + d_{\widetilde W_i}(x_0, L^{i,m}_{x'}r, z_m^+(x')) \right) \, \mathrm{d}x'.
    \end{align}
    Again, using the uniform boundedness of the integrand in $L^\infty$ with \eqref{eq:relative_large_good_set}, we obtain
    \begin{align}\label{eq:liminf_partial_inequality_8}
        &\int_{M_m} \inf_{r \in \R^M} \left( d_{\widetilde W_j}(x_0, z_m^-(x'), L^{j,m}_{x'}r) + d_{\widetilde W_i}(x_0, L^{i,m}_{x'}r, z_m^+(x')) \right) \, \mathrm{d}x' \nonumber \\
        & \hspace{1cm} \geq \int_{Q'(\rho_m)} \inf_{r \in \R^M} \left( d_{\widetilde W_j}(x_0, z_m^-(x'), L^{j,m}_{x'}r) + d_{\widetilde W_i}(x_0, L^{i,m}_{x'}r, z_m^+(x')) \right) \, \mathrm{d}x' - C\delta\rho_m^{N-1}.
    \end{align}
    Rescaling the integral by $1/\rho_m^{N-1}$, we derive 
    \begin{align}\label{eq:liminf_partial_inequality_9}
        &\frac{1}{\rho_m^{N-1}}\int_{Q'(\rho_m)} \inf_{r \in \R^M} \left( d_{\widetilde W_j}(x_0, z_m^-(x'), L^{j,m}_{x'}r) + d_{\widetilde W_i}(x_0, L^{i,m}_{x'}r, z_m^+(x')) \right) \, \mathrm{d}x \nonumber \\
        & \hspace{1cm} = \int_{Q'} \inf_{r \in \R^M} \left( d_{\widetilde W_j}(x_0, z_m^-(\rho_m x'), L^{j,m}_{\rho_m x'}r) + d_{\widetilde W_i}(x_0, L^{i,m}_{\rho_m x'}r, z_m^+(\rho_m x')) \right) \, \mathrm{d}x. 
    \end{align}

    Now again, since $z_m^\pm$ are uniformly bounded in $L^\infty$, also the $\inf$ is uniformly attained in a large ball (cf.\ Proposition \ref{prop:infimum_attained_in_ball}). Therefore, we notice that the integrand is uniformly bounded in $L^\infty(Q')$. By \eqref{eq:conv_1} and \eqref{eq:conv_2}, and the continuity of $S_{a_l}$ we also obtain
    \[
        z_m^+(\rho_m \cdot ) \xrightarrow{m \to \infty} a_i(x_0) \qquad \text{ and } \qquad  z_m^-(\rho_m \cdot ) \xrightarrow{m \to \infty} a_j(x_0)
    \]
    in $L^1(Q', \R^M)$ (analogously to \ref{eq:liminf_partial_inequality_10}). In particular, we have pointwise convergence almost everywhere. {Now, we also observe that for $\H^{N-1}$-a.e. $x' \in Q'$ that
    \[
        L^{j,m}_{\rho_m x'}u = \tilde l^{j,m}_{\rho_m x'} + \tilde L^{j,m}_{\rho_m x'}u \qquad \text{ and } L^{i,m}_{\rho_m x'}u = \tilde l^{i,m}_{\rho_m x'} + \tilde L^{i,m}_{\rho_m x'}u 
    \]
    with $\tilde l^{j,m}_{\rho_m x'}, \tilde l^{i,m}_{\rho_m x'} \in \R^M$ and $ \tilde L^{j,m}_{\rho_m x'}, \tilde L^{i,m}_{\rho_m x'} \in \R^{M \times M}$ such that 
    \[
        \tilde l^{j,m}_{\rho_m x'} \to 0 \qquad \text{ and } \qquad \tilde l^{i,m}_{\rho_m x'} \to 0
    \]
    and
    \[
        \tilde L^{i,m}_{\rho_m x'} \xrightarrow{m \to \infty} Id_{M} \qquad \text{ and } \qquad \tilde L^{j,m}_{\rho_m x'} \xrightarrow{m \to \infty} Id_{M}
    \]
    holds by Remark \ref{remark:for_liminf}.}
    In particular, we apply Lemma \ref{lemma:technical} and infer for $\H^{N-1}$-a.e. $x' \in Q'$
    \begin{align*}
        &\inf_{r \in \R^M} \left( d_{\widetilde W_j}(x_0, z_m^-(\rho_m x'), L^{j,m}_{\rho_m x'}r) + d_{\widetilde W_i}(x_0, L^{i,m}_{\rho_m x'}r, z_m^+(\rho_m x')) \right)  \\
        & \hspace{3cm} \xrightarrow{m \to \infty} \inf_{r \in \R^M} \left( d_{\widetilde W_j}(x_0, a_j(x_0), r) + d_{\widetilde W_i}(x_0, r, a_i(x_0)) \right).
    \end{align*}
    Applying the dominated convergence theorem, we infer 
    \begin{align}\label{eq:liminf_partial_inequality_10}
        & \lim_{m \to \infty} \int_{Q'} \inf_{r \in \R^M} \left( d_{\widetilde W_j}(x_0, z_m^-(\rho_m x'), L^{j,m}_{\rho_m x'}r) + d_{\widetilde W_i}(x_0, L^{i,m}_{\rho_m x'}r, z_m^+(\rho_m x')) \right) \, \mathrm{d}x' \nonumber \\
        & = \inf_{r \in \R^M} \left(  d_{\widetilde W_j}(x_0, a_j(x_0), r) + d_{\widetilde W_i}(x_0, r, a_i(x_0))\right) \nonumber \\
        & = d_{\widetilde W}(x_0, u^-(x_0), u^+(x_0)). 
    \end{align}
    Gathering \eqref{eq:liminf_partial_result}, \eqref{eq:liminf_partial_inequality_6}, \eqref{eq:liminf_partial_inequality_7}, \eqref{eq:liminf_partial_inequality_8}, \eqref{eq:liminf_partial_inequality_9}, \eqref{eq:liminf_partial_inequality_10}, we obtain  
    \begin{align*}
        \frac{d\mu}{d\lambda}(x_0) \geq (1-\delta)^2d_{\widetilde W}(x_0, u^-(x_0), u^+(x_0)) - C\delta = (1-\delta)^2d_{ W}(x_0, u^-(x_0), u^+(x_0)) - C\delta. 
    \end{align*}
    by our choice of truncation constants (cf.\ Corollary \ref{corol:truncation_constant}) and with $C > 0$ not depending on $\delta$.  
    Since $\delta > 0$ was arbitrary, we conclude.
\end{proof}


\section{Limsup inequality}

This section is devoted to the proof of the limsup inequality.

\begin{theorem}\label{thm:limsup}
    Let $\{ \varepsilon_n \}_n$ be an infinitesimal sequence. 
    Assume that \ref{H1}-\ref{H8} hold.
    Let $u \in L^1(\Omega; \R^M)$. Then, there exists $\{u_n\}_n \subset W^{1,2}(\Omega; \R^M)$ such that 
    \[
        \mathcal{F}_\infty(u) = \lim_{n \to \infty} \mathcal{F}_{\eps_n}(u_n).
    \]
    and $u_n \to u$ in $L^1(\Omega;\R^M)$.
\end{theorem}

To prove this, we need the following preliminary result.

\begin{lemma}\label{lemma:reparameterization}
    Fix $\lambda > 0$, $\varepsilon > 0$, $x \in \Omega$, {$i \in \{ 1,\ldots,k\}$} and $p,q \in \R^M$. Let $\gamma \in C^1([-1,1]; \R^M)$ with $\gamma'(s) \neq 0$ for all $s \in (-1,1)$, $\gamma(-1) = p$, $\gamma(1) = q$. Then, there exist $\tau > 0$ and $C > 0$ with
    $$ C \varepsilon \leq \tau \leq \frac{\varepsilon}{\sqrt{\lambda}} \int_{-1}^1 | \gamma'(s)| \dd s, $$
    and $g \in C^1([0, \tau]; [-1,1])$ such that
    \begin{equation*}
        (g'(t))^2 = \frac{\lambda + {W_i(x,\gamma(g(t)))}}{\varepsilon^2 | \gamma'(g(t))|^2},
    \end{equation*}
    for all $t \in (0,\tau)$, $g(0) = -1$, $g(\tau) = 1$, and
    \begin{align*}
        \int_0^\tau \left[ \frac{1}{\varepsilon} {W_i(x, \gamma(g(t)))} + \varepsilon |\gamma'(g(t))|^2 | g'(t)|^2 \right] \dd t \leq \int_{-1}^1 2 \sqrt{{W_i(x, \gamma(s))}} | \gamma'(s) | \dd s + 2 \sqrt{\lambda} \int_{-1}^1 | \gamma'(s) |\dd s.
    \end{align*}
\end{lemma}

A proof of this result can be found in \cite[Lemma 6.3]{CriGra} (see also \cite[Proposition 2]{modica87} and \cite[Lemma 3.2]{baldo}). {As a consequence of Lemma \ref{lemma:reparameterization}, we have the following result.

\begin{lemma}\label{lemma:reparameterization_2}
    Suppose $W$ satisfies \ref{H1}--\ref{H6} and \ref{H8}. Let $\varepsilon>0$, $\lambda > 0$, $x \in \Omega$, $i, j \in \{1,\dots k\}$, $i \leq j$, $p \in \{a_i(x), b_i(x)\}$ and $q \in \{a_j(x), b_j(x)\}$. Then, there exists a constant $C_\lambda > 0$ with the property that for every $\eps \in (0, C_\lambda^{-1})$ there exists $\gamma \in W^{1,1}((-1,1); \R^M)$ such that $\gamma(s) = p$ for $s < -C_\lambda \eps$, $\gamma(s) = q$ for $s > C_\lambda \eps$ and 
    \begin{align}\label{eq:optimal_profile_prop}
        \int_{-1}^0 \left[ \frac{1}{\varepsilon} W_i(x, \gamma(t)) + \varepsilon |\gamma'(t)|^2 \right] \dd t + \int_{0}^1 \left[ \frac{1}{\varepsilon} W_j(x, \gamma(t)) + \varepsilon |\gamma'(t)|^2 \right] \dd t \leq d_{W_{i,j}}(x; p,q) + C\lambda.
    \end{align}
    Moreover, we have 
    \begin{align}\label{eq:repara_bound_prop}
        \| \gamma \|_\infty + \eps \|\gamma ' \|_\infty < C.
    \end{align}
        
\end{lemma}
\begin{proof}
    Let $K = \max_{i=1, \dots, k} \{\|a_i\|_\infty, \|b_i\|_\infty \}$.
    By Proposition \ref{prop:infimum_attained_in_ball}, there exist $r \in \R^M$ with 
    \begin{align}\label{eq:uniform_bound_r}
        |r| \leq C_K
    \end{align}
    for a constant $C_K$ independent of $x$ such that 
    \begin{align}\label{eq:proof_repar_1}
        d_{W_{i}}(x; p,r) + d_{W_{j}}(x; r,q) \leq d_{W_{i,j}}(x; p,q) + \lambda.
    \end{align}
    Now, by Lemma \ref{lemma:reparameterization} (after a reparameterization and renaming of $\lambda$) there exist two curves $\widetilde \gamma_i \in W^{1,1}((-1,0); \R^M)$ and $\widetilde \gamma_j \in W^{1,1}((0,1); \R^M)$ and a constant $C_\lambda > 0$ {scaling like $\lambda^{-1}$} such that 
    \[
        \widetilde \gamma_i(s) = p \text{ and } \widetilde \gamma_j(-s) = q
    \]
    for $s \in (-1, C_\lambda \eps)$ and $\widetilde \gamma_i(0) = \widetilde \gamma_j(0) = r$, and such that 
    \begin{align}\label{eq:proof_repar_2}
        \int_{-1}^0 \left[ \frac{1}{\varepsilon} W_i(x, \widetilde \gamma_i(t)) + \varepsilon |\widetilde \gamma_i'(t)|^2 \right] \dd t \leq d_{W_i}(x; p,r) + \lambda \int_{-1}^0 |\widetilde \gamma_i'| \, ds,
    \end{align}
    and 
    \begin{align}\label{eq:proof_repar_3}
        \int_{0}^1 \left[ \frac{1}{\varepsilon} W_j(x, \widetilde \gamma_j(t)) + \varepsilon |\widetilde \gamma_j'(t)|^2 \right] \dd t \leq d_{W_j}(x; r,q) + \lambda \int_{0}^1 |\widetilde \gamma_j'| \, ds.
    \end{align}
        
    By Theorem \ref{existence-of-geodesics-with-bounded-path-length}, we can assume that $\widetilde \gamma_i$ and $\widetilde \gamma_j$ have uniformly bounded path length independent from $x$ due to \eqref{eq:uniform_bound_r}. In particular, defining 
    \[
        \gamma \coloneqq \begin{cases}
            \widetilde \gamma_i & \text{ in }(-1,0), \\
            \widetilde \gamma_j & \text{ in }(0,1),
        \end{cases}
    \]
    we get \eqref{eq:optimal_profile_prop} by combining \eqref{eq:proof_repar_1}, \eqref{eq:proof_repar_2} and \eqref{eq:proof_repar_3}. Since \eqref{eq:repara_bound_prop} if just a reformulation of the properties of the curves given by Lemma \ref{lemma:reparameterization}, we conclude.
\end{proof}}

\begin{proof}[Proof of Theorem \ref{thm:limsup}]
We divide the proof into several steps. {Thanks to Theorem \ref{thm:compactness}, without loss of generality we assume that $u \in BV(\Omega; \{ a,b\})$.}

\textbf{Step 0: Approximation with $C^2$ sets.}  
{First, let $\lambda, \kappa > 0$.
We consider 
\[
    v = a \ca_{E} + b(1-\ca_E)
\]
for a set of finite perimeter $E \subset \Omega$. By Proposition \ref{prop:approx_sets}, we can find $\delta > 0$ and a set $F \subset \Omega$ with piecewise $C^2$ boundary in $\Omega$ such that, for
\[
    u = a \ca_{F} + b(1-\ca_F),
\]
we have
\[
    |\f_\infty(u) - \f_\infty(v)| < \kappa, \qquad \| u - v \|_{L^1(\Omega)} < \kappa,
\]
and 
\begin{align}\label{eq:relbd_cond}
    \H^{N-1}(\partial_A (\partial F \cap A) ) = 0,
\end{align}
where $\partial_A$ is the relative boundary in $A$.
Observe that  
\[
    J_u \subset A \cup \bigcup_{i = 1}^k(\partial F \cap \Omega_i).
\]
Note also that for $i \neq j$ such that $\partial \Omega_i \cap \partial \Omega_j \neq \emptyset$, we can decompose the mutual boundary as follows
\[
    \partial \Omega_i \cap \partial \Omega_j = \text{int}_{\partial F}(\partial F \cap \partial \Omega_i \cap \partial \Omega_j) \cup (F \cap \partial \Omega_i \cap \partial \Omega_j) \cup ((F^c)^\circ \cap \partial \Omega_i \cap \partial \Omega_j) \cup \widetilde S
\]
with 
\begin{align}\label{eq:tilde_S_negligable_measure}
    \H^{N-1}(\widetilde S) = 0,
\end{align}
thanks to \eqref{eq:relbd_cond}. Now, we define 
\[
    S \coloneqq \widetilde S \cup S_A \cup S_{\partial F}
\]
where $S_{\partial F}$ and $S_A$ are the singular sets of $\partial F$ and $A$ in $\Omega$ respectively. By Definition \eqref{def:piece_wise_C2_admissible} and \eqref{eq:tilde_S_negligable_measure} we obtain
\begin{align}\label{eq:S_negligable_measure}
    \H^{N-1}(S) = 0.
\end{align}
Now, we set for $\theta > 0$ we set
\[
    S_\theta \coloneqq B_\theta (S \cap \partial \Omega),
\]
where $B_r(U)$ is the $r$-neighborhood around $U$.
We note that for almost all $\theta < 1$ we have 
\begin{align}\label{eq:S_theta_no_jumps}
    \hno(\partial S_\theta \cap J_u) = 0.
\end{align}
Moreover, we observe that  
\[
    (J_u \setminus S_\theta) \subset \bigsqcup_{i = 1}^k(\partial F \cap \Omega_i) \sqcup \text{int}_{\partial F}(\partial F \cap \partial \Omega_i \cap \partial \Omega_j) \sqcup (F \cap \partial \Omega_i \cap \partial \Omega_j) \sqcup ((F^c)^\circ \cap \partial \Omega_i \cap \partial \Omega_j).
\]
In particular, if we take $\theta_1 \in (0,\theta)$ we can write $J_u \setminus \overline{S_{\theta_1}}$ as a finite union of open, connected $C^2$ surfaces $\Sigma_l$, $l = 1, \dots, l_0$ such that $\Sigma_l$ lies in only one component of the above disjoint union. Moreover, we have by \eqref{eq:S_negligable_measure} 
\begin{align}\label{eq:jump_set_vanishes_in_Sdelta}
    \H^{N-1}(J_u \cap S_\theta) < \kappa.
\end{align}

\input{fig_global}

The idea is now to construct the recovery sequence for each $\Sigma_l$ separately. For readability, in Step 1 to 4, we will set $\Sigma = \Sigma_l$ for some $l$, and, for simplicity, we only treat the case  
\begin{align*}
    \Sigma \subset \text{int}_{\partial F}(\partial F \cap \partial \Omega_i \cap \partial \Omega_j),
\end{align*}
for some $i,j \in \{1,\dots, k\}$ with $i < j$. Without loss of generality, this corresponds to the case where $u$ jumps from $a_i$ to $b_j$, i.e., we assume with the convention in Definition \ref{def:limiting_functional}
\begin{align*}
    u^-(x) = a_i(x) \qquad \text{ and } \qquad u^+(x) = b_j(x)
\end{align*}
for all $x \in \Sigma$. The other cases are treated analogously.} \\

\textbf{Step 1: Construction of a local recovery sequence.}
For $0 < \alpha < \beta < 1$ define
\begin{equation*}
    r_n \coloneqq \varepsilon_n^\alpha, \qquad \eta_n \coloneqq \varepsilon_n^\beta.
\end{equation*}
This choice allows us to have $\varepsilon_n \ll \eta_n \ll r_n$.
{We define
\begin{equation*}
    \ell_n \coloneqq C_\lambda \varepsilon_n,
\end{equation*}
where $C_\lambda$ is the constant given by Lemma \ref{lemma:reparameterization_2}.}
We  can assume (after possibly increasing $S$ by a $\mathcal{H}^{N-1}$-negligible amount) that $\Sigma$ is the graph of a $C^2$ function, i.e., there exist an open set $D \subset \R^{N-1}$ and a $C^2$ diffeomorphism $\psi \colon D \to \Sigma \subset \R^N$.\\
For $y \in D$, we define the normal to the interface at $\psi(y)$ as $\nu(y)$, where we assume that $\nu$ points into the $\Omega_j$ region. With this, we define, for $\eta > 0$ chosen sufficiently small, the function $\Psi \colon D \times [-\eta, \eta] \to \R^N$ as 
$$ \Psi(y,t) \coloneqq \psi(y) + t \nu(y), $$
such that we have
$$ \Psi(D \times (0,\eta)) \subset \Omega_j, \qquad \Psi(D \times (-\eta,0)) \subset \Omega_i. $$
For $\eta$ small enough, the tubular neighborhood theorem guarantees that this map is a $C^1$ diffeomorphism, and that every point in it has a unique projection on the interface, meaning that $\Psi(y,t)$ get projected onto $\psi(y)$. \\
Take now a lattice of points $\{y_m\}_{m \in I_n} \subset D$ with spacing $2r_n$, where $I_n \subset \N$ is a suitable index set with
\begin{equation}\label{eq:estimate_index_set}
    \# | I_n | \leq C \frac{1}{r_n^{N-1}},
\end{equation}
where the constant $C$ does not depend on $n$. Define the set
\[
    U_n^m \coloneqq \{ x \in \Omega : \exists y \in Q'(y_m,r_n), \exists t \in [-\eta, \eta] : \Psi(y,t) = x \},
\]
as the local tubular neighborhood around $\psi(y_m)$.
\input{fig_tubular}

Let now $\{\varphi_n^m\}_{n \in \N}^{m \in I_n} \subset C^\infty (D)$ be a partition of unity such that
\begin{equation*}
    0 \leq \varphi_n^m \leq 1, \qquad \varphi_n^m \equiv 1 \in Q'(y_m,r_n-\eta_n), \qquad \varphi_n^m \equiv 0 \in D \setminus Q'(y_m,r_n+\eta_n)
\end{equation*}
\begin{equation*}
    \sum_{m \in I_n} \varphi_n^m(y) \equiv 1 \;\;\text{ in }\; D, \qquad | \nabla \varphi_n^m | \leq \frac{C}{2\eta_n}.
\end{equation*}
{Moreover, let $\gamma_m$ be the curve given by Lemma \ref{lemma:reparameterization_2} with respect to $x = \psi(y_m)$, namely a curve in $W^{1,1}(\R; \R^M)$ with $\gamma_m (s) = u^- (\psi(y_m)) = a_i (\psi(y_m))$ for $s < - \ell_n$, $\gamma_m (s) = u^+ (\psi(y_m)) = b_j(\psi(y_m))$ for $s > \ell_n$, such that 
\begin{align}\label{eq:curve_optimality_property}
    &\int_{-\ell_n}^0 \left[ \frac{1}{\varepsilon} W_i(\psi(y_m), \gamma_m(t)) + \varepsilon |\gamma_m'(t)|^2 \right] \dd t + \int_{0}^{\ell_n} \left[ \frac{1}{\varepsilon} W_j(\psi(y_m), \gamma_m(t)) + \varepsilon |\gamma_m'(t)|^2 \right] \dd t \nonumber \\
    & \hspace{5cm} \leq d_{W_{i,j}}(\psi(y_m); a_i(\psi(y_m)), b_j(\psi(y_m))) + C\lambda.
\end{align}
and 
\begin{align}\label{eq:curve_boundededness_property}
    \| \gamma_m \|_\infty + \eps\|\gamma_m '\|_\infty < C.
\end{align}

Now, we adjust the boundary conditions of $\gamma_m$ by setting for $y \in Q'(y_m, r_n)$ and $t \in \R$
\[
    \widetilde \gamma_m(y,t) = \gamma_m(t) + \zeta_n^+(t) (b_j(\Psi(y,t)) - b_j(\psi(y_m))) + \zeta_n^-(t) (a_i(\Psi(y,t)) - a_i(\psi(y_m)))
\]
for suitable cut-off functions $\zeta_n^+, \zeta_n^- \in C^\infty (\R)$ such that 
\[
    \zeta_n^-(s) = \zeta_n^+(-s) = 0 \text{ for $s > -\ell_n$ and } \zeta_n^-(s) = \zeta_n^+(-s) = 1 \text{ for $s < -2\ell_n$},
\]
and 
\[
    |(\zeta_n^\pm)'| \leq \frac{C}{\eps_n}.
\]
At this point, we observe that $\widetilde \gamma_m$ is uniformly bounded in $L^\infty$ since $a,b \in L^\infty$ and by \eqref{eq:curve_boundededness_property}.
In particular, for $n$ sufficiently large, we infer for a.e. $y \in Q'(y_m, r_n)$ and a.e. $t \in (-2\ell_n, 2\ell_n)$
\begin{align}\label{eq:relative_W_estimate}
    &|W_i(\psi(y_m), \gamma_m(t)) - W_i(\Psi(y,t), \widetilde \gamma_m(y,t))| \nonumber \\
    &\qquad \leq \omega_{W_i} \big(|\psi(y_m) - \Psi(y,t)| + |\gamma(t) - \widetilde \gamma_m (y,t)| \big) \nonumber \\
    &\qquad \leq \omega_{W_i} \big(|\psi(y_m) - \Psi(y,t)| + |a_i(\Psi(y,t)) - a_i(\psi(y_m))| + |b_j(\Psi(y,t)) - b_j(\psi(y_m))|\big) \nonumber \\
    &\qquad \leq \omega_{W_i} \big( r_n + 2\ell_n + \omega_{a_i} ( r_n + 2\ell_n ) + \omega_{b_j} ( r_n + 2\ell_n ) \big) \nonumber \\
    &\qquad \leq \omega(r_n + \varepsilon_n), 
\end{align}
and analogously with $W_j$, for a suitable uniform modulus of continuity $\omega$ only depending on $W_i$, $W_j$, $a_i$, $b_j$ and $\Psi$. Moreover, for the derivative we estimate for a.e. $y \in Q'(y_m, r_n)$ and a.e. $t \in (-2\ell_n, -\ell_n) \cup (\ell_n, 2\ell_n)$
\begin{align}\label{eq:estimate_derivative_y}
    |\nabla_y \widetilde \gamma_m| \leq C(|\nabla a_i| + |\nabla b_j|)
\end{align}
and 
\begin{align}\label{eq:estimate_derivative_t}
    |\partial_t \widetilde \gamma_m| \leq C\big(|\nabla a_i| + |\nabla b_j| + \eps_n^{-1}\omega(r_n + \eps_n)\big).
\end{align}}
Next, we define the recovery sequence in the transformed domain for $y \in D$ and $t \in \R$ by 
\begin{align*}
    v_n(y,t) \coloneqq \sum_{m \in I_n} \varphi_n^m(y) \widetilde \gamma_m(y,t),
\end{align*}
(see Figure \ref{fig:limsup_local}).
\input{fig_local}
This sequence respects boundary values since for $y \in Q'(y_m, r_n)$ and $s > 2\ell_n$ we derive
\begin{align*}
    v_n (y, s) &= \sum_{m \in I_n} \varphi_n^m(y) \widetilde \gamma_m(y,s) \\
    &= \sum_{m \in I_n} \varphi_n^m(y) b_j (\Psi(y, s)) \\
    &= b_j (\Psi(y, s)) \sum_{m \in I_n} \varphi_n^m(y) \\
    &= b_j (\Psi(y, s)).
\end{align*}
Analogously, the correct boundary conditions hold for $s < -2\ell_n$.
Since $a, b \in L^\infty(\Omega; \R^M)$ we infer by \eqref{eq:curve_boundededness_property}
\begin{equation}\label{eq:vn_bounded}
    \sup_{n \in \N} \| v_n \|_\infty \leq C < + \infty.
\end{equation}
Since $\Psi$ is a $C^1$-diffeomorphism, we can define the local recovery sequence $u_n$ on 
\[
    O \coloneqq \Psi(D \times (-\eta, \eta))
\]
by 
\begin{gather*}
    u_n(x) \coloneqq v_n(\Psi^{-1}(x)).
\end{gather*}

Note that this construction implies $u_n \in W^{1,2}(O; \R^M)$.
Finally, we denote the tubular neighbourhood by $\Sigma^n \coloneqq \Psi(D \times [-2\ell_n, 2\ell_n])$ (see Figure \ref{fig:tub_neighborhood}).\\

\textbf{Step 2: Localization in strip.}
Since we define the recovery sequence to be equal to the wells outside of the $2\ell_n$-tubular neighborhood of the interface, and the wells are in $W^{1,2}$, we get that it is only necessary to study what happens inside of the tubular neighborhood $\Sigma^n$.\\
Indeed,
\begin{align*}
    \mathcal{F}_{\varepsilon_n}(u_n, O) &= \int_O \left[ \frac{1}{\varepsilon_n} W(x, u_n) + \varepsilon_n | \nabla u_n|^2 \right] \dd x \\
    &= \int_{\Sigma^n} \left[ \frac{1}{\varepsilon_n} W(x, u_n) + \varepsilon_n | \nabla u_n|^2 \right] \dd x + { \int_{O \cap\{ \mathrm{dist}_\Sigma(x) > 2\ell_n \}}  \varepsilon_n |\nabla b_j|^2 \dd x} \\
    &\qquad+ {\int_{O \cap\{ \mathrm{dist}_\Sigma(x) < - 2\ell_n \}} \varepsilon_n |\nabla a_i|^2 \dd x}.
\end{align*}
{These last two terms can be bounded by
\begin{align*}
    \int_{O \cap\{ \mathrm{dist}_\Sigma(x) > 2\ell_n \}} \varepsilon_n |\nabla b_j|^2 \dd x &\leq \varepsilon_n \| \nabla b_j \|^2_{L^2(O; \R^{M \times N})}, \\
    \int_{O\cap \{\mathrm{dist}_\Sigma(x) < - 2\ell_n \}} \varepsilon_n |\nabla a_i|^2 \dd x &\leq \varepsilon_n \| \nabla a_i \|^2_{L^2(O; \R^{M \times N})},
\end{align*}
which vanish in the limit since $a_i, b_j \in W^{1,2}(\Omega; \R^M)$.} Therefore, we are left with estimating the integral over $\Sigma^n$.
In order to change coordinates with the diffeomorphism $\Psi$, we need to see how the gradient changes under this reparameterization. 

Thanks to the definition of $\Psi$, we have that
\begin{gather*}
    J \Psi = ( J \psi + t J \nu \; | \; \nu ).
\end{gather*}

Since $v_n (y,t) = u_n (\Psi(y,t))$, by the chain rule we get $ \nabla_{(y,t)} v_n = (J \Psi)^T \nabla_x u_n $, therefore
$$ \nabla_x u_n = ((J \Psi)^{T})^{-1} \nabla_{(y,t)} v_n. $$
To find $| \nabla_x u_n |^2$, we compute the dot product
$$ | \nabla_x u_n |^2 = (\nabla_{(y,t)} v_n)^T ( (J\Psi)^T J\Psi)^{-1} \nabla_{(y,t)} v_n. $$
The matrix $ (J\Psi)^T J\Psi$ yields
\begin{gather*}
    (J\Psi)^T J\Psi = \begin{pmatrix}
        (J\psi + t J \nu)^T (J\psi + t J \nu) & (J\psi + t J \nu)^T \nu \\
        \nu^T (J\psi + t J \nu) & \nu^T \nu
    \end{pmatrix}.
\end{gather*}
We denote the top left block by 
\begin{equation*}
    G = (J\psi + t J \nu)^T (J\psi + t J \nu).
\end{equation*}
Since $\nu$ is a unit normal vector, we have $\nu^T \nu = |\nu|^2 = 1$. Differentiating this identity we also find that $(J\nu)^T \nu = 0$. Moreover, since $\nu $ is orthogonal to the surface tangent vectors $J \psi$, we have $(J\psi)^T \nu = 0$. In particular, these imply that the matrix is equal to
\begin{gather*}
    (J\Psi)^T J\Psi = \begin{pmatrix}
        G & 0 \\
        0 & 1
    \end{pmatrix} \implies ((J\Psi)^T J\Psi)^{-1} = \begin{pmatrix}
        G^{-1} & 0 \\
        0 & 1
    \end{pmatrix}.
\end{gather*}
Splitting the gradient of $v_n$ into a tangential part and a normal part, we get
\begin{gather*}
    | \nabla_x u_n |^2 = \begin{pmatrix}
        ( \nabla_y v_n)^T & \partial_t v_n
    \end{pmatrix} \begin{pmatrix}
        G^{-1} & 0 \\ 0 & 1
    \end{pmatrix} \begin{pmatrix}
        \nabla_y v_n \\ 
        \partial_t v_n
    \end{pmatrix} = \nabla_y v_n : G^{-1} \nabla_y v_n + | \partial_t v_n |^2.
\end{gather*}

Note that $\mathrm{det}(J\Psi)$ is bounded, and up to a change of orientation of the parametrization $\psi$, we can assume without loss of generality that $\det(J\Psi) > 0$.

Therefore, applying the change of coordinates $x = \Psi(y,t)$, we get
\begin{align*}
    &\int_{\Sigma^n} \left[ \frac{1}{\varepsilon_n} W(x, u_n) + \varepsilon_n | \nabla u_n|^2 \right] \dd x \\
    &= \sum_{m \in I_n} \int_{U_n^m} \left[ \frac{1}{\varepsilon_n} W(x, u_n) + \varepsilon_n | \nabla u_n|^2 \right] \dd x \\
    &= \sum_{m \in I_n} \int_{Q'(y_m, r_n)} \int_{-2\ell_n}^{2\ell_n} \left[ \frac{1}{\varepsilon_n} W(\Psi, v_n) + \varepsilon_n | \partial_t v_n|^2 + \varepsilon_n (\nabla_y v_n : G^{-1} \nabla_y v_n ) \right] \mathrm{det}(J\Psi) \, \mathrm{d} t \, \mathrm{d} y.
\end{align*}
For brevity, we define
\begin{equation*}
    g_n \coloneqq \frac{1}{\varepsilon_n} W(\Psi, v_n) + \varepsilon_n | \partial_t v_n|^2 + \varepsilon_n (\nabla_y v_n : G^{-1} \nabla_y v_n).
\end{equation*}

Since $\psi$ is $C^2$ and $\nu$ is $C^1$, we have that $J\psi + t J\nu$ is continuous, and therefore $G$ as well. Moreover, its norm is bounded on compact sets. Note that due to the tubular neighborhood theorem, $G$ is always strictly positive definite. These two facts imply that on $D \times (-\eta, \eta)$ we have (after possibly decreasing $\eta$ slightly)
\begin{equation}\label{eq:tangential_derivative}
    \frac1C | \nabla_y v_n |^2 \leq \nabla_y v_n : G^{-1} \nabla_y v_n \leq C | \nabla_y v_n |^2.
\end{equation}
A proof of this can be found in \cite[Lemma 2.3]{lotoreichik2023spectralasymptoticsdiracoperator}.

Now, we split the integral as follows
\begin{align*}
    \sum_{m \in I_n} \int_{Q'(y_m, r_n)} \int_{-2\ell_n}^{2\ell_n} g_n \mathrm{det}(J\Psi) \,\mathrm{d}t \, \mathrm{d}y = &\sum_{m \in I_n} \int_{Q'(y_m,r_n-\eta_n)} \int_{-2\ell_n}^{2\ell_n} g_n  \mathrm{det}(J\Psi) \, \mathrm{d} t \, \mathrm{d} y \\
    + &\sum_{m \in I_n} \int_{Q'(y_m,r_n) \setminus Q'(y_m, r_n - \eta_n)} \int_{-2\ell_n}^{2\ell_n} g_n \mathrm{det}(J\Psi) \, \mathrm{d} t \, \mathrm{d} y.
\end{align*}
We denote the first term of the right-hand side by $I_1^n$ and the second term by $I_2^n$.\\

\textbf{Step 3: Estimate of $I_2^n$.}
We prove that the contribution of $I_2^n$ is neglibile in the limit.\\
Since $W$ is locally uniformly bounded thanks to Assumption \ref{H5}, we infer from \eqref{eq:vn_bounded} that for all $n \in \N$ we have
$$ W(\Psi, v_n) \leq C, $$
for some constant $C > 0$ that can be chosen independent of $n$. This implies
\begin{align*}
    \sum_{m \in I_n} \int_{Q'(y_m,r_n) \setminus Q'(y_m, r_n - \eta_n)} &\int_{-2\ell_n}^{2\ell_n} \cfrac{1}{\varepsilon_n} W(\Psi,v_n) \mathrm{det}(J\Psi)\, \dd t \, \dd y \\
    &\leq \frac{1}{r_n^{N-1}} \left[ r_n^{N-1} - (r_n - \eta_n)^{N-1} \right] C \frac{\ell_n}{\varepsilon_n} \\
    &= C \left[ 1 - \left( 1 - \frac{\eta_n}{r_n} \right)^{N-1} \right] \to 0,
\end{align*}
where we used $\ell_n = C_\lambda \varepsilon_n$, \eqref{eq:estimate_index_set} and that $\det (J \Psi)$ is uniformly bounded.
This term vanishes in the limit since $\eta_n \ll r_n$.

For the gradient part of $I_2^n$, we note that for a point $y \in Q'(y_m,r_n) \setminus Q'(y_m, r_n - \eta_n)$, there exists a finite set of indexes $J = J(y) \subset I_n$ and a constant $C(N)$ depending only on the dimension $N$, with
\begin{equation}\label{eq:J_bound}
    \# | J(y) | \leq C(N)
\end{equation}
independent of $y$ such that $\varphi_n^m(y) = 0$ for all $m \in I_n \setminus J$. Therefore for $I_2^n$ we have
$$ {v_n(y,t) \coloneqq \sum_{m \in J(y)} \varphi_n^m(y) \widetilde \gamma_m (y,t).} $$

{We first treat the derivative with respect to $t$. For this we infer, from \eqref{eq:estimate_derivative_t} and \eqref{eq:J_bound}, that for $t \in (-2\ell_n, -\ell_n) \cup (\ell_n, 2\ell_n)$ we have
\begin{align*}
    | \partial_t v_n (y,t)|^2 &= \left| \partial_t \sum_{m \in J(y)} \varphi_n^m(y) \widetilde \gamma_m (y,t) \right|^2 \\
    &\leq \sum_{m \in J(y)} |\partial_t \widetilde \gamma_m(y,t)|^2 \\
    &\leq C \big( |\nabla a_i (y,t)|^2 + | \nabla b_j (y,t) |^2 + \varepsilon_n^{-2} \omega (r_n + \varepsilon_n)^2 \big),
\end{align*}
while for $t \in (-\ell_n, \ell_n)$ we have 
\begin{align*}
    | \partial_t v_n (y,t)|^2 &= \left| \partial_t \sum_{m \in J(y)} \varphi_n^m(y) \gamma_m (t) \right|^2 
    \leq \sum_{m \in J(y)} | \gamma_m'(t)|^2
    \leq C \| \gamma_m' \|_\infty^2 \leq C \varepsilon_n^{-2},
\end{align*}
where this last inequality follows from \eqref{eq:curve_boundededness_property}.
This implies 
\begin{align*}
    &\sum_{m \in J(y)} \int_{Q'(y_m,r_n) \setminus Q'(y_m, r_n - \eta_n)} \int_{-2\ell_n}^{2\ell_n} \varepsilon_n | \partial_t v_n|^2 \mathrm{det}(J \Psi)  \,\dd t \, \dd y \nonumber\\
    &\qquad\qquad\leq C \varepsilon_n \| \nabla a_i \|^2_{L^2} + C \varepsilon_n \| \nabla b_j\|^2_{L^2} + C \omega (r_n + \varepsilon_n)^2 \frac{\ell_n}{\varepsilon_n} + C \left[ 1 - \left( 1 - \frac{\eta_n}{r_n} \right)^{N-1} \right] \frac{\ell_n}{\varepsilon_n} \to 0,
\end{align*}
where this last term vanishes since $\ell_n = C_\lambda \varepsilon_n$, together with $\mathrm{det}(J\Psi)$ being uniformly bounded, and since 
$$  1 - \left( 1 - \frac{\eta_n}{r_n} \right)^{N-1} \leq 2 (N-1) \frac{\eta_n}{r_n}, $$
and we chose $\eta_n, r_n$ such that $\eta_n \ll r_n$.}

{For the term involving the tangential derivative, using \eqref{eq:tangential_derivative} we just need to bound $|\nabla_y v_n|^2$. Therefore, using analogous reasoning as in the estimate for the normal derivative, we infer from \eqref{eq:tangential_derivative}, \eqref{eq:estimate_derivative_y} and $|\nabla \varphi_n^m| \leq \frac{1}{2\eta_n}$ that for $t \in (-2\ell_n, -\ell_n) \cup (\ell_n, 2\ell_n)$ we have
\begin{align*}
    | \nabla_y v_n(y,t)|^2 &= \left| \nabla_y \sum_{m \in J(y)} \varphi_n^m(y) \widetilde \gamma_m (y,t) \right|^2 \\
    &\leq C \sum_{m \in J(y)} | \nabla_y \varphi_n^m(y) \widetilde \gamma_m (y,t) + \varphi_n^m(y) \nabla_y \widetilde \gamma_m (y,t) |^2\\
    &\leq C \eta_n^{-2} + C | \nabla a_i (y,t)|^2 + C | \nabla b_j (y,t)|^2,
\end{align*}
while for $t \in (-\ell_n, \ell_n)$ we have
\begin{align*}
    | \nabla_y v_n(y,t)|^2 &= \left| \nabla_y \sum_{m \in J(y)} \varphi_n^m(y) \gamma_m (t) \right|^2
    \leq C \sum_{m \in J(y)} | \nabla_y \varphi_n^m(y) \gamma_m (t)|^2
    \leq C \eta_n^{-2}.
\end{align*}
Applying \eqref{eq:tangential_derivative} to these two inequalities, we obtain
\begin{align}\label{eq:an_equation_for_estimates}
    &\sum_{m \in J(y)} \int_{Q'(y_m,r_n) \setminus Q'(y_m, r_n - \eta_n)} \int_{-2\ell_n}^{2\ell_n} \varepsilon_n (\nabla_y v_n : G^{-1} \nabla_y v_n) \mathrm{det}(J\Psi)\, \dd t \, \dd y \nonumber\\
    &\qquad\qquad\qquad\leq C \varepsilon_n \| \nabla a_i \|^2_{L^2} + C \varepsilon_n \| \nabla b_j \|^2_{L^2} + \left[ 1 - \left( 1 - \frac{\eta_n}{r_n} \right)^{N-1} \right] \frac{\ell_n \varepsilon_n}{\eta_n^2} \to 0,
\end{align}
where we used that $\mathrm{det}(J \Psi)$ is uniformly bounded.
This last term vanishes since 
$$ \left[ 1 - \left( 1 - \frac{\eta_n}{r_n} \right)^{N-1} \right] \frac{\ell_n \varepsilon_n}{\eta_n^2} \leq C \frac{\varepsilon_n^2}{\eta_n^2}, $$
and we chose $\varepsilon_n, \eta_n$ such that $\varepsilon_n \ll \eta_n$. We can therefore conclude.}\\

\textbf{Step 4: estimate of $I_1^n$.}
We want to prove that $I_1^n$ converges to our desired limit, namely
\begin{align*}
    &\sum_{m \in I_n} \int_{Q'(y_m,r_n-\eta_n)} \int_{-2\ell_n}^{2\ell_n} \left[ \frac{1}{\varepsilon_n} W(\Psi, v_n) + \varepsilon_n | \partial_t v_n|^2 + \varepsilon_n (\nabla_y v_n : G^{-1} \nabla_y v_n) \right] \mathrm{det}(J\Psi)\, \dd t \, \dd y \\
    &\to \int_D \dd_W \big( \psi(y), u^-(\psi(y)),u^+(\psi(y)) \big) \, \sqrt{1 + |\nabla \psi |^2}\, \dd \mathcal{H}^{N-1}(y)\\
    &= \int_\Sigma \dd_W (x, a_i(x), b_j (x)) \, \dd \mathcal{H}^{N-1}(x).
\end{align*}

Substituting in $I_1^n$ the definition of $v_n$ for $y \in Q' (y_m, r_n - \eta_n)$, we are left with
$$ \sum_{m \in I_n} \int_{Q'(y_m,r_n - \eta_n)} \int_{-2\ell_n}^{2\ell_n} \left[ \frac{1}{\varepsilon_n} W(\Psi, \widetilde \gamma_m) + \varepsilon_n | \partial_t \widetilde \gamma_m|^2 + \varepsilon_n ( \nabla_y \widetilde \gamma_m : G^{-1} \nabla_y \widetilde \gamma_m ) \right] \mathrm{det}(J\Psi)\, \dd t\, \dd y. $$

{Using \eqref{eq:relative_W_estimate}, we can estimate the potential term by
\begin{align}\label{eq:limsup_potential_term_fin}
    &\sum_{m \in I_n} \int_{Q'(y_m,r_n - \eta_n)} \int_{-2\ell_n}^{2\ell_n} \frac{1}{\varepsilon_n} W(\Psi, \widetilde \gamma_m)  \mathrm{det}(J\Psi)\, \dd t \, \dd y \nonumber\\
    &\qquad= \sum_{m \in I_n} \int_{Q'(y_m,r_n - \eta_n)} \int_{-2\ell_n}^{0} \frac{1}{\varepsilon_n} W_i(\Psi, \widetilde \gamma_m) \mathrm{det}(J\Psi) \, \dd t \, \dd y \nonumber\\
    &\qquad\qquad+ \sum_{m \in I_n} \int_{Q'(y_m,r_n - \eta_n)} \int_{0}^{2\ell_n} \frac{1}{\varepsilon_n} W_j(\Psi, \widetilde \gamma_m) \mathrm{det}(J\Psi) \, \dd t \, \dd y \nonumber\\
    &\qquad \leq \sum_{m \in I_n} \int_{Q'(y_m,r_n - \eta_n)} \int_{-2\ell_n}^{0} \frac{1}{\varepsilon_n} W_i(\psi(y_m), \gamma_m)  \mathrm{det}(J\Psi) \, \dd t \, \dd y\nonumber\\
    &\qquad\qquad + \sum_{m \in I_n} \int_{Q'(y_m,r_n - \eta_n)} \int_{0}^{2\ell_n} \frac{1}{\varepsilon_n} W_j(\psi(y_m), \gamma_m) \mathrm{det}(J\Psi) \, \dd t \, \dd y \nonumber\\
    &\qquad\qquad + \sum_{m \in I_n} \int_{Q'(y_m,r_n - \eta_n)} \int_{-2\ell_n}^{2\ell_n} \frac{1}{\varepsilon_n} \omega (r_n + \varepsilon_n) \, \dd t \, \dd y\nonumber \\
    &\qquad \leq \sum_{m \in I_n} \int_{Q'(y_m,r_n - \eta_n)} \int_{-\ell_n}^{0} \frac{1}{\varepsilon_n} W_i(\psi(y_m), \gamma_m) \mathrm{det}(J\Psi) \, \dd t \, \dd y\nonumber\\
    &\qquad\qquad + \sum_{m \in I_n} \int_{Q'(y_m,r_n - \eta_n)} \int_{0}^{\ell_n} \frac{1}{\varepsilon_n} W_j(\psi(y_m), \gamma_m) \mathrm{det}(J\Psi) \, \dd t \, \dd y \nonumber\\
    &\qquad\qquad + C \omega (r_n + \varepsilon_n).
\end{align}
For the normal derivative, we apply \eqref{eq:estimate_derivative_t} to obtain
\begin{align}\label{eq:limsup_normal_term_fin}
    &\sum_{m \in I_n} \int_{Q'(y_m,r_n - \eta_n)} \int_{-2\ell_n}^{2\ell_n} \varepsilon_n | \partial_t \widetilde \gamma_m|^2 \mathrm{det}(J\Psi) \, \dd t \, \dd y \nonumber \\
    &\qquad\leq \sum_{m \in I_n} \int_{Q'(y_m,r_n-\eta_n)} \int_{-\ell_n}^{\ell_n} \varepsilon_n |\gamma_m'|^2 \mathrm{det}(J\Psi) \, \dd t \, \dd y \nonumber \\
    &\qquad \quad + C\eps_n\|\nabla a_i\|_{L^2(\Omega; \R^M)}^2 + C\eps_n\|\nabla b_j\|_{L^2(\Omega; \R^M)}^2 + C\omega(r_n + \eps_n)).
\end{align}
Analogously to \eqref{eq:an_equation_for_estimates}, we derive with \eqref{eq:estimate_derivative_y}
\begin{equation}\label{eq:limsup_tangential_fin}
    \sum_{m \in I_n} \int_{Q'(y_m,r_n-\eta_n)} \int_{-2\ell_n}^{2\ell_n} \varepsilon_n ( \nabla_y \widetilde \gamma_m : G^{-1} \nabla_y \widetilde \gamma_m) \mathrm{det}(J\Psi) \, \dd t \dd y \leq C\eps_n(\|a_i\|_{L^2} + \|b_j\|_{L^2}).
\end{equation}
Substituting \eqref{eq:limsup_potential_term_fin}, \eqref{eq:limsup_normal_term_fin} and \eqref{eq:limsup_tangential_fin} in $I_1^n$ and taking the limsup we can estimate
\begin{align*}
    &\limsup_{n \to \infty} \sum_{m \in I_n} \int_{Q'(y_m,r_n-\eta_n)} \int_{-2\ell_n}^{2\ell_n} g_n \mathrm{det}(J\Psi)(y,t)\, \dd t \dd y \\
    &\leq \limsup_{n \to \infty} \sum_{m \in I_n} \int_{Q'(y_m,r_n - \eta_n)} \int_{-\ell_n}^{0} \frac{1}{\varepsilon_n} W_i(\psi(y_m), \gamma_m) \mathrm{det}(J\Psi)(y,t) \, \dd t \, \dd y\nonumber\\
    &\qquad + \limsup_{n \to \infty} \sum_{m \in I_n} \int_{Q'(y_m,r_n - \eta_n)} \int_{0}^{\ell_n} \frac{1}{\varepsilon_n} W_j(\psi(y_m), \gamma_m) \mathrm{det}(J\Psi)(y,t) \, \dd t \, \dd y \nonumber\\
    &\qquad + \limsup_{n\to\infty} \sum_{m \in I_n} \int_{Q'(y_m,r_n-\eta_n)} \int_{-\ell_n}^{\ell_n} \varepsilon_n |\gamma_m'|^2 \mathrm{det}(J\Psi)(y,t) \, \dd t \, \dd y.
\end{align*}
Now, for $r_n$ and $\eps_n$ sufficiently small we infer, for $y \in Q'(y_m, r_n)$ and $t \in (-\ell_n, \ell_n)$, that
\[
    |\mathrm{det}(J\Psi)(y,t) - \mathrm{det}(J\Psi)(y_m,0)| < \lambda.
\]
Substituting this into the earlier chain of inequalities we get
\begin{align*}
    &\limsup_{n \to \infty} \sum_{m \in I_n} \int_{Q'(y_m,r_n-\eta_n)} \int_{-2\ell_n}^{2\ell_n} g_n \mathrm{det}(J\Psi)(y,t)\, \dd t \dd y \\
    &\leq \limsup_{n \to \infty} \sum_{m \in I_n} \int_{Q'(y_m,r_n - \eta_n)} \int_{-\ell_n}^{0} \frac{1}{\varepsilon_n} W_i(\psi(y_m), \gamma_m) \big( \mathrm{det}(J\Psi)(y_m,0) + \lambda \big) \, \dd t \, \dd y\nonumber\\
    &\qquad + \limsup_{n \to \infty} \sum_{m \in I_n} \int_{Q'(y_m,r_n - \eta_n)} \int_{0}^{\ell_n} \frac{1}{\varepsilon_n} W_j(\psi(y_m), \gamma_m) \big( \mathrm{det}(J\Psi)(y_m,0) + \lambda \big) \, \dd t \, \dd y \nonumber\\
    &\qquad + \limsup_{n\to\infty} \sum_{m \in I_n} \int_{Q'(y_m,r_n-\eta_n)} \int_{-\ell_n}^{\ell_n} \varepsilon_n |\gamma_m'|^2 \big( \mathrm{det}(J\Psi)(y_m,0) + \lambda \big) \, \dd t \, \dd y.
\end{align*}
Since $\gamma_m$ is independent of $y$, we can apply \eqref{eq:curve_optimality_property} to get
\begin{align*}
    &\limsup_{n \to \infty} \sum_{m \in I_n} \int_{Q'(y_m,r_n-\eta_n)} \int_{-2\ell_n}^{2\ell_n} g_n \mathrm{det}(J\Psi) \, \dd t \dd y \\
    &\leq \limsup_{n\to\infty} \sum_{m \in I_n} \mathcal{H}^{N-1} ( Q'(y_m, r_n-\eta_n) ) \big( \mathrm{det}(J\Psi)(y_m,0) + \lambda \big) \big( d_{W_{i,j}}(\psi(y_m); a_i(\psi(y_m)), b_j(\psi(y_m))) + C\lambda \big)
\end{align*}}

Let us now consider the term $ \mathrm{det} (J\Psi)  (y_m, 0)$: this would be the determinant of the matrix $(J\psi + t J \nu \;| \; \nu ) $ computed in $(y_m,0)$, or equivalently, the determinant of the matrix $(J \psi (y_m) \; | \; n(y_m) )$. In particular, this represents the $N$-dimensional volume of the hyper-rectangle spanned by the columns of the matrix. Here we can see that, since $\nu$ is a unitary vector by construction perpendicular to the other columns, this volume will also be equal to the $N-1$-dimensional volume spanned by the first $N-1$ columns, which is equal to $\sqrt{1 + | \nabla \psi (y_m) |^2}$.\\
In particular, since 
\[
    y \mapsto d_{W_{i,j}}(\psi(y); a_i(\psi(y)), b_j(\psi(y)))
\]
is continuous {on $\partial \Omega_i \cap \partial \Omega_j$ (see Remark \ref{remark:continuity_of_geodesic_distance}), we infer  
\begin{align*}
    &\limsup_{n \to \infty} \sum_{m \in I_n} \H^{N-1}(Q'(y_m, r_n - \eta_n)) \big( \det(J\Psi(y_m,0)) + \lambda \big) \big( d_{W_{i,j}}(\psi(y_m); a_i(\psi(y_m)), b_j(\psi(y_m)) ) + C\lambda \big) \\
    &\leq \limsup_{n \to \infty} \sum_{m \in I_n} \H^{N-1}(Q'(y_m, r_n)) \big( \sqrt{1 + | \nabla \psi(y_m) |^2} + \lambda \big) \big( d_{W_{i,j}}(\psi(y_m); a_i(\psi(y_m)), b_j(\psi(y_m)) ) + C\lambda \big) \\
    &= \int_{D} d_{W_{i,j}}(\psi(y); a_i(\psi(y)), b_j(\psi(y)))\sqrt{1 + |\nabla \psi(y)} \, d\H^{N-1}(y) \\
    &\qquad + C\lambda \int_D \sqrt{1 + |\psi(y)|^2} \dd y + \lambda \int_D d_{W_{i,j}}(\psi(y); a_i(\psi(y)), b_j(\psi(y))) \dd y + C \mathcal{H}^{N-1}(D) \lambda^2.
\end{align*}
By applying the area formula to the first term in the last equation, together with the fact that the other terms are bounded, we get
\begin{equation*}
    \int_\Sigma d_{W_{i,j}}(x; a_i(x), b_j(x)) \, d\H^{N-1}(x) + C \lambda \hno(\Sigma),
\end{equation*}
and we conclude by sending $\lambda \to 0$.}\\

{\textbf{Step 5: Interpolation in $S_\theta$.}
Let $\theta_2 \in (\theta_1, \theta)$. By the last steps, for each $\Sigma_l$ we found an open set $O_l$ that contains $\Sigma_l$ and such that $O_{l_1} \cap O_{l_2} = \emptyset$ for $l_1 \neq l_2$. Moreover, we constructed a sequence $u_n^l \in W^{1,2}(O_l)$ such that 
\[
    \limsup_{n \to \infty} \f_{\eps_n}(u_n^l, O_l) \leq\int_{\Sigma_l} d_{W}(y; u^-(y), u^+(y)) \, d\H^{N-1}(y) + C\lambda.
\]
Since $O_l$ are disjoint and 
\[
    u_n^l(x) = u(x)
\]
for $\text{dist}(x, J_u) > 2\ell_n$ by construction,
the function
\[
    \widetilde u_n(x) \coloneqq \begin{cases}
        u_n^l(x) & \text{ if }x \in O_l \\
        u(x) & \text{ if } \text{dist}(x, J_u) > 2\ell_n,
    \end{cases}
\]
yields an $W^{1,2}$ function on $O \coloneqq \bigcup_{l = 1}^{l_0} O_l \cup \Omega \setminus \overline{S_{\theta_1}}$.
Now, since $F$ and $\Omega_i$ have piecewise $C^2$ boundary, we can find a smooth function $\xi_n \in C^\infty(\Omega; [0,1])$  such that 
\[
    \xi_n = 0 \text{ on $J_u$ } \qquad \text{ and } \qquad \xi_n(x) = 1 \text{ if }\text{dist}(x, J_u) > \ell_n  
\]
and
\[
    |\nabla \xi_n| < \frac{C}{\eps_n}.
\]
Now, in $S_\theta$ we define 
\[
    u_n^\theta(x) = \xi_n(x)u(x).
\]
Immediately, we observe $u_n^\theta \in W^{1,2}(S_\theta; \R^M)$. Next, let $\zeta_\theta \in C_c^\infty(\Omega; [0,1])$ be another smooth cut-off such that 
\[
    \zeta_\theta(x) = 0 \text{ if  } x \in S_{\theta_2} \qquad \text{ and } \qquad \zeta_\theta(x) = 1 \text{ if }x \in \Omega\setminus S_{\theta} 
\]
and 
\[
    |\nabla \zeta_\theta| < \frac{C}{\theta - \theta_2}.
\]
Now, set 
\[
    u_n \coloneqq \zeta_\theta \widetilde u_n + (1- \zeta_\theta) u_n^\theta.
\]
Since $u_n$ is uniformly bounded, we infer
\begin{align*}
    \f_{\eps_n}(u_n, S_\theta) &\leq \frac{C}{\varepsilon_n} |\{\mathrm{dist}(x, J_u) < 2\ell_n\}|    \\
    & \hspace{2cm} + C\varepsilon_n \left( \frac{1}{(\theta - \theta_2)^2} + \sum_{i = 1}^k \left( \|\nabla a_i\|^2_{L^2} + \|\nabla b_i\|^2_{L^2} \right) \right) + \mathcal{F}_{\eps_n}(\widetilde u_n, S_\theta \cap O).
\end{align*}
Note that
\[
    \lim_{n \to \infty} \frac{1}{\varepsilon_n} |\{\mathrm{dist}(x, J_u \cap S_\theta) < 2\ell_n\}| = C_\lambda\H^{N-1}(J_u \cap S_\theta),
\]
since the boundaries of $\partial \Omega_i$ are transversal to $\partial \Omega$ in the sense of Definition \ref{def:piece_wise_C2_admissible}, by \eqref{eq:approximation_lemma_property_3} of Proposition \eqref{prop:approx_sets} and \eqref{eq:S_theta_no_jumps}. The constant $C_\lambda$ comes from Lemma \ref{lemma:reparameterization_2}.

Since  $J_u$ is piecewise $C^2$, using that $\widetilde u_n$ a recovery sequence in $O$ and that $d_W$ is bounded, for $n$ sufficiently large we obtain
\begin{align*}
    &\limsup_{n \to \infty} \left[ \frac{C}{\varepsilon_n}|\{\mathrm{dist}(x, J_u\cap S_\theta) < 2\ell_n\} | + C \varepsilon_n \left( \frac{1}{(\theta - \theta_2)^2} + \sum_{i = 1}^k \left( \|\nabla a_i\|^2_{L^2} + \|\nabla b_i\|^2_{L^2} \right) \right) + \mathcal{F}_{\varepsilon_n}(\widetilde u_n, S_\theta \cap O) \right] \\
    &\hspace{4cm}\leq C_\lambda \H^{N-1}(J_u \cap S_\theta) + C_\lambda \int_{J_u \cap S_\theta} d_W(y; u^-(y), u^+(y)) \, d\H^{N-1}(y) \\
    &\hspace{4cm}\leq C_\lambda \H^{N-1}(J_u \cap S_\theta) \\
    &\hspace{4cm}\leq C_\lambda \kappa.
\end{align*}
Recall that the last inequality follows from \eqref{eq:jump_set_vanishes_in_Sdelta}.
Finally, we observe again with \eqref{eq:jump_set_vanishes_in_Sdelta}
\begin{align*}
    \limsup_{n \to \infty} \f_{\eps_n}(u_n, \Omega) & \leq \limsup_{n \to \infty} \sum_{l = 1}^{l_0} \f_{\eps_n}(u_n^l, O_l) + C\H^{N-1}(J_u \cap S_\theta) \\
    &\leq \sum_{l = 1}^{l_0} \int_{\Sigma_l} d_{W}(y; u^-(y), u^+(y)) \, d\H^{N-1}(y) + C\lambda + C_\lambda \kappa \\
    &\leq \int_{J_u} d_{W}(y; u^-(y), u^+(y)) \, d\H^{N-1}(y) + C\lambda + C_\lambda \kappa.
\end{align*}
Choosing a diagonal sequence, first sending $\kappa \to 0$ and then $\lambda \to 0$, we conclude.}\\

\textbf{Step 6: General case.} Take $u \in BV(\Omega; \{a,b\})$. By Proposition \ref{prop:approx_sets}, we know that there exists a sequence of functions $\{ v_n \}_n \subset BV(\Omega; \{a,b\})$ with piecewise $C^2$ jump set, such that 
$$ v_n \to u \text{ in } L^1(\Omega; \R^M) \text{ as } n \to \infty, \qquad \lim_{n\to \infty} \mathcal{F}_\infty(v_n) = \mathcal{F}_\infty (u). $$
We can therefore apply the previous steps to the function $v_n$ to obtain a sequence of functions $\{w_m^n\}_m \subset W^{1,2}(\Omega; \R^M)$ such that  
$$ w_m^n \to v_n \text{ in } L^1(\Omega; \R^M) \text{ as } m \to \infty, \qquad \lim_{m\to \infty} \mathcal{F}_{\varepsilon_m}(w_m^n) = \mathcal{F}_\infty (v_n). $$
Using a diagonalization argument, we find a sequence $\{m_n\}_n$ such that the sequence of functions defined by $\{u_n\}_n = \{w_{m_n}^n\}_n$ has the following properties
$$ u_n \to u \text{ in } L^1(\Omega; \R^M) \text{ as } n \to \infty, \qquad \lim_{n\to \infty} \mathcal{F}_{\varepsilon_n} (u_n) = \mathcal{F}_\infty (u), $$
and we conclude.
\end{proof}


\appendix
\section{Auxiliary results}

Finally, we prove two technical results that we will need in the proof of the liminf inequality (see Theorem \ref{thm:liminf}).
{
\begin{lemma}\label{lemma:technical}
        Suppose that $g: \R^N \times \R^M \times \R^M \to [0, \infty)$ is continuous and for every compact set $K_0 \subset \R^N$ and $\lambda > 0$ we have that
        \begin{align}\label{eq:coercivity_aux}
            \{ (u, v) \in \R^M \times \R^M : \exists \, x \in K_0: g(x, u, v) < \lambda \}
        \end{align}
        is bounded.
        Furthermore, let $p \in \R^N, \{p_n\} \subset \R^N$, $\{l_n^1, l_n^2\} \subset \R^M$ and $L_n^1, L_n^2 \in \R^{M \times M}$ with $p_n \to p$ in $\R^N$, $l_n^1, l_n^2 \to 0$ in $\R^M$ and  $L_n^1, L_n^2 \to Id_M$ in $\R^{M \times M}$. Then,
        \[
            \inf_{r \in \R^M} g(p_n, l_n^1 + L_n^1 r, l_n^2 + L_n^2 r ) \xrightarrow{n \to \infty} \inf_{r \in \R^M} g(p, r, r).
        \]
    \end{lemma}
    \begin{proof}
        Let $K_0 \coloneqq \overline{\{p_n: n \in \N\}}$, $\lambda_0 \coloneqq \sup_{x \in K_0} \inf_{r \in \R^M} g(x, r,r) < \infty$ and set 
        \[
            K_1 \coloneqq \{ (u, v) \in \R^M \times \R^M : x \in K_0, g(x, u, v) < \lambda_0 + 1 \}.
        \]
        Note that the coercivity condition \eqref{eq:coercivity_aux} implies that we can find a ball $B_{R_0}(0)$ such that $K_1 \subset B_{R_0}(0)$ and that 
        \[
            g(x, u, v) \geq \lambda_0 + 1
        \]
        for all $x \in K_0$ and $(u,v) \in B_{R_0}(0)^C$. In particular, since for $n \in \N$ sufficiently large
        \[
            |l_n^1 + L_n^1 r| + |l_n^2 + L_n^2 r| \geq |r| - 2
        \]
        holds for all $r \in \R^M$, we observe 
        \[
            g(p_n, l_n^1 + L_n^1 r, l_n^2 + L_n^2 r ) \geq \lambda_0 + 1 
        \]
        for $|r| \geq R_1$ for a suitable constant $R_1 > 0$. Choosing some $r_0 \in \R^M$ with $g(p, r_0, r_0) \leq \lambda_0 + 1/2$ we also observe for $n$ sufficiently large
        \[
             \inf_{r \in \R^M} g(p_n, l_n^1 + L_n^1 r, l_n^2 + L_n^2 r ) \leq g(p_n, l_n^1 + L_n^1 r_0, l_n^2 + L_n^2 r_0 ) < \lambda_0 + 1.
        \]
        In particular, we infer for $n$ sufficiently large
        \begin{align}\label{eq:equi_coercivity}
            \inf_{r \in \R^M} g(p_n, l_n^1 + L_n^1 r, l_n^2 + L_n^2 r ) = \inf_{r \in \overline{B_{R_1}(0)}} g(p_n, l_n^1 + L_n^1 r, l_n^2 + L_n^2 r ).
        \end{align}
        Defining now, 
        \[
            G_n(r) \coloneqq  g(p_n, l_n^1 + L_n^1 r, l_n^2 + L_n^2 r ),
        \]
        \[
            G(r) \coloneqq  g(p, r, r)
        \]
        we infer from \eqref{eq:equi_coercivity} the equi-coercivity of the functions $G_n$. Observe now that $G_n$ converges continuously converges to $G$. By the fundamental property of $\Gamma$-convergence, the infimum of $G_n$ converges to the infimum of $G$. This concludes the proof.
    \end{proof}

    \begin{lemma}\label{lemma:technical_approx_limits}
        Let $u \in BV(\Omega; \R^M)$, $0 \in J_u$ with $\nu_{0} = e_N$, and $F \subset (0,1)$ be a set with $|F| = 1$. Then, there exists a sequence $\{\rho_m\} \subset F$ with $\rho_m \to 0^+$ such that 
        \[
            \H^{N-1}(J_u \cap \partial (-\rho_m/2, \rho_m/2)^N) = 0,
        \]
        and for the rescaled trace $u|_{\partial Q_{\rho_m}(0)}$ restricted to $(-\rho_m/2, \rho_m/2)^{N-1}\times \{\pm \rho_m / 2\}$ we have 
        \[
            u(\rho_m \cdot, \pm \rho_m/2) \to u^\pm(x_0)
        \]
        in $L^1((-1/2, 1/2)^{N-1})$.
    \end{lemma}
    
    \begin{proof}
        Since $u^\pm(x_0)$ are one side approximate limits we have 
        \[
            \lim_{\rho \to 0^+} \int_{(-1/2, 1/2)^{N-1}\times (0,1/2) } |u(\rho x) - u^+(x_0)| \, dx = 0
        \]
        and
        \[
            \lim_{\rho \to 0^+} \int_{(-1/2, 1/2)^{N-1}\times (-1/2,0) } |u(\rho x) - u^-(x_0)| \, dx = 0.
        \]
        In particular, we have by Fubini's Theorem
        \[
            \lim_{\rho \to 0} \int_0^{1} \int_{(-1/2, 1/2)^{N-1}} |u(\rho x', \rho t/2) - u^+(x_0)| \, dx' \, dt= 0,
        \]
        where $u(\rho x', \rho t)$ agrees with the two-sided trace at $\partial (-1/2,1/2)^N$ of $u$ that exists since 
        \begin{align}\label{eq:eligible_rho}
            \H^{N-1}(J_u \cap \partial (-\rho/2, \rho/2)^{N}) = 0
        \end{align}
        for all but countably many $\rho > 0$. Now, observe that 
        \begin{align*}
            &\int_{1/2}^1 \int_{(-1/2, 1/2)^{N-1}}|u(\rho tx', \rho t/2) - u^+(0)| \, dx' \, dt \\ 
            &= \int_{1/2}^1 \frac{1}{t^{N-1}}\int_{(-t/2, t/2)^{N-1}}|u(\rho x', \rho t/2) - u^+(0)| \, dx' \, dt \\
            &\leq 2^{N-1}\int_{1/2}^1 \int_{(-t/2, t/2)^{N-1}}|u(\rho x', \rho t/2) - u^+(0)| \, dx' \, dt \\
            & \qquad \xrightarrow{\rho \to 0} 0.
        \end{align*}
        Now, repeating above arguments for the negative trace and after passing to a subsequence $\tilde \rho_m$ we observe that for almost all $t \in (1/2, 1)$
        \begin{align*}
            &\lim_{m \to \infty } \int_{(-1/2, 1/2)^{N-1}}|u(\tilde \rho_m tx', \tilde \rho_m t/2) - u^+(0)| \, dx' \\
            &\qquad = \lim_{m \to \infty } \int_{(-1/2, 1/2)^{N-1}}|u(\tilde \rho_m tx', -\tilde \rho_m t/2) - u^-(0)| \, dx' = 0
        \end{align*}
        holds. Choosing now $t$ such that $\rho_m := \tilde \rho_m t$ fulfils \eqref{eq:eligible_rho} for all $m$ (which is possible since \eqref{eq:eligible_rho} holds for almost all $\rho$), we conclude.
    \end{proof}
}


\subsection*{Acknowledgment}
The authors were partially supported under NWO-OCENW.M.21.336, MATHEMAMI - Mathematical Analysis of phase Transitions in HEterogeneous MAterials with Materials Inclusions (https://doi.org/10.61686/EYVFT65519).

We really appreciated the fact that the referees took the time to go through the details of the technical proofs and spot many places where we needed to improve the manuscript.


\bibliographystyle{siam}
\bibliography{Bibliography}



\end{document}

%% file: fig_geodesic.tex
\begin{figure}
    \centering
    \definecolor{geoviolet}{rgb}{0.627,0.125,0.941}
    \colorlet{geovioletText}{geoviolet!70!black}
    \colorlet{geogreenText}{green!45!black}
    \begin{tikzpicture}[x=0.78pt, y=0.78pt, font=\small]
        \coordinate (a)  at (0,0);
        \coordinate (b)  at (240,-64);
        \coordinate (P1) at (13.09,7.97);
        \coordinate (P2) at (15.43,-55.59);
        \coordinate (q)  at (303.27,56.04);
        \coordinate (c1) at (186.71,-86.08);
        \coordinate (c2) at (233.74,-6.65);
        \filldraw[fill=geoviolet, fill opacity=0.10, draw=black, line width=0.4pt] (a) circle (57.69);
        \filldraw[fill=geoviolet, fill opacity=0.10, draw=black, line width=0.4pt] (a) circle (35.78);
        \filldraw[fill=green, fill opacity=0.30, draw=black, line width=0.4pt] (b) circle (57.69);
        \filldraw[fill=white, fill opacity=0.30, draw=black, line width=0.4pt] (b) circle (32);
        \draw[orange!90!black, line width=0.8pt] (13.09,7.97)
            .. controls (32.87,25.62) and (35.32,16.71) .. (52.05,15.12)
            .. controls (68.78,13.53) and (99.79,19.25) .. (101.94,17.56)
            .. controls (104.09,15.87) and (77.39,6.77) .. (60.20,2.38)
            .. controls (43.01,-2.00) and (35.32,-1.68) .. (31.75,-2.63)
            .. controls (28.18,-3.58) and (28.73,-5.82) .. (30.39,-8.52)
            .. controls (32.05,-11.21) and (34.83,-14.37) .. (39.66,-19.30)
            .. controls (44.48,-24.23) and (51.34,-30.94) .. (52.84,-37.39)
            .. controls (54.34,-43.86) and (50.47,-50.07) .. (42.19,-51.70)
            .. controls (33.91,-53.34) and (21.21,-50.39) .. (15.00,-50.42)
            .. controls (8.79,-50.45) and (9.06,-53.45) .. (22.49,-58.79)
            .. controls (35.92,-64.12) and (62.52,-71.81) .. (81.43,-85.54)
            .. controls (100.34,-99.27) and (111.56,-119.05) .. (127.86,-119.14)
            .. controls (144.15,-119.22) and (165.51,-99.60) .. (185.05,-87.13)
            .. controls (204.58,-74.66) and (222.29,-69.33) .. (225.81,-66.66)
            .. controls (229.33,-64.00) and (218.67,-64.00) .. (208.78,-57.59)
            .. controls (198.89,-51.18) and (189.77,-38.36) .. (186.55,-30.10)
            .. controls (183.33,-21.84) and (186.00,-18.13) .. (195.89,-21.88)
            .. controls (205.78,-25.62) and (222.89,-36.81) .. (228.78,-39.74)
            .. controls (234.67,-42.67) and (229.33,-37.33) .. (229.17,-27.11)
            .. controls (229.01,-16.88) and (234.02,-1.76) .. (242.48,6.92)
            .. controls (250.96,15.60) and (262.89,17.83) .. (270.54,26.76)
            .. controls (278.18,35.68) and (281.53,51.29) .. (303.27,56.04);
        \draw[red, line width=0.6pt] (P1) -- (P2);
        \draw[red, line width=0.6pt] (c1) -- (c2);
        \foreach \p in {a,b,P1,P2,q,c1,c2} \fill (\p) circle (1.8);
        \node[below left=-1pt] at (a) {$a$};
        \node[below right=-1pt] at (b) {$b$};
        \node[right=1pt] at (q) {$q$};
        \node[above left=-1pt] at (P1) {$a+P_1$};
        \node[anchor=north east, inner sep=1pt] at ($(P2)+(-3,-4)$) {$a+P_2$};
        \node at (101,-90) {$\varphi_m$};
        \node[anchor=base east] at (10,-28) {$\theta_{1,m}$};
        \node[anchor=base west] at (203,-98) {$\theta_{2,m}$};
        \node[anchor=base west, text=geovioletText] at (48,38) {$\mathcal{T}^{1,m}$};
        \node[anchor=base west, text=geogreenText] at (281,-18) {$\mathcal{T}^{2,m}$};
        \node[anchor=base east] at (-20,-95) {$\bigcup_{k\geq m}\mathcal{T}^{1,k}$};
        \node[anchor=base west] at (306,-100) {$\bigcup_{k\geq m}\mathcal{T}^{2,k}$};
    \end{tikzpicture}
    \caption{The substitution of $\widetilde \varphi_m$ in Theorem \ref{existence-of-geodesics-with-bounded-path-length}.}
    \label{fig:geodesic}
\end{figure}

%% file: fig_steps_approx.tex
\begin{figure}[htbp]
\centering
\begin{tikzpicture}

\newcommand{\OmegaClip}{\clip[rounded corners=0.5cm] (0,0) rectangle (7,4.5);}
\newcommand{\OmegaBoundary}{\draw[line width=0.6pt, rounded corners=0.5cm] (0,0) rectangle (7,4.5);}
\newcommand{\PartitionFill}{%
  \begin{scope}
    \OmegaClip
    \fill[fillJ] (-1,-1) rectangle (8,5.5);
    \fill[fillI] (-1,-1) -- (3,-1) -- (3,3) -- (-1,3.8) -- cycle;
    \fill[fillL] (-1,3.8) -- (3,3) -- (4.6667,5.5) -- (-1,5.5) -- cycle;
  \end{scope}}
\newcommand{\Tlines}{%
  \draw[iface] (3,0) -- (3,3) -- (0,3.6);
  \draw[iface] (3,3) -- (4,4.5);}
\newcommand{\EonePath}{(3,0.5) -- (3,1.4) .. controls (2.4,1.75) and (1.6,1.75) .. (1.2,1.4)
  .. controls (1.3,0.5) and (1.5,0) .. (1.9,0) .. controls (2.4,0) and (2.8,0.2) .. (3,0.5) -- cycle}
\newcommand{\EtwoPath}{(3,1.9) -- (3,2.6) .. controls (3.8,3.0) and (4.9,2.9) .. (5.1,2.3)
  .. controls (5.3,1.6) and (4.0,1.4) .. (3,1.9) -- cycle}
\newcommand{\Ehatched}{%
  \fill[pattern=north east lines, pattern color=black!35] \EonePath;
  \fill[pattern=north east lines, pattern color=black!35] \EtwoPath;
  \draw[black!75, line width=0.9pt] \EonePath;
  \draw[black!75, line width=0.9pt] \EtwoPath;}
\newcommand{\EaOnePath}{plot[smooth cycle, tension=0.6] coordinates {(2.88,0.65) (2.88,1.0) (2.86,1.3)
  (2.4,1.55) (1.7,1.5) (1.36,1.3) (1.4,0.7) (1.6,0.22) (1.95,0.11) (2.5,0.25) (2.8,0.48)}}
\newcommand{\EaTwoPath}{plot[smooth cycle, tension=0.6] coordinates {(3.12,2.05) (3.12,2.42)
  (3.4,2.6) (4.2,2.7) (4.9,2.3) (4.65,1.88) (4.0,1.75) (3.4,1.86)}}
\newcommand{\PslabPath}{(2.72,0.58) rectangle (3,1.32)}
\newcommand{\NslabPath}{(3,1.98) rectangle (3.28,2.52)}
\newcommand{\Sballs}[2]{%
  \begin{scope}
    \OmegaClip
    \foreach \x/\y in {3/3, 3/0, 0/3.6, 4/4.5, 3/0.5, 3/1.4, 1.2/1.4, 3/1.9, 3/2.6, 1.9/0}{
      \filldraw[fill=#1, draw=#2, dotted, line width=0.8pt] (\x,\y) circle (0.17);
    }
  \end{scope}}
\newcommand{\Sdots}{%
  \foreach \x/\y in {3/3, 3/0, 0/3.6, 4/4.5, 3/0.5, 3/1.4, 1.2/1.4, 3/1.9, 3/2.6, 1.9/0}{
    \fill[cNone] (\x,\y) circle (1.2pt);
  }}
\newcommand{\OmegaLabels}{%
  \node[font=\small] at (0.5,0.42) {$\Omega_1$};
  \node[font=\small] at (1.0,4.1) {$\Omega_3$};
  \node[font=\small] at (6.45,0.42) {$\Omega_2$};}
\newcommand{\BasePic}{\PartitionFill \Ehatched \Tlines \OmegaBoundary \OmegaLabels}

\begin{scope}[shift={(0,6.3)}]
  \node[paneltitle] at (-0.1,4.85) {(a)\ \ Step 1: isolation of singularities};
  \BasePic
  \Sballs{cNone!12}{cNone}
  \Sdots
  \node[lab, font=\small] at (2.15,0.85) {$E$};
  \node[lab, font=\small] at (4.3,2.25) {$E$};
  \node[font=\small] at (1.35,3.62) {$A$};
  \draw[leader] (1.07,1.5) -- (0.55,2.25);
  \node[font=\small, anchor=south] at (0.55,2.22) {$S$};
  \draw[leader] (3.12,3.07) -- (3.85,3.55);
  \node[font=\small, anchor=west] at (3.82,3.58) {$S_A$};
  \draw[leader] (3.02,-0.06) -- (3.6,-0.55);
  \node[font=\footnotesize, anchor=west] at (3.55,-0.6) {$\partial\Omega\cap\overline A$};
  \draw[leader] (1.88,-0.06) -- (1.3,-0.55);
  \node[font=\footnotesize, anchor=east] at (1.35,-0.6) {$\partial\Omega\cap\overline{\partial E}$};
\end{scope}

\begin{scope}[shift={(8,6.3)}]
  \node[paneltitle] at (-0.1,4.85) {(b)\ \ Step 2: the contact set $D=P\cup N$};
  \BasePic
  \Sballs{cNone!5}{cNone!40}
  \draw[auxset] (3,0.95) ellipse (0.35 and 0.37);
  \draw[auxset] (3,2.25) ellipse (0.35 and 0.27);
  \draw[goodline] (3,0.67) -- (3,1.23);
  \draw[goodline] (3,2.07) -- (3,2.43);
  \draw[errline] (3,0.58) -- (3,0.67);
  \draw[errline] (3,1.23) -- (3,1.32);
  \draw[errline] (3,1.98) -- (3,2.07);
  \draw[errline] (3,2.43) -- (3,2.52);
  \draw[-latex, line width=0.6pt] (3,0.8) -- (3.5,0.8) node[right, font=\scriptsize, inner sep=1pt] {$\nu_A$};
  \draw[-latex, line width=0.6pt, black!60] (3,1.1) -- (3.5,1.1) node[right, font=\scriptsize, inner sep=1pt, text=black] {$\nu_E$};
  \node[font=\scriptsize, anchor=west] at (4.15,0.95) {$D^+$: $\nu_E=\nu_A$};
  \draw[-latex, line width=0.6pt] (3,2.15) -- (3.5,2.15) node[right, lab, font=\scriptsize, inner sep=1pt] {$\nu_A$};
  \draw[-latex, line width=0.6pt, black!60] (3,2.35) -- (2.5,2.35) node[left, font=\scriptsize, inner sep=1pt, text=black] {$\nu_E$};
  \node[font=\scriptsize, anchor=west] at (0.15,2.62) {$D^-$: $\nu_E=-\nu_A$};
  \node[lab, font=\scriptsize, anchor=east] at (2.62,0.95) {$P$};
  \node[font=\scriptsize, anchor=north west, inner sep=1pt] at (3.3,0.66) {$U$};
  \node[font=\scriptsize, anchor=east] at (2.62,2.1) {$N$};
  \node[lab, font=\scriptsize, anchor=north west, inner sep=1pt] at (3.32,2.03) {$V$};
  \node[font=\footnotesize] at (3.5,-0.42)
    {$\mathcal H^{N-1}(P\setminus D^+)+\mathcal H^{N-1}(N\setminus D^-)<\alpha$};
\end{scope}

\begin{scope}[shift={(0,0)}]
  \node[paneltitle] at (-0.1,4.85) {(c)\ \ Step 3: slabs $P(s_0)$, $N(s_0)$ and $R$};
  \PartitionFill
  \Ehatched
  \filldraw[fill=cGood!30, draw=cGood!70!black, line width=0.6pt] \PslabPath;
  \filldraw[fill=cGood!30, draw=cGood!70!black, line width=0.6pt] \NslabPath;
  \Tlines
  \OmegaBoundary
  \OmegaLabels
  \Sballs{cNone!12}{cNone}
  \draw[latex-latex, line width=0.4pt] (2.72,1.0) -- (3,1.0);
  \node[lab, font=\scriptsize, anchor=east] at (2.68,1.0) {$s_0$};
  \draw[-latex, line width=0.6pt] (2.7,0.76) -- (2.25,0.76) node[left, lab, font=\scriptsize, inner sep=1pt] {$-\nu_A$};
  \draw[-latex, line width=0.6pt] (3.3,2.25) -- (3.75,2.25) node[right, lab, font=\scriptsize, inner sep=1pt] {$+\nu_A$};
  \draw[leader] (2.86,1.18) -- (4.05,0.95);
  \draw[leader] (3.2,2.08) -- (4.05,1.05);
  \node[font=\footnotesize, anchor=west, align=left] at (4.05,1.0)
    {$R=\big(P(s_0)\cup N(s_0)\big)$\\[-1pt]$\qquad\setminus\overline S$};
  \node[font=\small] at (1.35,3.62) {$A$};
\end{scope}

\begin{scope}[shift={(8,0)}]
  \node[paneltitle] at (-0.1,4.85) {(d)\ \ Step 4: $G$ and $F=G\cup P(s_0)\cup N(s_0)$};
  \PartitionFill
  \draw[cGood!70!black, line width=2.4pt] \EaOnePath;
  \draw[cGood!70!black, line width=2.4pt] \EaTwoPath;
  \draw[cGood!70!black, line width=2.4pt] \PslabPath;
  \draw[cGood!70!black, line width=2.4pt] \NslabPath;
  \fill[cGood!18] \EaOnePath;
  \fill[cGood!18] \EaTwoPath;
  \fill[cGood!18] \PslabPath;
  \fill[cGood!18] \NslabPath;
  \draw[black!45, dashed, line width=0.5pt] \EonePath;
  \draw[black!45, dashed, line width=0.5pt] \EtwoPath;
  \Tlines
  \OmegaBoundary
  \OmegaLabels
  \draw[ealine] \EaOnePath;
  \draw[ealine] \EaTwoPath;
  \draw[goodline] (3,0.58) -- (3,1.32);
  \draw[goodline] (3,1.98) -- (3,2.52);
  \node[font=\small] at (2.05,0.85) {$F$};
  \node[font=\small] at (4.2,2.25) {$F$};
  \draw[leader] (3.04,1.15) -- (4.0,1.15);
  \node[font=\footnotesize, anchor=west] at (3.98,1.15) {$\partial F\cap A=P\cup N$};
  \draw[leader] (4.78,1.86) -- (5.35,0.62);
  \node[font=\footnotesize, anchor=west, text=cEal] at (5.3,0.55) {$\partial G$};
  \draw[leader] (5.0,2.72) -- (5.55,3.3);
  \node[font=\footnotesize, anchor=west, text=black!60] at (5.5,3.35) {$\partial E$};
  \node[font=\footnotesize] at (3.5,-0.42)
    {$d_H(E,G)<\alpha,\qquad \mathcal H^{N-1}\big((F\triangle E)\cap A\big)\le 3\alpha$};
\end{scope}

\begin{scope}[shift={(0.2,-1.35)}, font=\footnotesize]
  \draw[iface] (0,0) -- (0.8,0);
  \node[anchor=west] at (0.95,0) {interface $A=\Omega\cap\bigcup_i\partial\Omega_i$};
  \filldraw[pattern=north east lines, pattern color=black!35, draw=black!75, line width=0.9pt] (0,-0.62) rectangle (0.8,-0.38);
  \node[anchor=west] at (0.95,-0.5) {set $E$ (hatched)};
  \filldraw[fill=cNone!12, draw=cNone, dotted, line width=0.8pt] (0.4,-1.0) circle (0.13);
  \fill[cNone] (0.4,-1.0) circle (1.2pt);
  \node[anchor=west] at (0.95,-1.0) {balls of $S$ and singular points ($\Sigma_1$)};
  \draw[goodline] (0,-1.5) -- (0.8,-1.5);
  \node[anchor=west] at (0.95,-1.5) {$D^\pm$; also $\partial F\cap A=P\cup N$};
  \draw[errline] (0,-2.0) -- (0.8,-2.0);
  \node[anchor=west] at (0.95,-2.0) {$P\setminus D^+$, $N\setminus D^-$};
  \draw[auxset] (8,0) -- (8.8,0);
  \node[anchor=west] at (8.95,0) {open sets $U$, $V$ ($P=U\cap A$, $N=V\cap T$)};
  \filldraw[fill=cGood!30, draw=cGood!70!black, line width=0.6pt] (8,-0.62) rectangle (8.8,-0.38);
  \node[anchor=west] at (8.95,-0.5) {slabs $P(s_0)$, $N(s_0)$};
  \draw[ealine] (8,-1.0) -- (8.8,-1.0);
  \node[anchor=west] at (8.95,-1.0) {$\partial G$};
  \filldraw[fill=cGood!18, draw=cGood!70!black, line width=1.2pt] (8,-1.62) rectangle (8.8,-1.38);
  \node[anchor=west] at (8.95,-1.5) {$F=G\cup P(s_0)\cup N(s_0)$};
  \draw[black!45, dashed, line width=0.5pt] (8,-2.0) -- (8.8,-2.0);
  \node[anchor=west] at (8.95,-2.0) {original $\partial E$ (panel (d))};
\end{scope}

\end{tikzpicture}

\caption{The main construction carried out in the approximation.}
\label{fig:approx_lemma}
\end{figure}

%% file: fig_nice_set_liminf_blowup.tex
\definecolor{cGood}{RGB}{27,120,55}     
\definecolor{cNone}{RGB}{200,30,30}     
\definecolor{cMulti}{RGB}{230,120,0}    
\definecolor{cOmegaL}{RGB}{118,42,131}  
\colorlet{fillI}{blue!9}                
\colorlet{fillJ}{black!11}              
\colorlet{fillL}{yellow!40}             

\tikzset{
  good/.style ={cGood,  line width=0.8pt},
  none/.style ={cNone,  line width=1.1pt, dotted},
  multi/.style={cMulti, line width=0.8pt, dashed},
  omegal/.style={cOmegaL,line width=0.8pt, dash dot},
  iface/.style={black, line width=1.1pt},
  lab/.style={fill=white, fill opacity=0.85, text opacity=1, inner sep=1.5pt, rounded corners=1pt},
}

\def\interface{%
  (-2.3,3) .. controls (-2.3,1.6) and (-1.9,0.6) .. (-1.2,0.25)
  .. controls (-0.6,0.05) and (-0.3,0) .. (0,0)
  .. controls (0.4,0) and (1.2,0.08) .. (1.6,0.12)
  .. controls (1.85,0.13) and (1.9,0.25) .. (1.9,0.45)
  .. controls (1.9,0.68) and (1.78,0.8) .. (1.5,0.8)
  .. controls (1.2,0.8) and (0.9,0.82) .. (0.9,1.02)
  .. controls (0.9,1.2) and (1.1,1.25) .. (1.4,1.25)
  .. controls (2.0,1.25) and (2.6,1.1) .. (3,1.05)}

\begin{tikzpicture}[>=Stealth, scale=0.85]
\fill[fillJ] (-3,-3) rectangle (3,3);
\fill[fillI] \interface -- (3,3) -- cycle;
\begin{scope}
  \clip \interface -- (3,3) -- cycle;
  \fill[fillL] (-1.45,-3) rectangle (-0.55,3);
  \draw[iface] (-1.45,-3) -- (-1.45,3);
  \draw[iface] (-0.55,-3) -- (-0.55,3);
\end{scope}

\draw[black!55, dashed, line width=0.5pt] (-3,0) -- (3,0);

\foreach \x in {-2.75,-2.5} \draw[none]  (\x,-3) -- (\x,3);
\foreach \x in {-1.25,-0.75} \draw[omegal] (\x,-3) -- (\x,3);
\foreach \x in {1.2,1.6} \draw[multi] (\x,-3) -- (\x,3);


\path[name path=J] \interface;   
\foreach \x/\n in {-2.0/g1,-1.7/g2,-0.3/g3,0.35/g4,2.3/g5,2.7/g6}{
  \draw[good, name path=slice] (\x,-3) -- (\x,3);
  \path[name intersections={of=J and slice, by=p\n}];
}
\draw[iface] \interface;          
\foreach \n in {g1,g2,g3,g4,g5,g6}
  \fill[cGood!70!black] (p\n) circle (1.6pt);

\draw[line width=0.6pt] (-3,-3) rectangle (3,3);
\draw[line width=1.6pt] (-3,3) -- (3,3);
\draw[line width=1.6pt] (-3,-3) -- (3,-3);
\node[font=\small, above] at (0,3.02) {trace $\to a_i(x_0)$ in $L^1$};
\node[font=\small, below] at (0,-3.02) {trace $\to a_j(x_0)$ in $L^1$};

\fill (0,0) circle (1.9pt);
\node[lab, font=\small, anchor=north west] at (0.04,-0.08) {$x_0$};
\draw[->, line width=0.9pt] (0,0) -- (0,0.85);
\node[lab, font=\footnotesize, anchor=south] at (0,0.88) {$\nu(x_0)=e_N$};

\node[lab, font=\small, align=center] at (1.05,2.15) {$\Omega_i$\\[-1pt]$u\approx a_i(x_0)$};
\node[lab, font=\small, align=center] at (1.05,-1.75) {$\Omega_j$\\[-1pt]$u\approx a_j(x_0)$};
\node[lab, font=\footnotesize, align=center] at (-1.0,2.15) {$\Omega_l$\\[-1pt]\scriptsize $l\neq i,j$};

\draw[decorate, decoration={brace, amplitude=5pt}, line width=0.6pt] (-3.15,0.04) -- (-3.15,2.96)
  node[midway, left=6pt, font=\small] {$Q^+(\rho_m)$};
\draw[decorate, decoration={brace, amplitude=5pt}, line width=0.6pt] (-3.15,-2.96) -- (-3.15,-0.04)
  node[midway, left=6pt, font=\small] {$Q^-(\rho_m)$};
\node[font=\footnotesize, anchor=west] at (3.05,2.88) {$t=\rho_m/2$};
\node[font=\footnotesize, anchor=west] at (3.05,-2.88) {$t=-\rho_m/2$};
\node[font=\footnotesize, anchor=west] at (3.05,0) {$t=0$};

\def\by{-4.2}
\draw[->, line width=0.5pt] (-3.3,\by) -- (3.5,\by) node[right, font=\small] {$x'$};
\draw[none,  line width=3pt] (-3,\by) -- (-2.3,\by);
\draw[good,  line width=3pt] (-2.3,\by) -- (-1.45,\by);
\draw[omegal,line width=3pt] (-1.45,\by) -- (-0.55,\by);
\draw[good,  line width=3pt] (-0.55,\by) -- (0.9,\by);
\draw[multi, line width=3pt] (0.9,\by) -- (1.9,\by);
\draw[good,  line width=3pt] (1.9,\by) -- (3,\by);
\foreach \x in {-3,3} \draw[line width=0.5pt] (\x,\by-0.12) -- (\x,\by+0.12);
\node[font=\footnotesize, below] at (-3,\by-0.1) {$-\rho_m/2$};
\node[font=\footnotesize, below] at (3,\by-0.1) {$\rho_m/2$};
\node[font=\small, anchor=east] at (-3.35,\by) {$Q'(\rho_m)$};
\foreach \x in {-1.875,0.175,2.45}
  \node[font=\scriptsize, text=cGood!80!black] at (\x,\by+0.3) {$M_m$};
\node[font=\scriptsize\itshape, text=black!60] at (0,\by-1.0) {deviations from the flat interface exaggerated};

\begin{scope}[legend/.style={anchor=north west, align=left, font=\footnotesize, inner sep=0pt}]
  \node[legend] (L1) at (5.75,2.9) {good slice, $x'\in M_m$:\\ one jump, $\Omega_j$ below,\\ $\Omega_i$ above};
  \node[legend] (L2) at ([yshift=-8pt]L1.south west) {no jump:\\ $\mathcal H^0(J_{u_{x'}})=0$};
  \node[legend] (L3) at ([yshift=-8pt]L2.south west) {$\ge 2$ jumps:\\ $\mathcal H^0(J_{u_{x'}})\ge 2$};
  \node[legend] (L4) at ([yshift=-8pt]L3.south west) {one jump, but crosses\\ $\Omega_l$, $l\neq i,j$ \ ($N_m$)};
  \foreach \n/\st in {L1/good,L2/none,L3/multi,L4/omegal}
    \draw[\st] ([xshift=-24pt, yshift=-0.9ex]\n.north west) -- ([xshift=-4pt, yshift=-0.9ex]\n.north west);
\end{scope}

\end{tikzpicture}

%% file: fig_nice_set_liminf_measure.tex
\definecolor{cGood}{RGB}{27,120,55}     
\definecolor{cNone}{RGB}{200,30,30}     
\definecolor{cMulti}{RGB}{230,120,0}    
\definecolor{cOmegaL}{RGB}{118,42,131}  
\colorlet{fillI}{blue!9}                
\colorlet{fillJ}{black!11}              
\colorlet{fillL}{yellow!40}             

\tikzset{
  good/.style ={cGood,  line width=0.8pt},
  none/.style ={cNone,  line width=1.1pt, dotted},
  multi/.style={cMulti, line width=0.8pt, dashed},
  omegal/.style={cOmegaL,line width=0.8pt, dash dot},
  iface/.style={black, line width=1.1pt},
  lab/.style={fill=white, fill opacity=0.85, text opacity=1, inner sep=1.5pt, rounded corners=1pt},
}

\begin{tikzpicture}
  \node[anchor=west, font=\small, inner xsep=0pt] at (0,1.75) {$\mathcal H^{N-1}(Q'(\rho_m))=\rho_m^{N-1}$
    \ \footnotesize(upper bounds; lengths schematic)};
  \def\bh{0.45}
  \fill[cGood!45]   (0,0)     rectangle (4.165,\bh);
  \fill[cNone!40]   (4.165,0) rectangle (4.41,\bh);
  \fill[cMulti!45]  (4.41,0)  rectangle (4.9,\bh);
  \fill[cOmegaL!35] (4.9,0)   rectangle (5.88,\bh);
  \draw[good,  line width=1.2pt] (0.1,\bh/2) -- (4.065,\bh/2);
  \draw[none,  line width=1.4pt] (4.2,\bh/2) -- (4.375,\bh/2);
  \draw[multi, line width=1.2pt] (4.45,\bh/2) -- (4.86,\bh/2);
  \draw[omegal,line width=1.2pt] (4.94,\bh/2) -- (5.84,\bh/2);
  \foreach \x in {4.165,4.41,4.9} \draw[line width=0.6pt] (\x,0) -- (\x,\bh);
  \draw[line width=0.6pt] (0,0) rectangle (5.88,\bh);
  \node[font=\small, anchor=north] at (2.08,-0.08) {$M_m$:\ \ $\textcolor{cGood}{\mathcal H^{N-1}(M_m)}\ge(1-\delta)\,\rho_m^{N-1}$};
  \draw[cNone, line width=0.6pt]   (4.29,\bh) -- (4.29,1.1) -- (5.93,1.1);
  \node[font=\small, anchor=west, inner xsep=2pt] at (5.97,1.1)
    {$\mathcal H^0(J_{u_{x'}})=0$: $\textcolor{cNone}{\hno\text{-measure}}\le\tfrac{\delta}{8}\,\rho_m^{N-1}$};
  \draw[cOmegaL, line width=0.6pt] (5.88,\bh/2) -- (6.18,\bh/2);
  \node[font=\small, anchor=west, inner xsep=2pt] at (6.25,\bh/2)
    {$N_m$ (one jump, crosses $\Omega_l$): $\textcolor{cOmegaL}{\hno\text{-measure}}\le\tfrac{\delta}{2}\,\rho_m^{N-1}$};
  \draw[cMulti, line width=0.6pt]  (4.655,0) -- (4.655,-0.65) -- (5.93,-0.65);
  \node[font=\small, anchor=west, inner xsep=2pt] at (5.97,-0.65)
    {$\mathcal H^0(J_{u_{x'}})\ge 2$: $\textcolor{cMulti}{\hno\text{-measure}}\le\tfrac{\delta}{4}\,\rho_m^{N-1}$};
\end{tikzpicture}

%% file: fig_global.tex
\begin{figure}
    \centering
    \begin{tikzpicture}[scale=0.85,
        declare function={Tx(\t)=5+0.3*sin(50*\t);}]
        \coordinate (P3) at ({Tx(2)},2);
        \coordinate (P2) at ({Tx(3.5)},3.5);
        \coordinate (P1) at ({Tx(4.5)},4.5);
        \fill[red!12, even odd rule, rounded corners=25pt] (0,0) rectangle (10,6.4)
            [rounded corners=17pt] (0.45,0.45) rectangle (9.55,5.95);
        \fill[gray!22]
            (P1) .. controls (3.6,5.6) and (1.4,5.1) .. (1.6,3.4)
                 .. controls (1.8,2.0) and (3.6,1.4) .. (P3)
            -- plot[domain=2:3.5, samples=20, variable=\t] ({Tx(\t)},\t)
            -- (P2) .. controls (8.4,2.6) and (8.4,5.5) .. (P1) -- cycle;
        \draw[green!55!black, line width=1.6pt] plot[domain=0:2, samples=30, variable=\t] ({Tx(\t)},\t);
        \draw[green!55!black, line width=1.6pt] plot[domain=4.5:6.4, samples=30, variable=\t] ({Tx(\t)},\t);
        \draw[violet, line width=2.2pt] plot[domain=2:3.5, samples=30, variable=\t] ({Tx(\t)},\t);
        \draw[orange!90!black, line width=1.6pt] plot[domain=3.5:4.5, samples=20, variable=\t] ({Tx(\t)},\t);
        \draw[blue!70!black, line width=1.3pt]
            (P1) .. controls (3.6,5.6) and (1.4,5.1) .. (1.6,3.4)
                 .. controls (1.8,2.0) and (3.6,1.4) .. (P3);
        \draw[blue!70!black, line width=1.3pt]
            (P2) .. controls (8.4,2.6) and (8.4,5.5) .. (P1);
        \foreach \P in {P1,P2,P3}{
            \fill[red!45, opacity=0.55] (\P) circle (0.38);
            \fill (\P) circle (1.5pt);
        }
        \draw[thick, rounded corners=25pt] (0,0) rectangle (10,6.4);
        \draw[red!60!black, dashed, rounded corners=17pt] (0.45,0.45) rectangle (9.55,5.95);
        \node at (1.5,1.2) {$\Omega_1$};
        \node at (8.6,1.2) {$\Omega_2$};
        \node at (2.9,3.5) {$F$};
        \node at (2.9,3.0) {\scriptsize $u=a_1$};
        \node at (6.9,4.0) {\scriptsize $u=a_2$};
        \node at (2.6,0.95) {\scriptsize $u=b_1$};
        \node at (7.6,0.95) {\scriptsize $u=b_2$};
        \node[right] at (10.05,6.0) {$\partial\Omega$};
        \draw[-{Stealth[length=1.5mm]}] (10.6,5.2) node[right] {\scriptsize $\Omega\setminus\Omega_\theta\subset S_\theta$} -- (9.75,5.0);
        \draw[-{Stealth[length=1.5mm]}] (7.8,2.4) node[right] {\scriptsize $S_\theta$} -- ($(P3)+(0.33,0.05)$);
        \draw[-{Stealth[length=1.5mm]}] (7.8,2.4) -- ($(P2)+(0.35,-0.1)$);
        \draw[-{Stealth[length=1.5mm]}] (7.8,2.4) -- ($(P1)+(0.36,-0.12)$);
        \node[left] at ({Tx(5.6)-0.05},5.6) {\scriptsize $A$};
        \begin{scope}[shift={(0,-0.6)}]
            \draw[blue!70!black, line width=1.3pt] (0,0) -- (0.6,0);
            \node[right] at (0.65,0) {\scriptsize $\partial F\cap\Omega_l$: \ $\dd_{W_l}(x;a_l(x),b_l(x))$};
            \draw[violet, line width=2.2pt] (0,-0.5) -- (0.6,-0.5);
            \node[right] at (0.65,-0.5) {\scriptsize $\mathrm{int}_{\partial F}(\partial F\cap\partial\Omega_1\cap\partial\Omega_2)$: \ $\dd_{W_{1,2}}(x;a_1(x),b_2(x))$};
            \draw[orange!90!black, line width=1.6pt] (0,-1.0) -- (0.6,-1.0);
            \node[right] at (0.65,-1.0) {\scriptsize $F\cap\partial\Omega_1\cap\partial\Omega_2$: \ $\dd_{W_{1,2}}(x;a_1(x),a_2(x))$};
            \draw[green!55!black, line width=1.6pt] (0,-1.5) -- (0.6,-1.5);
            \node[right] at (0.65,-1.5) {\scriptsize $(F^c)^\circ\cap\partial\Omega_1\cap\partial\Omega_2$: \ $\dd_{W_{1,2}}(x;b_1(x),b_2(x))$};
            \fill[red!45, opacity=0.55] (0.3,-2.0) circle (0.13);
            \node[right] at (0.65,-2.0) {\scriptsize $S_\theta$: crude interpolation};
        \end{scope}
    \end{tikzpicture}
    \caption{The decomposition of the jump set of $u=a\ca_F+b(1-\ca_F)$ for two materials $\Omega_1,\Omega_2$ separated by the interface $A$. 
    }
\end{figure}

%% file: fig_tubular.tex
\begin{figure}
    \centering
    \begin{tikzpicture}[
        scale=1.0,
        declare function={
            cx(\s)=1.6*\s;
            cy(\s)=0.9*sin(\s r);
            nn(\s)=sqrt(2.56+0.81*cos(\s r)^2);
            nx(\s)=-0.9*cos(\s r)/nn(\s);
            ny(\s)=1.6/nn(\s);
        }]
        \def\h{0.6}
        \def\l{0.13}
        \def\d{0.07}
        \def\smin{0.2}
        \def\smax{6.2}
        \def\yD{2.7}
        \draw[thick] ({cx(\smin)},\yD) -- ({cx(\smax)},\yD);
        \node[right] at ({cx(\smax)},\yD) {$D\subset\R^{N-1}$};
        \foreach \k in {0,...,10}{
            \pgfmathsetmacro{\sb}{\smin+0.6*\k}
            \draw ({cx(\sb)},\yD-0.08) -- ({cx(\sb)},\yD+0.08);
        }
        \foreach \k in {0,...,9}{
            \pgfmathsetmacro{\sc}{\smin+0.3+0.6*\k}
            \fill[blue!70!black] ({cx(\sc)},\yD) circle (1.6pt);
        }
        \node[below] at ({cx(1.7)},\yD-0.05) {\scriptsize $y_m$};
        \draw[decorate,decoration={brace,amplitude=4pt}]
            ({cx(1.4)},\yD+0.15) -- node[above=4pt] {\scriptsize $Q'(y_m,r_n)$} ({cx(2.0)},\yD+0.15);
        \draw[decorate,decoration={brace,amplitude=4pt}]
            ({cx(3.8)},\yD+0.15) -- node[above=4pt] {\scriptsize $2r_n$} ({cx(4.4)},\yD+0.15);
        \draw[-{Stealth[length=2mm]}, thick, gray]
            ({cx(6.0)+0.4},\yD-0.15) to[bend left=30] node[right] {$\psi$} ({cx(6.0)+0.45},{cy(6.0)+0.15});
        \foreach \k in {1,...,9}{
            \pgfmathsetmacro{\sa}{\smin+0.6*\k-\d}
            \pgfmathsetmacro{\sb}{\smin+0.6*\k+\d}
            \fill[gray!30]
                plot[domain=\sa:\sb, samples=6] ({cx(\x)+\h*nx(\x)},{cy(\x)+\h*ny(\x)})
                -- plot[domain=\sb:\sa, samples=6] ({cx(\x)-\h*nx(\x)},{cy(\x)-\h*ny(\x)})
                -- cycle;
        }
        \fill[blue!15]
            plot[domain=\smin:\smax, samples=120] ({cx(\x)+\l*nx(\x)},{cy(\x)+\l*ny(\x)})
            -- plot[domain=\smax:\smin, samples=120] ({cx(\x)-\l*nx(\x)},{cy(\x)-\l*ny(\x)})
            -- cycle;
        \draw[red!70!black, dashed, thick]
            plot[domain=\smin:\smax, samples=120] ({cx(\x)+\h*nx(\x)},{cy(\x)+\h*ny(\x)});
        \draw[red!70!black, dashed, thick]
            plot[domain=\smin:\smax, samples=120] ({cx(\x)-\h*nx(\x)},{cy(\x)-\h*ny(\x)});
        \foreach \k in {0,...,10}{
            \pgfmathsetmacro{\sb}{\smin+0.6*\k}
            \draw[green!50!black]
                ({cx(\sb)-\h*nx(\sb)},{cy(\sb)-\h*ny(\sb)}) -- ({cx(\sb)+\h*nx(\sb)},{cy(\sb)+\h*ny(\sb)});
        }
        \draw[blue!70!black, very thick] plot[domain=\smin:\smax, samples=120] ({cx(\x)},{cy(\x)});
        \foreach \k in {0,...,9}{
            \pgfmathsetmacro{\sc}{\smin+0.3+0.6*\k}
            \fill[blue!70!black] ({cx(\sc)},{cy(\sc)}) circle (1.6pt);
        }
        \node[above left] at ({cx(1.7)},{cy(1.7)}) {\scriptsize $\psi(y_m)$};
        \node at ({cx(1.7)-0.55*nx(1.7)+0.05},{cy(1.7)-0.42*ny(1.7)}) {\scriptsize $U^m_n$};
        \node[right] at ({cx(\smax)+0.05},{cy(\smax)}) {$\Sigma$};
        \node at ({cx(0.6)},{cy(0.6)+1.05}) {$\Omega_j$};
        \node at ({cx(0.9)},{cy(0.9)-1.2}) {$\Omega_i$};
        \draw[-{Stealth[length=2mm]}, thick]
            ({cx(4.7)},{cy(4.7)}) -- ({cx(4.7)+0.9*nx(4.7)},{cy(4.7)+0.9*ny(4.7)})
            node[above right=-2pt] {\scriptsize $\nu$};
        \draw[-{Stealth[length=1.5mm]}] ({cx(3.4)+0.2},{cy(3.4)-1.25}) node[below] {\scriptsize $\Sigma^n$}
            -- ({cx(3.4)-0.05*nx(3.4)},{cy(3.4)-0.05*ny(3.4)});
        \draw[-{Stealth[length=1.5mm]}] ({cx(2.45)+0.5},{cy(2.45)-1.35}) node[below] {\scriptsize $\Psi(D\times\{-\eta\})$}
            -- ({cx(2.45)-\h*nx(2.45)},{cy(2.45)-\h*ny(2.45)});
        \draw[-{Stealth[length=1.5mm]}] ({cx(5.65)-0.6},{cy(5.65)-1.15}) node[below] {\scriptsize overlap zones}
            -- ({cx(5.6)-0.3*nx(5.6)},{cy(5.6)-0.3*ny(5.6)});
    \end{tikzpicture}
    \caption{Partition of the tubular neighbourhood $\Psi(D\times(-\eta,\eta))$ of the interface $\Sigma=\psi(D)$. Top: the lattice $\{y_m\}_{m\in I_n}\subset D$ with spacing $2r_n$. Bottom: the sets $U^m_n$, separated by the normal segments (green); the grey bands are the images of the overlap zones of the partition of unity, and the shaded strip around $\Sigma$ is $\Sigma^n$, outside of which the recovery sequence coincides with the wells.}
    \label{fig:tub_neighborhood}
\end{figure}

%% file: fig_local.tex
\begin{figure}
    \centering
    \begin{tikzpicture}[scale=0.84,
        declare function={smooth(\s)=3*\s^2-2*\s^3;}]
        \def\r{3.6}
        \def\e{0.5}
        \def\l{0.75}
        \def\T{2.9}
        \def\Y{5.3}
        \pgfmathsetmacro{\gl}{(2*\l+\T)/2}
        \fill[green!10] (-\Y,{2*\l}) rectangle (\Y,\T);
        \fill[green!10] (-\Y,{-2*\l}) rectangle (\Y,-\T);
        \fill[orange!20] (-\Y,\l) rectangle (\Y,{2*\l});
        \fill[orange!20] (-\Y,-\l) rectangle (\Y,{-2*\l});
        \fill[blue!15] (-\Y,-\l) rectangle (\Y,\l);
        \foreach \c in {-\r,\r}{
            \fill[pattern=north east lines, pattern color=gray!80]
                ({\c-\e},{-2*\l}) rectangle ({\c+\e},{2*\l});
            \draw[gray, densely dotted] ({\c-\e},-\T) -- ({\c-\e},\T);
            \draw[gray, densely dotted] ({\c+\e},-\T) -- ({\c+\e},\T);
            \draw[green!50!black, thick] (\c,-\T) -- (\c,\T);
        }
        \foreach \tt in {-2,-1,1,2}{
            \draw[gray!70] (-\Y,{\tt*\l}) -- (\Y,{\tt*\l});
        }
        \draw[blue!70!black, very thick] (-\Y,0) -- (\Y,0);
        \node[blue!70!black, fill=blue!15, inner sep=1pt] at (\Y-0.35,0.25) {\scriptsize $\Sigma$};
        \node at (0,\gl) {\scriptsize $v_n(y,t)=b_j(\Psi(y,t))$, \quad $\Psi(y,t)\in\Omega_j$};
        \node at (0,-\gl) {\scriptsize $v_n(y,t)=a_i(\Psi(y,t))$, \quad $\Psi(y,t)\in\Omega_i$};
        \node at (0,{1.5*\l}) {\scriptsize $\widetilde\gamma_m=\gamma_m+\zeta^+_n\big(b_j\circ\Psi-b_j(\psi(y_m))\big)$};
        \node at (0,{-1.5*\l}) {\scriptsize $\widetilde\gamma_m=\gamma_m+\zeta^-_n\big(a_i\circ\Psi-a_i(\psi(y_m))\big)$};
        \node at (0,{0.45*\l}) {\scriptsize $v_n=\gamma_m(t)$, potential $W_j(\psi(y_m),\cdot)$};
        \node at (0,{-0.45*\l}) {\scriptsize $v_n=\gamma_m(t)$, potential $W_i(\psi(y_m),\cdot)$};
        \node[gray!50!black] at (-\Y+0.4,{\T-0.22}) {\tiny cell $m-1$};
        \node[gray!50!black] at (\Y-0.4,{\T-0.22}) {\tiny cell $m+1$};
        \def\ya{-\T-0.35}
        \draw[-{Stealth[length=2mm]}] (-\Y,\ya) -- (\Y+0.3,\ya) node[right] {$y$};
        \foreach \x/\lab in {-\r/{y_m-r_n},0/{y_m},\r/{y_m+r_n}}{
            \draw (\x,\ya-0.07) -- (\x,\ya+0.07);
            \node[below] at (\x,\ya-0.05) {\scriptsize $\lab$};
        }
        \draw[decorate,decoration={brace,amplitude=4pt,mirror}]
            ({-\r+\e},\ya-0.55) -- node[below=4pt] {\scriptsize $Q'(y_m,r_n-\eta_n)$: \ $\varphi^m_n\equiv1$, contributes to $I^n_1$} ({\r-\e},\ya-0.55);
        \foreach \c in {-\r,\r}{
            \draw[decorate,decoration={brace,amplitude=3pt}]
                ({\c-\e},{2*\l+0.05}) -- node[above=3pt, fill=green!10, inner sep=1pt] {\tiny $2\eta_n$} ({\c+\e},{2*\l+0.05});
            \node[fill=white, inner sep=1pt] at (\c,{-2*\l-0.25}) {\tiny $I^n_2$};
        }
        \def\xa{-\Y-0.35}
        \draw[-{Stealth[length=2mm]}] (\xa,-\T) -- (\xa,\T+0.3) node[above] {$t$};
        \foreach \tt/\lab in {-2/{-2\ell_n},-1/{-\ell_n},0/{0},1/{\ell_n},2/{2\ell_n}}{
            \draw (\xa-0.07,{\tt*\l}) -- (\xa+0.07,{\tt*\l});
            \node[left] at (\xa-0.05,{\tt*\l}) {\scriptsize $\lab$};
        }
        \def\xz{\Y+1.1}
        \def\wz{1.6}
        \draw[-{Stealth[length=2mm]}] (\xz,-\T) -- (\xz,\T+0.3) node[above] {$t$};
        \draw (\xz,-\T) -- (\xz+\wz+0.2,-\T);
        \draw[dashed, gray] (\xz+\wz,-\T) -- (\xz+\wz,\T);
        \node[below] at (\xz,-\T) {\scriptsize $0$};
        \node[below] at (\xz+\wz,-\T) {\scriptsize $1$};
        \foreach \tt in {-2,-1,1,2}{
            \draw[gray!70] (\xz-0.07,{\tt*\l}) -- (\xz+\wz+0.2,{\tt*\l});
        }
        \draw[green!50!black, very thick]
            (\xz,-\T) -- (\xz,\l)
            -- plot[domain=0:1, samples=30, variable=\s] ({\xz+\wz*smooth(\s)},{\l+\s*\l})
            -- (\xz+\wz,\T);
        \draw[orange!80!black, very thick, densely dashed]
            (\xz+\wz,-\T) -- (\xz+\wz,{-2*\l})
            -- plot[domain=0:1, samples=30, variable=\s] ({\xz+\wz*(1-smooth(\s))},{-2*\l+\s*\l})
            -- (\xz,\T);
        \node[green!50!black, right] at (\xz+\wz,{\gl}) {\scriptsize $\zeta^+_n$};
        \node[orange!80!black, right] at (\xz+\wz,{-\gl}) {\scriptsize $\zeta^-_n$};
    \end{tikzpicture}
    \caption{Structure of the recovery sequence in local coordinates around the cell $Q'(y_m,r_n)$.}
    \label{fig:limsup_local}
\end{figure}